\documentclass[11pt]{article}
\usepackage{tikz}
\usepackage{amsmath,amsfonts,amsthm,enumerate,amssymb,gastex}
\newtheorem{Lemma}{{Lemma}}[section]
\newtheorem{theo}{{Theorem}}[section]

\newtheorem{prop}{{Proposition}}[section]
\newtheorem{cor}{{Corollary}}[section]

\newtheorem{remark}{{\bf Remark}}[section]

\begin{document}

\title{On enumeration of a class of maps on Klein bottle}
\author{{Dipendu Maity}\\[2mm]
{\normalsize School of Mathematics}\\{\normalsize Harish-Chandra Research Institute}\\ {\normalsize Chhatnag Road, Jhunsi Allahabad-211019, India.}\\
{\small dipendumaity@hri.res.in}\\\\
{Ashish Kumar Upadhyay}\\[2mm]
{\normalsize Department of Mathematics}\\{\normalsize Indian Institute of Technology Patna }\\{\normalsize Patliputra Colony, Patna 800\,013,  India.}\\
{\small upadhyay@iitp.ac.in}}
\maketitle

\vspace{-5mm}


\begin{abstract} There are eleven types of semi-equivelar maps on the torus and the Klein bottle \cite{dm:semi16}. Three of these are eqivelar maps and the remaining are semi-equivelar maps. We classify all eleven types of semi-equivelar maps on the Klein bottle.
\end{abstract}

{\small

{\bf AMS classification\,: 52C20, 52B70, 51M20, 57M60}

{\bf Keywords\,:} Polyhedral maps on Klein bottle, Vertex-transitive maps, Semi-Equivelar Maps}

\bigskip


\section{Introduction and definitions}

The following definitions one can see in \cite{bondy:gt08, dm:semi16, mu:torus-hc13} and graph related terms in Bondy and Murthy \cite{bondy:gt08}. These are being reproduced here for the sake of completeness. A {\em map} $M$ is an embedding of a finite graph $G$ on a surface $S$ such that the closure of components of $S \setminus G$, called the faces of $M$, are closed $2$-cells, that is, each homeomorphic to $2$-disk. By a map we mean a polyhedral map on the torus. A {\em n-cycle} $C_{n}$ is a finite connected $2$-regular graph with $n$ vertices, and the {\em face sequence} of a vertex $v$ in a map $K$ is a finite sequence $(a^{p}, b^{q}$, \dots, $m^{r})$ of powers of positive integers $a, b,$ \dots, $m \geq 3$ and $p, q,$ \dots, $r \geq 1$ in {\em cyclic} order such that through the vertex $v , p$ number of $C_{a}$ ($C_{a}$ denote the $a$-cycle), $q$ number of $C_{b}$, \dots, $r$ number of $C_{m}$ are incident. In this case, we say that $K$ is {\em semi-equivelar} map of type $\{a^{p}, b^{q}$, \dots, $m^{r}\}$ \cite{dm:semi16, mu:torus-hc13}. A map of type $\{p^q\}$ is a semi-equivelar map of type $\{p^q\}$. 
We say two maps of fixed type on the Klein bottle are {\em isomorphic} if there exists a {\em homeomorphism} of the Klein bottle which sends vertices to vertices, edges to edges, faces to faces and preserves incidents. More precisely, 
if we consider two polyhedral complexes $K_{1}$ and $K_{2}$ then an isomorphism to be a map $f ~:~ K_{1}\rightarrow K_{2}$ such that $f|_{V(K_{1})} : V(K_{1}) \rightarrow V(K_{2})$ is a bijection and $f(\sigma)$ is a cell in $K_{2}$ if and only if $\sigma$ is cell in $K_{1}$. 

There are eleven types of semi-equivelar maps on the torus and the Klein bottle, \cite{dm:semi16}. These types are $\{3^{6}\}, \{4^{4}\}, \{6^{3}\}, \{3^{3}, 4^{2}\}, \{3^{2}, 4, 3, 4\}, \{3, 6, 3, 6\}, \{3^{4}, 6\}, \{4, 8^{2}\}, \{3, 12^{2}\},\{4,$ $6, 12\}, \{3, 4, 6, 4\}$. 
In 1973, Altshuler \cite{alt:enu73} has shown the construction and enumeration of maps of types $\{3^{6}\}$ and $\{6^{3}\}$ on the torus. He has studied the maps of type $\{3^{6}\}$ on the torus with $n$ vertices and has given a characterization of these maps by introducing some arithmetic conditions. In 1986, Kurth \cite{kurth:enu86} has given enumeration of all equivelar maps of types $\{3^{6}\}, \{4^{4}\}$ and $\{6^{3}\}$ on the torus using quotients of planar tessellations by squares or equilateral triangles, respectively, under the action of a group of isometries generated by two translations. In 1983, Negami \cite{neg:emb97} has given uniqueness and faithfulness of embedding of $6$-regular toroidal graphs and then classified the maps of type $\{3^{6}\}$ by classifying the $6$-regular graphs on the torus.
Similarly in 2008, Brehm and K\"{u}hnel \cite{brehm:equi08} have presented a classification of equivelar maps of types $\{3^{6}\}, \{4^{4}\}, \{6^{3}\}$ on the torus using their different isotropy groups. 
In \cite{mu:torus-hc13}, Maity and Upadhyay have given a construction to enumerate all semi-equivelar maps of types $\{3^{3}, 4^{2}\}, \{3^{2}, 4, 3, 4\}, \{3, 6, 3, 6\}, \{3^{4}, 6\}, \{4, 8^{2}\}, \{3,$ $12^{2}\}, \{4, 6, 12\}, \{3, 4, 6, 4\}$ on the torus for arbitrary number of vertices.
By the Euler characteristic equation, it is easy to see that there is no bound on number of vertices of the toroidal maps. Therefore, the above articles have presented computational steps or formulas to calculate the non-isomorphic maps on the torus. Thus, by \cite{alt:enu73, brehm:equi08, kurth:enu86, mu:torus-hc13, neg:emb97, tu:enu}, we know

\begin{prop}\label{prop:torus}  Let $n$ denote the number of vertices of maps of type $\{3^{6}\}, \{4^{4}\},\{6^{3}\}, \{3^{3},$ $ 4^{2}\}, \{3^{2}, 4, 3, 4\}, \{3, 6, 3, 6\}, \{3^{4}, 6\},$ $\{4,$ $8^{2}\}, \{3, 12^{2}\}, \{4, 6, 12\}$ or $\{3, 4, 6, 4\}$ on the torus. Then we can classify maps on $n$ vertices upto isomporphism.   
\end{prop}

This paper is devoted to a study of all semi-equivelar maps on the Klein bottle with $n$ vertices. We devise a way (computational steps) to enumerate and characterise all semi-equivelar maps of types $\{3^{6}\}, \{4^{4}\}, \{6^{3}\}, \{3^{3}, 4^{2}\}, \{3^{2}, 4, 3, 4\}, \{3, 6, 3, 6\}, \{3^{4}, 6\}, \{4, 8^{2}\}, \{3,$ $12^{2}\}, \{4, 6, 12\}, \{3, 4, 6, 4\}$ on the Klein bottle. In next theorem we give a formula on $n$ number of vertices which enumerate number of non-isomorphic maps on the Klein bottle.

\begin{theo}\label{theo:main} Let $X \in \{\{3^{6}\}, \{4^{4}\}, \{6^{3}\}, \{3^{3}, 4^{2}\}, \{3^{2}, 4, 3, 4\}, \{3, 6, 3, 6\}, \{3^{4}, 6\}, \{4, 8^{2}\},$ $\{3,$ $12^{2}\}, \{4, 6, 12\}, \{3, 4, 6, 4\}\}$. Let $i(n)_{X}$ denote the number of non-isomorphic semi-equivelar maps with $n$ vertices of type $X$ on the Klein bottle. Then, $(1)~ i(n)_{\{3^{6}\}}$ = $\sum_{i=1,2} i \times |\{(m,n)~|~m \geq 3, n \geq 3m, ~gcd(n,2m) = i \times m\}| + |\{(m,n)~|~m \geq 2, m \mid n, n \geq  5m, ~gcd(n, 2m) = m\}|, (2)~i(n)_{\{4^{4}\}} = \sum_{i=1,2} i \times |\{(m, n)~|~m \geq 3, n \geq 3m, gcd(n,2m) = i \times m\}|, (3)~i(n)_{\{6^{3}\}} = i(\frac{n}{2})_{\{3^{6}\}}, (4)~i(n)_{\{3^{3},4^{2}\}} = \sum_{i=1, 2} i \times |\{(m,n)~|~m \geq 4, 2 \mid m, n \geq 3m, gcd(n,2m) = m\}| + |\{(m, n)~|~m \geq 5, n = 2m\}|  + |\{(m, n)~|~m \geq 4, n \geq 4m, 2m \mid n\}|, (5)~i(n)_{\{3^{2}, 4, 3, 4\}} = |\{(m,n)~|~m \geq 3, 2 \nmid m, 2m \mid n, n \geq 12\}|, (6)~i(n)_{\{4,8^{2}\}} = |\{(l,m,n)~|~$ $m \geq 3, 2 \mid m, n \geq 8m, 4m \mid n, 0 \le l \le (\frac{n}{4m}-1)\}| + |\{(l, m, n)~|~ m \geq 3, 2 \nmid m, n \geq 8m, 4m \mid n, 0 \le l \le (\frac{n}{4m}-1)\}|, (7)~i(n)_{\{3,6,3,6\}} = |\{(l, m, n)~|~ m \geq 3, 3m \mid n, n \geq 9m, 0 \le l \le (\frac{n}{3m}-1)\}| +  |\{(m, n)~|m \geq 2, 6m+2 \mid n, n \geq 6(3m+1)\}|  +  |\{(m, n)~|~m \geq 1, 2(m+2) \mid n, n \geq 10(m+2)\}|$  $+$  $|\{(m, n)~|~m \geq 1, 4m+5 \mid n, n \geq 12(4m+5)\}|, (8)~i(n)_{\{3,12^2\}} = |\{(l, m, n) ~|~ m \geq 3, 3m \mid n, n \geq 9m, 0 \le l \le (\frac{n}{6m}-1)\}| + |\{(m, n)~|~m \geq 2, 4(3m+1) \mid n, n \geq 12(3m+1)\}| + |\{(m, n)~|~m \geq 1, 4(m+2) \mid  n, n \geq 20(m+2)\}| + |\{(m, n)~|~m \geq 1, 4m+5 \mid n, n \geq 24(4m+5)\}|, (9)~i(n)_{\{3^4,6\}} = 0, (10)~i(n)_{\{4, 6, 12\}} = \{(l, m, n)~|~ m \geq 4, 2 \mid m, 6m \mid n, n \geq 12m, 0 \le l \le (\frac{n}{6m}-1)\}| + |\{(m, n)~|~m \geq 2, 2 \mid m, 12m \mid n, n\geq 24m\}|, (11)~i(n)_{\{3,4, 6, 4\}} = |\{(l, m, n)~|~ m \geq 4, 2 \mid m, 3m \mid n, n \geq 6m, 0 \le l \le (\frac{n}{3m}-1)\}| + |\{(m, n)~|~m \geq 2, 6m \mid (n-3m), n \geq 9m\}|$  $+$  $|\{(m, n)~|~m \geq 2, 6m \mid n, n \geq 12m\}|$.
\end{theo}




We introduced some specific types of cylindrical polyhedral representations which are bounded by identical cycles with different orientations or by mobius strips that are consists of polygonal faces of semi-equivelar maps on the Klein bottle. By these polyhedral representations, we give a complete characterizations of all semi-equivelar maps on the Klein bottle. So, we have 

\begin{theo}\label{classification-semi-maps} Let $M$ be a map of type $X$ with $n$ vertices on the Klein bottle. Then, the map $M$ is isomorphic to a polyhedral representation $K(r, s,  0), K(r,  s,  1)$ or $K(r,  s)$ (these polyhedral representations are defined later in the subsequent sections) for some $r \mid n$ and $s = f(n,  r)$. 
\end{theo}

\begin{remark} {\rm In Theorem \ref{classification-semi-maps}, the $K(r, s,  0), K(r,  s,  1)$ are called planar polyhedral representations which are bounded by identical cycles with different orientations and $K(r,  s)$ is called polyhedral representations which is bounded by mobius strips. We have computed all eleven types of semi-equivelar maps with on few number of vertices in \cite{mu:enu1509} using the idea of the proof of Theorem \ref{classification-semi-maps}.
}\end{remark}


The proofs of all above theorems are presented in Section \ref{result:all}. We study above three polyhedral representations of all eleven types of semi-equivelar maps in the subsequent sections. We use the results of the following sub-sections of Section \ref{Semi-equivelar-maps} in Section \ref{result:all} and prove our main results. 

\section{Semi-equivelar Maps}\label{Semi-equivelar-maps}
\noindent\textbf{2.1 Maps of type $\{3^{6}\}$~:~}\label{361}
Let $M$ be a map of type $\{3^{6}\}$ on the Klein bottle. A path $P(\dots, u_{i-1}, u_{i}, u_{i+1}, \dots)$ in edge graph of $M$ is of type $A$ if $lk(u_{i}) = C(a, b, u_{i+1}, c, d, u_{i-1})$ implies $lk(u_{i-1}) = C(g, a, u_{i}, d, u, u_{i-2})$ and $lk(u_{i+1}) = C(b, e, u_{i+2}, f, c, u_{i})$ for all $i$. A walk $Walk(u_1, e_{1,2}, u_2, \dots, u_i, e_{x-1, x}, u_x)$ in edge graph of $M$ is a sequence of vertices and edges such that $e_{t-1, t} = u_{t-1}u_t, 1 \le t \le x$. We denote it by $Walk(u_1, u_2 \dots, u_x)$. 

\smallskip

$Claim :$ If $W(w_1, \dots, w_r)$ is a maximal walk of type $A$ in $EG(M)$ (edge graph of $M$) then $W(w_1, \dots, w_r)$ is a closed walk of type $A$.

\smallskip

By the defintion of $A$, let $lk(w_{r-1}) = C(a, b, w_{r}, c, d, w_{r-1})$, it implies that $lk(w_{r}) = C(b, e, w_{r+1}, f, c, w_{r})$. If $w_{r} = w_{1}$ then $C(w_{1}, w_{2},$ \dots, $w_{r})$ is a closed walk of type $A$. If $w_{r} \neq w_{1}$, it implies that $w_{r} = w_i$ for some $2 \le i \le r$. Hence, we get a closed walk $R = Walk(w_i, w_{i+1}, \dots, w_r)$ and which is not of type $A$ at $w_i$. This implies that the number of triangles which are incident at $w_i$ is one or two on one side of $R$. If the number of incident triangles is one then we define an another closed walk as $R$ of lesser length using the incident faces of the walk $W$. (The similar description of this construction is given details for a cycle in \emph{Theorem 1} \cite{alt:hc72} (also see it in the example in \emph{Fig. 2} \cite{alt:hc72}).) Similarly, if the number triangles which are incident at $w_i$ is two then one can see that we get a cycle of same type (as in \emph{Fig. 3} \cite{alt:hc72}). Thus, we get a closed walk as $R$ of lesser length using the triangles which are incident with $R$. By induction on length, it is not possible to construct walk as $R$ at each step as length is decreasing. Therefore, $w_r = w_1$ and $W(w_1, \dots, w_r)$ is a closed walk of type $A$. This proves the claim.

\smallskip

$Claim :$ A closed walk $W(w_1, \dots, w_r)$ of type $A$ is non-contractible.

\smallskip

If the walk $W$ bounds a $2$-disk then we consider triangles which are incident with $W$ in the $2$-disk. There are three triangles which are incident at each vertex of $W$ in the $2$-disk. Let $w_{i-1}w_iv_i, w_{i-1}v_iv_{i+1}, w_{i}v_{i+1}w_i$ denote three triangles incident at $w_i$ for all $i$. These incident triangles define a walk $W' = Walk(v_1,$ \dots, $v_r)$ of type $A$ and has a length $r$ as $W$. Similarly, we consider walk $W'$ and repeat as above. Hence, we get an infinitely many sequence of walks namely $W, W', W'', \dots$ where each walk has length $r$. But, the $2$-disk consists of finite number of triangles. Therefore, the walk $W$ does not bound any $2$-disk. Which is a contradiction. So, $W$ is non-contractible. This proves the claim.
We show in the next lemma that the map $M$ contains a cycle of type $A$.  

\begin{Lemma}\label{lem36:1} The map $M$ contains a cycle of type $A$.
\end{Lemma}

\begin{proof} Let $W$ be a closed walk of type $A$ in edge graph of $M$ and $S_W$ denote a set of faces which are incident with walk $W$. Observe that by considering incident faces of $W$, the boundary of the geometric carrier $S_W$ consists of two closed walks. It follows that $\partial S_W = \{W_2, W_s\}$ where $W_2$ and $W_s$ are two closed walks of type $A$ and have same length. Here we say that the walks $W_2, W_s$ are  homologous to $W$. Since $W$ is a walk, let $w \in V(W)$ repeat twice in $W$. Observe that by the definition of $A$, there is no repetition of edges in $W$. Since $w \in V(W)$ repeats twice in $W$, it follows that there are two sub-paths $P_1 = P(w_{l-1}, w, w_{l+1})$ and $P_2 = P(w_{k-1}, w, w_{k+1})$ of $W$ where $w_{l-1}, w_{l+1}, w_{k-1}, w_{k+1} \in  lk(w)$. Let $v \in lk(w)$. Let $W_1 (= W), W_2, \dots, W_s$ be a sequence of walks which are homologous to $W$ in $M$. 

\smallskip

$Claim : $ The vertex $v$ repeats twice in $W_1 \cup W_2 \cup \dots \cup W_s$

\smallskip

If $P_1 \subset W$ then $v \in  W_i$ for some $i$ and $W_i \in \partial S_W$. Similarly, if $P_2 \subset W$ then $v \in  W_j$ for some $j$ and $W_j \in \partial S_W$. So, the vertex $v$ repeats twice in $W_i \cup W_j$. Therefore, every vertex of $M$ repeats twice in $W_1 \cup W_2 \cup \dots \cup W_s$ since $u \in W_1 \cup W_2 \cup \dots \cup W_s~ \forall~ u \in V(M)$. This proves the Claim. 

Similarly as in the Claim, if a vertex repeats thrice in $W$ then every vertex of $M$ repeats thrice in $W_1 \cup W_2 \cup \dots \cup W_s$. Since $W$ is non-contractible, it implies that the $W$ represents a generator of the fundamental group of $M$. The fundamental group of $M$ contains exactly two generator up to homologous. Let $W', W''$ be two non-homologous closed walks of type $A$ through a vertex in $M$. Here, if the walk $W'$ contains a vertex, say $w$ which repeats twice or thrice in $W'$ then we have following cases. If the vertex $w$ repeats thrice then by the above Claim, every vertex of $M$ repeats thrice in the union of homologous walks of $W'$. This implies that the union of walks contains all the edges of $M$. Which is not possible since $W', W''$ are non-homologus and which implies that $E(W') \cap E(W'') = \emptyset$. So, walk $W'$ contains the vertex $w$ twice. Let $L_1, \dots, L_l$ denote sequence of homologous walks of $W'$ in $M$. Since $W'$ contains the vertex $w$ twice, it implies by Claim that $L_1 \cup \dots \cup L_l$ contains each vertex twice of $M$. Again, $E(W') \cap E(W'') = \emptyset$ implies that $E(W'') \cap E(L_i) = \emptyset$ for all $1 \le i \le l$. Hence, walk $W''$ contains all distinct vertices since degree of each vertex of $M$ is six. So, $W''$ is a cycle. Therefore, map $M$ contains a cycle of type $A$.
\end{proof}

For an example, the $W(v_1, w_1, x_1, u_1, v_2, w_3, x_4, u_5, u_4, w_4, x_4, u_4, v_6, w_7, x_1, u_2, v_7,$ $ w_7,$ $ x_7,$ $u_7,$ $ v_3, w_4, x_5, u_6, v_3, w_3, x_3, u_3, v_7, w_1, x_2, u_3, v_6, w_6, x_6, u_6, v_4, w_5, x_6, u_7, v_2, w_2, x_2, u_2, v_1, w_2, \linebreak x_3, u_4, v_5, w_5, x_5, u_5, v_5, w_6, x_7, u_1, v_1)$ is a closed walk of type $A$ and $C(x_1, \dots, x_7)$ is a cycle of type $A$ in Figure 1. In $W$ each vertex repeats exactly twice. 

Let $C$ be a cycle of type $A$ in $M$ and $u \in V(C)$. A set $S \subset F(M)$ is of faces such that $V(\triangle) \cap V(C) \neq \emptyset$ for all $\triangle \in S$. Then, we have

\begin{Lemma}\label{lem36:2} The geometric carrier $|S|$ is a cylinder or an union of `a M\"obius strip and a cylinder'.
\end{Lemma}

\begin{proof} Let $SQ$ be a sequence of faces which are incident with the cycle $C$. Let $u_1 \in V(C)$ and $lk(u_1) = C(v_r, v_1, u_{2}, w_2, w_1, u_{r})$. We cut $|SQ|$ along a path $P(v_1, u_1, w_1)$ of type $A$. Then we have the following cases. 

\textbf{Case 1 :} We assume that each face in $SQ$ appears exactly once. Here, we have the following sub-cases. If the path $P(v_1, u_1, w_1)$ identifies with $P(v_1, u_1, w_1)$ with out any twist then $|SQ|$ is cylinder (see an example in Figure 2). In this acse, we denote it by $S(C, M)$. If $P(v_1, u_1, w_1)$ identifies with $P(v_1, u_1, w_1)$ with a twist then the cycle $C$ is not of type $A$ at the vertex $u_1$. (For an example in Figure 3, $C(w_1, \dots, w_r)$ is a cycle and not of type $A$ at $w_1$.) Therefore, $|S| = |SQ| = S(C, M)$ is a cylinder if the faces are distinct in $SQ$. 

\textbf{Case 2 :} We assume that not all faces are distinct in $SQ$. Let $SQ'$ denote an another sequence of faces which are incident with $C$ and lie on one side of $C$. Here, the length of $C$ is either odd or even.                

Let $C = C(w_1, w_2, \dots, w_{2r+1})$ and $SQ' = \{\triangle_1, \triangle_2, \dots, \triangle_{4r+2}\}$ denote the sequence of faces in order. If $\triangle_1, \triangle_2, \dots, \triangle_{4r+2}$ are all distinct, i.e., $\triangle_i \neq \triangle_j$ for all $i \neq j$ then $|SQ'|$ forms a cylinder with disjoint boundary cycles. If $\triangle_i = \triangle_j$ in $SQ'$ for some $i < j$ then $SQ'' :=\{\triangle_1, \triangle_2, \dots, \triangle_{t}\} \subset SQ'$ denotes a subsequence where $\{\triangle_1, \triangle_2, \dots, \triangle_{t}\} = \{\triangle_1, \triangle_2, \dots, \triangle_{4r+2}\}$ as a set, and $\triangle_1$ and $\triangle_{t}$ appear once in $\{\triangle_1, \triangle_2, \dots, \triangle_t\}$.       

\smallskip

$Claim :$ The triangles $\triangle_1, \triangle_2, \dots, \triangle_{t}$ appear once in $SQ''$.

\smallskip

Let $\{\triangle_l, \dots, \triangle_s\}$ denote a largest sub-sequence in $SQ''$ where each $\triangle_l, \dots, \triangle_s$ appears more than once in $SQ''$.
This implies that the faces $\triangle_{l-1}$ and $\triangle_{s+1}$ are distinct and appear once in $SQ''$. Let $\triangle_s = abc$ and $\triangle_s \cap C = ab$. Then, $\triangle_s \cap \triangle_{s+1} = bc$. Since $\triangle_s$ repeats in $SQ''$, it implies that there are $\triangle_{m-1}$ and $\triangle_{m+1}$ in $SQ''$ such that $\triangle_m = \triangle_s$ ($m<s$). That is, $\triangle_s$ lies between $\triangle_{m-1}$ and $\triangle_{m+1}$ in $SQ''$. Hence, $\triangle_s \cap \triangle_{m-1}$ and $\triangle_s \cap \triangle_{m+1}$ are each an edge. Since $\triangle_{s+1}$ appears first time in $SQ''$ after $\triangle_s$, it follows that either $\triangle_{s-1} \neq \triangle_{m-1}$ or $\triangle_{s-1} \neq \triangle_{m+1}$. If $\triangle_{s-1} \neq \triangle_{m+1}$ then $\triangle_{s+1} \neq \triangle_{s-1} \neq \triangle_{m+1}$. Now, by definition of type $A$, one edge, say $e$ of $\triangle_s$ must belongs to $C$ or $C'$. Hence, the triangle $\triangle_s$ must have at least four distinct edges $\triangle_{s+1} \cap \triangle_s, \triangle_{s-1}\cap \triangle_s, \triangle_{m+1}\cap \triangle_s, e$. This is a contradiction as triangle consists three edges. Similarly, let $\triangle_s \cap C = a$. Then by the above similar argument, $\triangle_s$ does not have sufficient number of free edges for identification later in $SQ''$. Thus, the triangles $\triangle_1, \triangle_2, \dots, \triangle_{t}$ appear once in $SQ''$. This proves the Claim.

Let $\triangle_i = ABC \in SQ'$. Now $\triangle_i \cap \triangle_{i-1}$ and $\triangle_i \cap \triangle_{i+1}$ are each an edge. Let $\triangle_i \cap \triangle_{i-1} = AB$ and $\triangle_i \cap \triangle_{i+1} = AC$. Hence, $\triangle_i \cap C$ is either an edge $BC$ or a vertex. Let $\triangle_i \cap C = BC$. Then, $\triangle_i \cap C'$ is a vertex. Hence, the common edge $BC$ and the common vertex can not repeat twice in both $C, C'$ since $C, C'$ are cycles. But $\triangle_i$ repeats twice in $SQ'$ only if $C = C'$. Since $\triangle_i$ is repeating, it implies that it is repeating twice. Hence, each triangle appears twice in the sequence $SQ'$. Thus, the number of distinct triangle in $SQ'$ is $2r+1$. That is, $t = 2r+1$. So, the sequence $\triangle_1, \triangle_2, \dots, \triangle_{2r+1}$ appears twice in the sequence $\triangle_1, \triangle_2, \dots, \triangle_{4r+2}$, i.e., $\{\triangle_1, \triangle_2, \dots, \triangle_{4r+2}\} = \{\triangle_1, \triangle_2, \dots, \triangle_{2r+1}, \triangle_1, \triangle_2, \dots, \triangle_{2r+1}\}$. In this case, we call that $|SQ'|$ is a M\"obius strip. (See an example in Figure 4.)     

When length$(C) = 2r$ then the number of triangles incident with $SQ'$ is even. In this case, we consider link of $w_{r+1}$ which is end vertex of the path $P(w_1, \dots, w_{r+1}) \subset C(w_, w_2, \dots, w_{2r})$. Observe that the number of triangles which are incident with $C$ at $w_{r+1}$ is two in $SQ'$. This contradicts with the type $A$.   

Let $S_1$ denote a set of incident triangles which are lie on one side and $S_2$ denote a set of triangles which are incident with $C$ and lie on the other side of $C$. If $S_1 \cap S_2 \ne \emptyset$ then, let $\triangle \in S_1 \cap S_2$. Let $S_i = \{\triangle_{i, 1}, \triangle_{i, 2}, \dots, \triangle_{i, m}\}$ and $\triangle = \triangle_{1, j} = \triangle_{2, t}$ for some $t, j$. Then, by the above similar argument, $\{\triangle_{1, j-1}, \triangle_{1, j+1}\} = \{\triangle_{2, t-1}, \triangle_{2, t+1}\}$. Again, we consider $\triangle_{1, j-1}$ and repeat as above and continue. Hence, we get $S_1 = S_2$ after finite number of steps and the geometric carrier $|S_1|$ is bounded by identical cycle $C$. In this case, it defines a $K(r, 1, k)$ representation (this polyhedral representation is defined later in this section) of $M$. (For an example in Figure 1, $K(7, 4, 1)$ is a representation of $M$ where $s = 4$.)  But, by Lemma \ref{lem36:6}, $K(r, 1, k)$ does not exist. Therefore, $S_1 \cap S_2 = \emptyset$ and $S = S_1 \cup S_2$ where one of $|S_1|, |S_2|$ is a cylinder and another is a M\"obius strip. Observe that both $|S_1|, |S_2|$ can not be M\"obius strips as $S_1 \cap S_2 = \emptyset$. We denote M\"obius strip by $\mathcal{M}_\triangle(C, M)$ which is bounded by $C$ and $|S|$ by $S\mathcal{M}_\triangle(C, M)$. (For examples, $\mathcal{M}_\triangle(C, M)$ and $S\mathcal{M}_\triangle(C, M)$ are in Figure 5.)
\end{proof}

\begin{remark}{\rm Note that we use these above notions of $S(C, M), \mathcal{M}_{F_1,\dots,F_t}(C, M)$ and $S\mathcal{M}_{F_1,\dots,F_t}(C, M)$ in the subsequent sections. Here, the $S(C, M)$ used to denote a cylinder which is geometric career of the faces that are incident with $C$. The $\mathcal{M}_{F_1,\dots,F_t}(C, M)$ used to denote a M\"obius strip which is bounded by $C$ and consists of faces of types $F_1,\dots,F_t$. The $S\mathcal{M}_{F_1,\dots,F_t}(C, M) := |S \cup \mathcal{M}|$ used to denote union of a cylinder $|S|$ and a M\"obius strip $|\mathcal{M}|$ such that $|S| \cap |\mathcal{M}| = C$ and consists of faces of types $F_1,\dots,F_t$. In this section, $t=1$ and $F_1 = \triangle$.}
\end{remark}

\vspace{.1cm}

\begin{picture}(0,0)(5,35)
\tiny
\setlength{\unitlength}{1.6mm}
\drawpolygon(5,5)(40,5)(40,25)(5,25)


\drawline[AHnb=0](10,5)(10,25)
\drawline[AHnb=0](15,5)(15,25)
\drawline[AHnb=0](20,5)(20,25)
\drawline[AHnb=0](25,5)(25,25)
\drawline[AHnb=0](30,5)(30,25)
\drawline[AHnb=0](35,5)(35,25)

\drawline[AHnb=0](5,10)(40,10)
\drawline[AHnb=0](5,15)(40,15)
\drawline[AHnb=0](5,20)(40,20)

\drawline[AHnb=0](5,20)(10,25)
\drawline[AHnb=0](5,15)(15,25)
\drawline[AHnb=0](5,10)(20,25)
\drawline[AHnb=0](5,5)(25,25)
\drawline[AHnb=0](10,5)(30,25)
\drawline[AHnb=0](15,5)(35,25)
\drawline[AHnb=0](20,5)(40,25)
\drawline[AHnb=0](25,5)(40,20)
\drawline[AHnb=0](30,5)(40,15)
\drawline[AHnb=0](35,5)(40,10)

\put(5,4){\tiny ${v_{1}}$}
\put(10,4){\tiny ${v_{2}}$}
\put(15,4){\tiny ${v_{3}}$}
\put(20,4){\tiny ${v_{4}}$}
\put(25,4){\tiny ${v_{5}}$}
\put(30,4){\tiny ${v_{6}}$}
\put(35,4){\tiny ${v_{7}}$}
\put(40,4){\tiny ${v_{1}}$}

\put(5.5,9){\tiny ${w_{1}}$}
\put(10.5,9){\tiny ${w_{2}}$}
\put(15.5,9){\tiny ${w_{3}}$}
\put(20.5,9){\tiny ${w_{4}}$}
\put(25.5,9){\tiny ${w_{5}}$}
\put(30.5,9){\tiny ${w_{6}}$}
\put(35.5,9){\tiny ${w_{7}}$}
\put(40.5,9){\tiny ${w_{1}}$}

\put(5.5,14){\tiny ${x_{1}}$}
\put(10.5,14){\tiny ${x_{2}}$}
\put(15.5,14){\tiny ${x_{3}}$}
\put(20.5,14){\tiny ${x_{4}}$}
\put(25.5,14){\tiny ${x_{5}}$}
\put(30.5,14){\tiny ${x_{6}}$}
\put(35.5,14){\tiny ${x_{7}}$}
\put(40.5,14){\tiny ${x_{1}}$}

\put(5.5,19){\tiny ${u_{1}}$}
\put(10.5,19){\tiny ${u_{2}}$}
\put(15.5,19){\tiny ${u_{3}}$}
\put(20.5,19){\tiny ${u_{4}}$}
\put(25.5,19){\tiny ${u_{5}}$}
\put(30.5,19){\tiny ${u_{6}}$}
\put(35.5,19){\tiny ${u_{7}}$}
\put(40.5,19){\tiny ${u_{1}}$}

\put(5.5,24){\tiny ${v_{2}}$}
\put(10.5,24){\tiny ${v_{1}}$}
\put(15.5,24){\tiny ${v_{7}}$}
\put(20.5,24){\tiny ${v_{6}}$}
\put(25.5,24){\tiny ${v_{5}}$}
\put(30.5,24){\tiny ${v_{4}}$}
\put(35.5,24){\tiny ${v_{3}}$}
\put(40.5,24){\tiny ${v_{2}}$}

\put(19,1){\scriptsize Figure 1 : $M$}

\end{picture}

\begin{picture}(0,0)(-65,30)
\tiny
\setlength{\unitlength}{1.6mm}
\drawpolygon(5,10)(40,10)(40,20)(5,20)
\drawpolygon(5,5)(20,5)(25,10)(5,10)
\drawpolygon(5,20)(25,20)(25,25)(10,25)


\drawline[AHnb=0](10,5)(10,25)
\drawline[AHnb=0](15,5)(15,25)
\drawline[AHnb=0](20,5)(20,25)
\drawline[AHnb=0](25,10)(25,20)
\drawline[AHnb=0](30,10)(30,20)
\drawline[AHnb=0](35,10)(35,20)

\drawline[AHnb=0](5,10)(40,10)
\drawline[AHnb=0](5,15)(40,15)
\drawline[AHnb=0](5,20)(40,20)

\drawline[AHnb=0](5,20)(10,25)
\drawline[AHnb=0](5,15)(15,25)
\drawline[AHnb=0](5,10)(20,25)
\drawline[AHnb=0](5,5)(25,25)
\drawline[AHnb=0](10,5)(25,20)
\drawline[AHnb=0](15,5)(30,20)
\drawline[AHnb=0](20,5)(35,20)
\drawline[AHnb=0](30,10)(40,20)
\drawline[AHnb=0](35,10)(40,15)
\drawline[AHnb=0](40,10)(40,10)

\put(5,4){\tiny ${w_{5}}$}
\put(10,4){\tiny ${w_{6}}$}
\put(15,4){\tiny ${w_{7}}$}
\put(20,4){\tiny ${w_{1}}$}

\put(5.5,9){\tiny ${w_{1}}$}
\put(10.5,9){\tiny ${w_{2}}$}
\put(15.5,9){\tiny ${w_{3}}$}
\put(20.5,9){\tiny ${w_{4}}$}
\put(25.5,9){\tiny ${w_{5}}$}
\put(30.5,9){\tiny ${w_{6}}$}
\put(35.5,9){\tiny ${w_{7}}$}
\put(40.5,9){\tiny ${w_{1}}$}

\put(5.5,14){\tiny ${x_{1}}$}
\put(10.5,14){\tiny ${x_{2}}$}
\put(15.5,14){\tiny ${x_{3}}$}
\put(20.5,14){\tiny ${x_{4}}$}
\put(25.5,14){\tiny ${x_{5}}$}
\put(30.5,14){\tiny ${x_{6}}$}
\put(35.5,14){\tiny ${x_{7}}$}
\put(40.5,14){\tiny ${x_{1}}$}

\put(5.5,19){\tiny ${u_{1}}$}
\put(10.5,19){\tiny ${u_{2}}$}
\put(15.5,19){\tiny ${u_{3}}$}
\put(20.5,19){\tiny ${u_{4}}$}
\put(25.5,19){\tiny ${u_{5}}$}
\put(30.5,19){\tiny ${u_{6}}$}
\put(35.5,19){\tiny ${u_{7}}$}
\put(40.5,19){\tiny ${u_{1}}$}

\put(10.5,24){\tiny ${u_{5}}$}
\put(15.5,24){\tiny ${u_{6}}$}
\put(20.5,24){\tiny ${u_{7}}$}
\put(25.5,24){\tiny ${u_{1}}$}

\put(13,1){\scriptsize Figure 6 : $K(S'(M), \mathcal{M}', \mathcal{M}'')$}

\end{picture}

\vspace{5cm}

\begin{picture}(0,0)(-80,2)
\tiny
\setlength{\unitlength}{1.6mm}
\drawpolygon(5,5)(20,5)(20,15)(5,15)
\drawpolygon(22,5)(37,5)(37,15)(22,15)


\drawline[AHnb=0](10,5)(10,15)
\drawline[AHnb=0](15,5)(15,15)
\drawline[AHnb=0](27,5)(27,15)
\drawline[AHnb=0](32,5)(32,15)
\drawline[AHnb=0](5,10)(20,10)
\drawline[AHnb=0](22,10)(37,10)

\drawline[AHnb=0](5,10)(10,15)
\drawline[AHnb=0](5,5)(15,15)
\drawline[AHnb=0](10,5)(20,15)
\drawline[AHnb=0](15,5)(20,10)
\drawline[AHnb=0](22,10)(27,15)
\drawline[AHnb=0](22,5)(32,15)
\drawline[AHnb=0](27,5)(37,15)
\drawline[AHnb=0](32,5)(37,10)

\put(20,10){\dots}
\put(20,5){\dots}
\put(20,15){\dots}

\put(5,4){\tiny ${v_{1}}$}
\put(10,4){\tiny ${v_{2}}$}
\put(15,4){\tiny ${v_{3}}$}
\put(20,4){\tiny ${v_{4}}$}
\put(22,4){\tiny ${v_{r-2}}$}
\put(27,4){\tiny ${v_{r-1}}$}
\put(32,4){\tiny ${v_{r}}$}
\put(37.2,4){\tiny ${v_{1}}$}

\put(5.5,9){\tiny ${w_{1}}$}
\put(10.5,9){\tiny ${w_{2}}$}
\put(15.5,9){\tiny ${w_{3}}$}
\put(20,9){\tiny ${w_{4}}$}
\put(22,9){\tiny ${w_{r-2}}$}
\put(27,9){\tiny ${w_{r-1}}$}
\put(32,9){\tiny ${w_{r}}$}
\put(37.2,9){\tiny ${w_{1}}$}

\put(5.5,14){\tiny ${x_{1}}$}
\put(10.5,14){\tiny ${x_{2}}$}
\put(15.5,14){\tiny ${x_{3}}$}
\put(20,14){\tiny ${x_{4}}$}
\put(22,14){\tiny ${x_{r-2}}$}
\put(27,14){\tiny ${x_{r-1}}$}
\put(32,14){\tiny ${x_{r}}$}
\put(37.2,14){\tiny ${x_{1}}$}

\put(12,1){\scriptsize Figure 2 : $S(C, M)$ : Cylinder}

\end{picture}

\begin{picture}(0,0)(15,1)
\tiny
\setlength{\unitlength}{1.6mm}
\drawpolygon(11,5)(26,5)(26,15)(11,15)
\drawpolygon(38,5)(43,5)(43,10)(38,10)
\drawpolygon(45,5)(60,5)(60,10)(45,10)
\drawpolygon(38,10)(43,15)(38,15)


\drawpolygon(28,10)(43,10)(43,15)(28,15)
\drawpolygon(28,5)(38,5)(38,10)(28,10)
\drawpolygon(45,10)(60,10)(60,15)(45,15)

\drawline[AHnb=0](16,5)(16,15)
\drawline[AHnb=0](21,5)(21,15)
\drawline[AHnb=0](33,5)(33,15)
\drawline[AHnb=0](38,5)(38,15)
\drawline[AHnb=0](11,10)(26,10)
\drawline[AHnb=0](28,10)(43,10)
\drawline[AHnb=0](55,5)(55,15)
\drawline[AHnb=0](50,5)(50,15)
\drawline[AHnb=0](45,10)(60,10)

\put(26,10){\dots}
\put(26,5){\dots}
\put(26,15){\dots}

\drawline[AHnb=0](11,10)(16,15)
\drawline[AHnb=0](11,5)(21,15)
\drawline[AHnb=0](16,5)(26,15)
\drawline[AHnb=0](21,5)(26,10)
\drawline[AHnb=0](28,10)(33,15)
\drawline[AHnb=0](28,5)(38,15)
\drawline[AHnb=0](33,5)(43,15)
\drawline[AHnb=0](45,5)(50,10)
\drawline[AHnb=0](50,10)(55,15)
\drawline[AHnb=0](50,5)(60,15)
\drawline[AHnb=0](55,5)(60,10)
\drawline[AHnb=0](38,5)(43,10)
\drawline[AHnb=0](45,10)(50,15)

\put(43,10){\ldots}
\put(43,5){\ldots}
\put(43,15){\ldots}

\put(11,4){\tiny ${v_{1}}$}
\put(16,4){\tiny ${v_{2}}$}
\put(21,4){\tiny ${v_{3}}$}
\put(26,4){\tiny ${v_{4}}$}
\put(28,4){\tiny ${v_{r-1}}$}
\put(33,4){\tiny ${v_{r}}$}
\put(38,4){\tiny ${v_{r+1}}$}
\put(41,4){\tiny ${v_{r+2}}$}
\put(45,4){\tiny ${v_{2r-1}}$}
\put(50,4){\tiny ${v_{2r}}$}
\put(55,4){\tiny ${v_{2r+1}}$}
\put(60,4){\tiny ${v_{1}}$}

\put(11.5,9){\tiny ${w_{1}}$}
\put(16.5,9){\tiny ${w_{2}}$}
\put(21.5,9){\tiny ${w_{3}}$}
\put(26,9){\tiny ${w_{4}}$}
\put(28,9){\tiny ${w_{r-1}}$}
\put(33.5,9){\tiny ${w_{r}}$}
\put(38,9){\tiny ${w_{r+1}}$}
\put(39.5,10.5){\tiny ${w_{r+2}}$}
\put(45,9){\tiny ${w_{2r-1}}$}
\put(50.5,9){\tiny ${w_{2r}}$}
\put(55,9){\tiny ${w_{2r+1}}$}
\put(58,10.5){\tiny ${w_{1}}$}

\put(10,16){\tiny ${w_{r+1}}$}
\put(15,16){\tiny ${w_{r+2}}$}
\put(20,16){\tiny ${w_{r+3}}$}
\put(24,16){\tiny ${w_{r+4}}$}
\put(27.8,16){\tiny ${w_{2r-1}}$}
\put(32,16){\tiny ${w_{2r}}$}
\put(37,16){\tiny ${w_{2r+1}}$}
\put(42,16){\tiny ${w_{1}}$}

\put(44,16){\tiny ${w_{r-2}}$}
\put(49,16){\tiny ${w_{r-1}}$}
\put(54,16){\tiny ${w_{r}}$}
\put(57,16){\tiny ${w_{r+1}}$}

\put(33,1){\scriptsize Figure 4}

\end{picture}

\begin{picture}(0,0)(-80,20)
\tiny
\setlength{\unitlength}{1.6mm}
\drawpolygon(5,5)(20,5)(20,15)(5,15)
\drawpolygon(22,5)(37,5)(37,15)(22,15)


\drawline[AHnb=0](10,5)(10,15)
\drawline[AHnb=0](15,5)(15,15)
\drawline[AHnb=0](27,5)(27,15)
\drawline[AHnb=0](32,5)(32,15)
\drawline[AHnb=0](5,10)(20,10)
\drawline[AHnb=0](22,10)(37,10)

\drawline[AHnb=0](5,10)(10,15)
\drawline[AHnb=0](5,5)(15,15)
\drawline[AHnb=0](10,5)(20,15)
\drawline[AHnb=0](15,5)(20,10)
\drawline[AHnb=0](22,10)(27,15)
\drawline[AHnb=0](22,5)(32,15)
\drawline[AHnb=0](27,5)(37,15)
\drawline[AHnb=0](32,5)(37,10)

\put(20,10){\dots}
\put(20,5){\dots}
\put(20,15){\dots}

\put(5,4){\tiny ${v_{1}}$}
\put(10,4){\tiny ${v_{2}}$}
\put(15,4){\tiny ${v_{3}}$}
\put(20,4){\tiny ${v_{4}}$}
\put(22,4){\tiny ${v_{r-2}}$}
\put(27,4){\tiny ${v_{r-1}}$}
\put(32,4){\tiny ${v_{r}}$}
\put(37.2,4){\tiny ${x_{1}}$}

\put(5.5,9){\tiny ${w_{1}}$}
\put(10.5,9){\tiny ${w_{2}}$}
\put(15.5,9){\tiny ${w_{3}}$}
\put(20,9){\tiny ${w_{4}}$}
\put(22,9){\tiny ${w_{r-2}}$}
\put(27,9){\tiny ${w_{r-1}}$}
\put(32,9){\tiny ${w_{r}}$}
\put(37.2,9){\tiny ${w_{1}}$}

\put(5.5,14){\tiny ${x_{1}}$}
\put(10.5,14){\tiny ${x_{2}}$}
\put(15.5,14){\tiny ${x_{3}}$}
\put(20,14){\tiny ${x_{4}}$}
\put(22,14){\tiny ${x_{r-2}}$}
\put(27,14){\tiny ${x_{r-1}}$}
\put(32,14){\tiny ${x_{r}}$}
\put(37.2,14){\tiny ${v_{1}}$}

\put(19,0){\scriptsize Figure 3 }

\end{picture}

\begin{picture}(0,0)(15,20)
\tiny
\setlength{\unitlength}{1.6mm}
\drawpolygon(11,5)(26,5)(26,15)(11,15)
\drawpolygon(38,5)(43,5)(43,10)(38,10)
\drawpolygon(45,5)(60,5)(60,10)(45,10)


\drawpolygon(11,5)(26,5)(26,15)(11,15)
\drawpolygon(28,10)(38,10)(43,15)(28,15)
\drawpolygon(28,5)(38,5)(38,10)(28,10)

\drawline[AHnb=0](16,5)(16,15)
\drawline[AHnb=0](21,5)(21,15)
\drawline[AHnb=0](33,5)(33,15)
\drawline[AHnb=0](38,5)(38,15)
\drawline[AHnb=0](11,10)(26,10)
\drawline[AHnb=0](28,10)(43,10)
\drawline[AHnb=0](55,5)(55,10)
\drawline[AHnb=0](50,5)(50,10)
\drawline[AHnb=0](45,10)(60,10)

\put(26,10){\dots}
\put(26,5){\dots}
\put(26,15){\dots}

\drawline[AHnb=0](11,10)(16,15)
\drawline[AHnb=0](11,5)(21,15)
\drawline[AHnb=0](16,5)(26,15)
\drawline[AHnb=0](21,5)(26,10)
\drawline[AHnb=0](28,10)(33,15)
\drawline[AHnb=0](28,5)(38,15)
\drawline[AHnb=0](33,5)(43,15)
\drawline[AHnb=0](45,5)(50,10)
\drawline[AHnb=0](50,5)(55,10)
\drawline[AHnb=0](55,5)(60,10)

\drawline[AHnb=0](38,5)(43,10)

\put(43,10){\dots}
\put(43,5){\dots}

\put(11,4){\tiny ${v_{1}}$}
\put(16,4){\tiny ${v_{2}}$}
\put(21,4){\tiny ${v_{3}}$}
\put(26,4){\tiny ${v_{4}}$}
\put(28,4){\tiny ${v_{r-1}}$}
\put(33,4){\tiny ${v_{r}}$}
\put(38,4){\tiny ${v_{r+1}}$}
\put(41,4){\tiny ${v_{r+2}}$}
\put(45,4){\tiny ${v_{2r-1}}$}
\put(50,4){\tiny ${v_{2r}}$}
\put(55,4){\tiny ${v_{2r+1}}$}
\put(60,4){\tiny ${v_{1}}$}

\put(11.5,9){\tiny ${w_{1}}$}
\put(16.5,9){\tiny ${w_{2}}$}
\put(21.5,9){\tiny ${w_{3}}$}
\put(26,9){\tiny ${w_{4}}$}
\put(28.5,9){\tiny ${w_{r-1}}$}
\put(33.5,9){\tiny ${w_{r}}$}
\put(38.5,9){\tiny ${w_{r+1}}$}
\put(40,11){\tiny ${w_{r+2}}$}
\put(44,11){\tiny ${w_{2r-1}}$}
\put(49,11){\tiny ${w_{2r}}$}
\put(54,11){\tiny ${w_{2r+1}}$}
\put(59,11){\tiny ${w_{1}}$}

\put(11,16){\tiny ${w_{r+1}}$}
\put(16,16){\tiny ${w_{r+2}}$}
\put(20.5,16){\tiny ${w_{r+3}}$}
\put(24,16){\tiny ${w_{r+4}}$}
\put(28,16){\tiny ${w_{2r-1}}$}
\put(33,16){\tiny ${w_{2r}}$}
\put(36,16){\tiny ${w_{2r+1}}$}
\put(43,16){\tiny ${w_{1}}$}

\put(28,0){\scriptsize Figure 5 : $S\mathcal{M}_\triangle(C, M)$}

\end{picture}
\vspace{2.3cm}

Let $C$ be a cycle of type $A$ in $M$. The $S(C, M)$ is a cylinder and let $\partial S(C, M) = \{C_{1}, C_{2}\}$. Observe as above that cycle $C_1, C_2$ are of type $A$ and length($C$) = length($C$) = length($C_{2}$). 
Let $C_{1}, C_{2},$ \dots, $C_{m}$ be a sequence of cycles which are homologous to $C$ in $M$. For some $i, j$, there is a cylinder in $M$ which is bounded by $C_i$ and $C_j$. So, $C_i$ and $C_j$ are homologous and have same length (see its similar argument in $Section~2$ \cite{mu:torus-hc13}). That is, we get a cylinder $S_{l, t}$ for some cycles $C_l$ and $C_t$ such that $\partial S_{l, t} = \{C_{t}, C_{l}\}$. So, length($C_{t}$) = length($C_{l}$). Thus, length($C_{i}$) = length($C_{j}$) $\forall~ i, j\in \{1, 2, \dots, m\}$.

By Lemma \ref{lem36:2}, we assume that $S(C, M)$ contains in $M$. Here, $\partial S(C, M) = \{C', C''\}$ where $C', C''$ are two cycles of type $A$. Let $IF(L)$ denote a set of faces which are incident with $L$ and $F(S(L, M))$ denote a set of faces of $S(L, M)$. At $1^{st}$ step, let $S = F(S(C, M)) \cup IF(C') \cup IF(C'')$ and $S_1 = S$. Next step, we consider $S_1$ in place of $S$ and continue with as above. In this process, let $S_i$ denote a set at $i^{th}$ step and $\partial S_i = \{L_{i, 1}, L_{i, 2}\}$. 
Here,  $L_{i, 1} \neq L_{i, 2}$, $L_{i, 1} = L_{i, 2}$, $L_{i, 1} = \partial \mathcal{M}'$, $L_{i, 1} = \partial \mathcal{M}''$ or `$L_{i, 1} = \partial \mathcal{M}'$ and $L_{i, 2} = \partial \mathcal{M}''$' for some M\"{o}bius strips $\mathcal{M}', \mathcal{M}''$. This process stops if $L_{j, 1} = L_{j, 2}$ or `$L_{j, 1} = \partial \mathcal{M}'$ and $L_{j, 2} = \partial \mathcal{M}''$'. Let the above process stop at $k^{th}$ step. Then, the $|S_k|$ is bounded by identical cycles or two M\"{o}bius strips. Hence, we get a cylindrical representation $S(M)$ of $M$ which is bounded by two identical cycles, i.e., $L_{k, 1} = L_{k, 2}$ or a $S\mathcal{M}(M)$ which is bounded by two M\"{o}bius strips, i.e., $L_{k, 1} = \partial \mathcal{M}'$ and $L_{k, 2} = \partial \mathcal{M}''$. Again, by Lemma \ref{lem36:2}, we assume that $S\mathcal{M}_{\triangle}(C, M)$ contains in $M$. Here $\partial \mathcal{M}_{\triangle}(C, M)$ is a cycle. We repeat as above and consider $S\mathcal{M}_{\triangle}(C, M)$ in place of $S(C, M)$ in the above process. In this process, we get one boundary at each step. Thus, we get a $S\mathcal{M}(M)$ which is bounded by two M\"{o}bius strips. Therefore, by combining above two cases, we have

\begin{Lemma} \label{lem36:3} The map $M$ has a $S(M)$ or a $S\mathcal{M}(M)$ representation.
\end{Lemma}

We first assume that $S(M)$ denotes $M$. So, $\partial S(M) = \{C, C'\}$ and $C = C'$. 

\smallskip

$Claim :$ The $S(M)$ has a $(r, s, k)$-representation ($K(r, s, k)$ representation) for some $r, s, k \in \mathbb{N}$.

\smallskip

We recall similar definition of $(r, s, k)$-representation as in $Section~2$ \cite{mu:torus-hc13}. Let $v \in V(C)$. By definition of $A$, we have three paths of type $A$ through $v$. Let $L_1, L_2, L_3$ denote three paths of type $A$ through $v$ in $S(M)$. Let $L_1 = P(a_1, a_2, \dots, a_r)$ and $C = L_1$. We make a cut starting from $v$ along $P = P(a_1 = w_1, \dots, w_s = a_{k+1}) \subset L_3$ until reaching $C'$ again for the first time. Let $s$ denote the number of cycles which are homologous to $C$ along $P$. Observer that the number $s$ is equal to the length of the path $P$, that is, $s = m$. So, length($C$) = $r$ and $s$ is the number of horizontal cycles. Thus, we get a planar polyhedral representation of $M$ and denote it by {\em $(r, s)$-representation}. We say that cycle $C$ is the {\em lower (base) horizontal cycle} \cite{mu:torus-hc13} and the other cycle $C'$ is {\em upper horizontal cycle} \cite{mu:torus-hc13} in $(r, s)$-representation. The $(r, s)$-representation has identification of vertical sides (non homologous to the horizontal cycles) in the natural manner but the identification of the horizontal sides needs some shifting with twist so that a vertex in the lower(base) side is identified with a vertex in the upper side. Let $C' = C(a_{k+1}(=w_m), a_{k}, a_{k-1}, \dots, a_k)$. Then, the vertex $a_{k+1}$ is the starting vertex of the upper horizontal cycle $C'$ in $(r, s)$-representation. The $k =$ length$(P(a_1, \dots, a_{k+1}))$ where $P(a_1, \dots, a_{k+1}) \subset C$. So, we denote the $(r, s)$-representation by {\em $(r, s, k)$-representation} where $r, s, k$ are defined as above. We say that boundaries of $(r, s, k)$-representation are the cycles and paths along which we took the cuts to construct $(r, s, k)$-representation. By this construction, a $(r, s, k)$-representation exists in $S(M)$. In this article, we denote $(r, s, k)$-representation by $K(r, s, k)$. (For example in Figure 1, $K(7, 4, 1)$ denote a $(7, 4, 1)$-representation, $C = C(v_1, \dots, v_7), C' = C(v_2, v_1, v_7, v_6, v_5, v_4, v_3), P = P(v_1, w_1, x_1, u_1, v_2)$.) It's clear from the definition that $S(M)$ has a $(r, s, k)$-representation. This proves the claim.   

Next we  assume that $S\mathcal{M}(M)$ denotes $M$. 

\smallskip

$Claim :$ The $S\mathcal{M}(M)$ has $K(S'(M),  \mathcal{M}', \mathcal{M}'')$ representation.

\smallskip

The $S\mathcal{M}(M)$ bounded by two M\"{o}bius strips. We cut along the boundaries of the M\"{o}bius strips in $S\mathcal{M}(M)$. In this case, we get three components. These are a cylinder namely $S'(M)$ and two M\"{o}bius strips namely $\mathcal{M}', \mathcal{M}''$ where $\mathcal{M}', \mathcal{M}'' \in \{\mathcal{M}_\triangle(L, M)\}$ for some cycle $L$ of type $A$. We denote this representation by $K(S'(M),  \mathcal{M}', \mathcal{M}'')$. This proves the claim. 
(For an example in Figure 6, $K(S'(M),  \mathcal{M}', \mathcal{M}'')$ denotes a representation of $M, S'(M) = |\{x_1x_2w_1, w_1w_2x_2, x_2x_3w_2, w_2w_3x_3, x_3x_4w_3, w_3w_4x_4, x_4x_5w_4, w_4w_5x_5,$ $ x_5x_6w_5,$ $ w_5w_6x_6, x_6x_7w_6, w_6w_7x_7, x_7x_1w_7, w_7w_1x_1, u_1u_2x_1,x_1x_2u_2, u_2u_3x_2, x_2x_3u_3,$ $ u_3u_4x_3,$ $ x_3x_4u_4,$ $ u_4u_5x_4,$ $ x_4x_5 u_5,$ $  u_5u_6x_5,$ $ x_5x_6u_6,$ $ u_6u_7x_6, x_6x_7u_7, u_7u_1x_7, x_7x_1u_1\}|, \mathcal{M}' = |\{u_4u_5u_1, u_1u_2u_5,$ $ u_5u_6u_2,$ $ u_2u_3u_6, u_6u_7u_3, u_3u_4u_7, u_7u_1u_4\}|$ and $\mathcal{M}'' = |\{w_1w_2w_5, w_5w_6w_2, w_2w_3w_6, w_6w_7w_3,$ $ w_3w_4w_7, w_7w_1w_4, w_4w_5w_1\}|$.) 

So, the map $M$ has a $K(r, s, k)$ representation or a $K(S'(M), \mathcal{M}', \mathcal{M}'')$ representation. We first assume that $K(r, s, k)$ represents $M$. Let $L, L'$ denote the lower and upper horizontal cycle respectively in $K(r, s, k)$. Let $b$ be a vertex in $L$. By definition of $A$, let $L_1, L_2, L_3$ denote three paths of type $A$ through $b$. Let $L_1 = L = L', L_2 = P(b, \dots, x)$ and $L_3 = P(b, \dots, y)$ where $x, y \in V(L)$. The paths $L_2, L_3$ are vertical and non-homologous to $L$ in $K(r, s, k)$. Let triangles $\triangle'_1, \triangle'_2, \dots, \triangle'_m$ lie on one side of $L_2$ and $\triangle''_1, \triangle''_2, \dots, \triangle''_m$ lie on the other side of $L_2$ in $K(r, s, k)$. Observe that the geometric career $|\{\triangle'_1, \triangle'_2, \dots, \triangle'_m\}|$ is bounded by $L_2, e', L_{2, 1} :=P(b',$ \dots, $x'),$ and $e$ where $e = bb'\in E(L)$ and $e'=xx' \in E(L')$ in $K(r, s, k)$. Similarly, the geometric career $|\{\triangle''_1, \triangle''_2, \dots, \triangle''_m\}|$ is bounded by $L_2, \bar{e}', L_{2, 2},$ and $\bar{e}$ where $\bar{e} \in E(L)$ and $\bar{e}' \in E(L')$ in $K(r, s, k)$. Hence, we get two paths $L_{2, 1}$ and $L_{2, 2}$. Again, we consider $L_{2, 2}$ in place of $L_2$ and repeat as above and continue. Hence, we get a sequence of paths $L_{2, 1}, L_{2, 2},$ \dots, $L_{2, s}$. Similarly, we repeat with above process and consider cycle $L_3$ in place of $L_2$. Hence, we get an another sequence of paths namely $L_{3, 1}, L_{3, 2}$ \dots, $L_{3, s}$. Let $L_{2, 1} = P(b',$ \dots, $x')$. We extend $L_{2, 1}$ along the path of type $A$. Since $L = L'$ and have opposite orientations, so, by definition of $A$, after the vertex $x'$ the path $L_{2, 1}$ extends along a path $L_{3, j}$ for some $j$. Here, we repeat this process for each $L_{2, i}$ where $1 \le i \le s$. Let $L_{2, i}$ extends along $L_{3, j}$ for some $j$. So, we get a path $L_{2, i} \cup L_{3, j}$. Thus, we get a sequence of paths $R_1$ := $L_{2, 1} \cup L_{3, i_1}, R_2$ := $L_{2, 2} \cup L_{3, i_2}$, \dots, $R_s$ := $L_{2, s} \cup L_{3, i_s}$. So, we have three paths $L, L_{2, i}$ and $L_{3, j}$ through $b$ for some $i, j$. The paths $L_{2, i}$ and $L_{3, j}$ are part of some $R_{l_1}$  and $R_{l_2}$. Therefore, let $P', P'', P'''$ be three paths of type $A$ through $b \in V(L)$. Then, $P' \subset R_k, P'' \subset R_{k'}$ and $P''' \subset L$ for some $k, k'$. Since the paths $R_k$ and $R_{k'}$ are same type and part of some walk of type $A$ in $M$, it follows that we have only one cycle $L$ of type $A$ through $b$. Therefore, we have only one cycle of type $A$ through each vertex of $M$. Combining above all cases, we have

\begin{Lemma}\label{lem36:4} The map $M$ has a $K(r, s, k)$ or a $K(S'(M), \mathcal{M}', \mathcal{M}'')$ representation but not both. 
\end{Lemma}

\begin{proof} By the preceding section, the map $M$ has either a $K(r, s, k)$ or a $K(S'(M), \mathcal{M}', \mathcal{M}'')$ representation. In a map $M$, we have one cycle of type $A$ as above though each vertex of $M$. We are using this cycle to construct both the representations. Since there is only one cycle of type $A$ through a vertex and we are using this cycle to construct the above two representations, it implies that map $M$ can not have both the representations simultaneously. So, $K(r, s, k)$ and $K(S'(M), \mathcal{M}', \mathcal{M}'')$ are non isomorphic. That is, $M$ does not have both the representations. Therefore, we study both the representations bellow separately. This completes the proof.
\end{proof}

\begin{Lemma}\label{lem36:5} The cycles of type $A$ in $M$ have same length.
\end{Lemma}

\begin{proof} Let $v\in V(M)$. The $M$ has either a $K(r, s, k)$ or a $K(S'(M), \mathcal{M}', \mathcal{M}'')$ representation. In $K(r, s, k)$, we have one cycle and one walk through $v$ as above. Let $C'_1, \dots, C'_s$ denote a sequence of horizontal cycles of type $A$ which are homologous to $C'_1$ in $K(r, s, k)$. By the preceding argument, length($C'_i)$ = length($C'_j$) for all $1 \le i, j \le s$. So, there is an unique cycle of type $A$ in $K(r, s, k)$ upto homologous. Therefore, the cycles of type $A$ have same length. Similarly we can proceed with argument as above in $K(S'(M), \mathcal{M}', \mathcal{M}'')$ and hence, we get same result. This completes the proof.
\end{proof}

In Lemma \ref{lem36:4}, the representations $K(r, s, k)$ and $K(S'(M), \mathcal{M}', \mathcal{M}'')$ are non-isomorphic and these representations are exist in maps of type $\{3^6\}$ on the torus. So, we study both the representations separately and classify them in bellow.

\textbf{Classification of $K(r, s, k)$ on $n$ vertices :} We define admissible relations among $r, s, k$ of $K(r, s, k)$ such that $K(r, s, k)$ represents a map after identifying its boundaries.

\begin{Lemma}\label{lem36:6} The maps of type $\{3^{6}\}$ of the form $K(r, s, k)$ exist if and only if the following holds : (i) rs $\geq$ 9, (ii) r $\geq$ 3, (iii) s $\ge$ 3 and (iv) $k \in \{ t : 0 \leq t \leq r-1\}$ if $s \ge 3$. 
\end{Lemma}

\begin{proof} One can repeat similar arguments as in $Section~ 2$ \cite{mu:torus-hc13} and study each cases. If $s \le 2$ or $r \le 2$ or rs $<$ 9 in $K(r, s, k)$ then we get some vertex whose link is not a cycle. Similarly, it is easy to observe that $k \in \{ t : 0 \leq t \leq r-1\}$ if $s \ge 3$ in $K(r, s, k)$ since length of the lower horizontal cycle is $r$. 
\end{proof}

Let $M_{1}$ and $M_{2}$ be two maps of type $\{3^{6}\}$ on the Klein bottle. Let $K(r_{i}, s_{i}, k_{i})$ denote a planar polyhedral representation of $M_i$. Now we have

\begin{Lemma}\label{lem36:iso}
Let $K(r_{i}, s_{i}, k_{i})$ denote a $(r_i, s_i, k_i)$ -representation of $M_{i}$ on $n$ vertices for $i \in \{1, 2\}$. Then,\ map $M_{1} \cong M_{2}$ if $(r_1, s_1, l_1) = (r_2, s_2, l_2)$ where $ l_i \in \{k_i~ mod(2), (k_i + r_i)~ mod(2)\}$.
\end{Lemma}

\begin{proof} First assume that $r_1 = r_2, s_1 = s_2, k_1 =k_2$. Let $C(1, 0), C(1, 1), \dots C(1, s_1-1)$ denote a sequence of horizontal cycles of type $A$ in $K(r_{1}, s_{1}, k_{1})$. Define $C(1, 0) := C(u_{0,0}, u_{0,1}, \dots, u_{0, r_{1}-1}), C(1, 1) :=C(u_{1, 0}, u_{1, 1}, \dots, u_{1, r_{1}-1})$,\dots, $C(1, s_{1}-1) :=C(u_{s_{1}-1, 0},$ $u_{s_{1}-1, 1}, \dots, u_{s_{1}-1, r_{1}-1})$. Again, let $C(2, 0), C(2, 1),$ \dots $C(2, s_2-1)$ denote a sequence of horizontal cycles of type $A$ in $K(r_{2}, s_{2}, k_{2})$. Define $C(2, 0) := C(v_{0, 0}, v_{0, 1}, \dots, v_{0, r_{2}-1}), C(2, 1) :=C(v_{1, 0}, v_{1, 1}, \dots, v_{1, r_{2}-1}, v_{1, 0})$,\dots, $C(2, s_{2}-1) :=C(v_{s_{2}-1, 0},  v_{s_{2}-1, 1},  \dots, v_{s_{2}-1, r_{2}-1})$. Since maps of type $\{3^6\}$ contain unique cycle of type $A$ up to homologous, it follows by Lemma \ref{lem36:6} that $r_1 = r_2$ and $s_1 =s_2$. So, we have the following cases.

\textbf{Case 1 :} If $k = k_{1} = k_{2}$ then define a map $f_{1} : V(K(r_{1}, s_{1}, k_{1}))\rightarrow V(K(r_{2}, s_{2},$ $k_{2}))$ by $f_1(u_{t,i}) = v_{t,i}$ for $0\leq t\leq s-1$ and $0\leq i\leq r-1$. By definition, $f_1(lk(u_{t,i})) = f_1(C(\textit{u}_{t-1, i-1},$ $ u_{t-1, i},$ $ \textit{u}_{t-1, i+1},$ $ u_{t, i+1},$ $ u_{t+1, i+1},$ $ u_{t+1, i},$ $u_{t, i-1})) = (f_1(\textit{u}_{t-1, i-1}),$ $f_1(u_{t-1, i}),$ $ f_1(\textit{u}_{t-1, i+1}),$ $ f_1(u_{t, i+1}),$ $f_1(u_{t+1, i+1}),$ $f_1(u_{t+1, i}),$ $ f_1(u_{t, i-1})) = C(\textit{v}_{t-1, i-1},$ $v_{t-1, i},$ $\textit{v}_{t-1, i+1},$ $v_{t, i+1}, $ $ v_{t+1, i+1},$ $ v_{t+1, i}, v_{t, i-1})$. So, $f_1(lk(u_{t,i})) = lk(v_{t, i})$. That is, $f_1$ sends vertices to vertices, edges to edges, faces to faces and also, preserves incidents. So, $f_{1}$ is an isomorphism map. Therefore, $K(r_1, s_1, k_1) \cong  K(r_2, s_2, k_2)$.   

\textbf{Case 2 :} If $k_1 \ne k_2$ then by assumption, let $l_1 = l_2 = 0$. Here, we have the following cases : $2 \mid k_i$ or $2 \mid (k_i + r_i)$. 

Let $2 \mid k_i$. The vertex $u_{0, \frac{k_i}{2}} \in V(C(i, 0))$. Let $Q_1, Q_2, Q_3$ denote three paths of type $A$ through $u_{0, \frac{k_i}{2}}$ in $K(r_i, s_i, k_i)$. Let $C(i, s_i) := C(i, 0), Q_1 = C(i, 0)$ and $Q_3 = P(u_{0, \frac{k_i}{2}},$ \dots, $w (=u_{0, \frac{k_i}{2}}))$ where $w \in V(C(i, s_i))$. We identify $K(r_i, s_i, k_i)$ along vertical boundaries and cut along $Q_3$. In this process, we are not changing the horizontal cycles and number of horizontal cycles vertically. So in this case, we get a $K(r_i, s_i, 0)$ representation of $M_i$. 

Let $2 \mid (k_i + r_i)$. Similarly as above, $u_{0, \frac{k_i +r_i}{2}} \in V(C(i, 0))$. Let $L_1, L_2, L_3$ denote three paths through $u_{0, \frac{k_i +r_i}{2}}$ in $T(r_i, s_i, k_i)$. Let $C(i, s_i) := C(i, 0), L_1 = C(i, 0)$ and $L_3 = P(u_{0, \frac{k_i +r_i}{2}}, \dots, v (=u_{0, \frac{k_i +r_i}{2}}))$ where $v \in V(C(i, s_i))$. We identify $T(r_i, s_i, k_i)$ along vertical boundaries and cut along $L_3$. Hence, we get a $K(r_i, s_i, 0)$ representation of $M_i$. 

So, we get a $K(r_i, s_i, 0)$ representation from $K(r_i, s_i, k_i)$ if $2 \mid k_i$ or $2 \mid (k_i + r_i)$. Since $r_1 = r_2$ and $s_1 = s_2$, it follows by $f_1, K(r_1, s_1, 0) \cong  K(r_2, s_2, 0)$.

Again, let $l_1 = l_2 = 1$. If $2 \mid (k_i - 1)$ then similarly as above, we identify $K(r_i, s_i, k_i)$ along their vertical boundaries and cut along the vertical path of type $A$ through $u_{0, \frac{k_i - 1}{2}} \in V(C(i, 0))$ in $K(r_i, s_i, k_i)$. Again, if $2 \nmid (k_i - 1)$ then $2 \mid (k_i + r_i-1)$. In this case, we consider vertical path of type $A$ through $u_{0, \frac{k_i +r_i-1}{2}} \in V(C(i, 0))$. Here, in both the cases, we get a $K(r_i, s_i, 1)$ representation of $M_i$. Since $r_1 =r_2$ and $s_1 = s_2$, if follows by $f_1, K(r_1, s_1, 1) \cong  K(r_2, s_2, 1)$. 

Hence by $f_1$, $M_{1} \cong M_{2}$. This completes the proof.
\end{proof}

Thus, we have 

\begin{cor} \label{cor36:1} Let $K(r_{i}, s_{i}, k_{i})$ denote a $(r_{i}, s_{i}, k_{i})$-representation of $M_{i}$ on $n$ vertices for $i \in \{1, 2\}$. Then, $K(r_{1}, s_{1}, k_{1}) \not\cong K(r_{2}, s_{2}, k_{2}) ~\forall~ r_{1} \neq r_{2}, K(r_{1}, s_{1}, k_{1}) \not\cong K(r_{2}, s_{2}, k_{_{2}}) ~\forall~ s_{1} \neq s_{2}, K(r_{i}, s_{i}, 0) \cong K(r_{i}, s_{i}, k_{i})$ for $2 \mid k_i$ or $2 \mid (k_i + r_i)$, and $K(r_{i}, s_{i},  1) \cong K(r_{i}, s_{i}, k_{i})$ for $2 \mid (k_i - 1)$ or $2 \mid (k_i + r_i - 1)$.
\end{cor}

\textbf{Classification of $K(l, t)$ on $n$ vertices :} 
Let $S\mathcal{M}(M)$ denote a representation of $M$ on $n$ vertices. By Lemma \ref{lem36:4}, the $S\mathcal{M}(M)$ has a $K(S'(M), \mathcal{M}', \mathcal{M}'')$ representation where $C' = \partial\mathcal{M}', C''=\partial\mathcal{M}''$ are two homolgous cycles of type $A$ and $\partial S'(M) = \{C', C''\}$. Hence, we denote $K(S'(M), \mathcal{M}', \mathcal{M}'')$ by $K(l, t)$ if $l = length(C') = length(C'')$ and $t$ denote the number of cycles of type $A$ in $S'(M)$.

\begin{Lemma}\label{lem36.7} The maps of type $\{3^{6}\}$ of the form $K(l, t)$ exist if and only if the following holds : (i) $tl \ge 10$, (ii) $t \ge 2$,  (iii) $2 \nmid l$ and $l \ge 5$.
\end{Lemma}

\begin{proof}  If $t = 1$ then we get some vertex in $C'$ whose link is not a cycle. So, $S'(M)$ contains at least two disjoint cycles of type $A$ and hence, $t \ge 2$. Again, the $\mathcal{M}'$ is a M\"{o}bius strip which has one boundary. By the argument as in Lemma \ref{lem36:2}, $\partial \mathcal{M}' \subset \partial S'(M)$. Let $t_1$ denote the number of cycles of type $A$ in $\mathcal{M}'$. Then, by the argument of Lemma \ref{lem36:2}, $t_1 = 1$. Similarly, let $t_2$ denote the number of cycles of type $A$ in $\mathcal{M}''$ then  $t_2 = 1$. Therefore, $(t + t_1 + t_2-2) \ge 2$ as $\partial S'(M) = \{\partial \mathcal{M}', \partial \mathcal{M}'' \}$. The cycles of type $A$ covers all the vertices of the map $M$. So, $n = (t+t_1+t_2-2)l = tl$. This implies that $tl \mid n, l \mid n$ and $t \mid n$. By the argument of Lemma \ref{lem36:2}, $2 \nmid l$. We follow similar argument as in Lemma \ref{lem36:6} and hence, $l \ge 5$ and $tl \ge 10$. This completes the proof.   
\end{proof}

Let $K(l_i, t_i)$ denote a representation of $M_i$ on $n$ vertices. Then, we have 

\begin{Lemma}\label{lem36.8}
 The $K(l_1, t_1) \cong K(l_2, t_2)$ if $l_1 = l_2$.  
\end{Lemma}

\begin{proof} Let $l_1 = l_2$. Let $C_1, \dots, C_s$ denote cycles of type $A$ in $M_1$. Let $C_1 = \partial \mathcal{M}'_1$ and $C_s = \partial \mathcal{M}''_1$. Again, let $L_1, \dots, L_s$ denote cycles of type $A$ in $M_2$. Let $L_1 = \partial \mathcal{M}'_2$ and $L_s = \partial \mathcal{M}''_2$. We follow similar argument as in Lemma \ref{lem36:iso}, and hence, we define an isomorphism map $f_1 : V(M_1) \rightarrow V(M_2)$ by $f_1(V(C_i)) = V(L_i)$ for $1 \le i \le s$. So, by $f_1, K(l_1, t_1) \cong K(l_2, t_2)$ i.e. $M_1 \cong M_2$. This completes the proof.
\end{proof}

Thus, we have 

\begin{cor} \label{cor36:2} Let $K(l_{i}, t_i)$ denote  $K(l_i, S'(M_i), \mathcal{M}'_i, \mathcal{M}''_i)$ representation of $M_{i}$ on $n$ vertices for $i \in \{1, 2\}$. Then,\ $K(l_{1}, t_{1}) \not\cong K(l_{2}, t_{2}) ~\forall~ l_{1} \neq l_{2}$ and $K(l_{1}, t_{1}) \not\cong K(l_{2}, t_{2})~ \forall~ t_{1} \neq t_{2}$.
\end{cor}

\noindent \textbf{2.2 Maps of type $\{4^4\}$~:~}
Let $M$ be a map of type $\{4^{4}\}$ on the Klein bottle. A path $P(\dots, u_{i-1},$ $u_{i},$ $u_{i+1}, \dots)$ in edge graph of $M$ is of type $B$ if $lk(u_{i}) = C(\textbf{a}, b, \textbf{c} u_{i+1}, \textbf{d}, e, \textbf{f},$ $u_{i-1})$ implies $lk(u_{i+1})=C(\textbf{b}, c, \textbf{g} u_{i+2}, \textbf{h}, d, \textbf{e}, u_{i})$ and $lk(u_{i-1}) = C(\textbf{k}, a, \textbf{b} u_{i}, \textbf{e}, f, \textbf{l}, u_{i-2})$ at $u_{i}$ for all $i$.
In the definition, the bold vertex represents non adjacent vertex. Let $W$ be a maximal walk of type $B$. Then, similarly as in Section 2.1, it is either a cycle or a walk. Since degree of a vertex is four, it implies that every closed walk of type $B$ is a cycle. Let $C$ be a cycle of type $B$ in $M$ and $S$ denote a set of faces which are incident with $C$. We argue similarly as in Lemma \ref{lem36:2}. In the proof we consider quadrangles in place of triangles and repeat similar argument as in Lemma \ref{lem36:2}. Hence, we get a geometric carrier $|S|$ which is either $S(C, M), \mathcal{M}(C, M)$ or $S\mathcal{M}_{4}(C, M)$. So, the map $M$ contains $S(C, M), \mathcal{M}(C, M)$ or  $S\mathcal{M}_{4}(C, M)$ for some cycle $C$ of type $B$ in $M$. (see Figure 7, it is an example of a M\"{o}bius strip $\mathcal{M}(C, M)$ where $C = C(w_1, \dots, w_r)$.) We consider each of $S(C, M), \mathcal{M}(C, M)$ and $\mathcal{M}_{4}(C, M)$ and repeat the similar arguments as in Lemma \ref{lem36:3}. Thus, we get a representation of $M$ which is either a $S(M)$ which is bounded by identical cycles or a $S\mathcal{M}(M)$ with two M\"{o}bius strips $\mathcal{M}', \mathcal{M}'' \in \{\mathcal{M}(C, M), \mathcal{M}_{4}(L, M) \mid$ for some cycles $L, C$ of type $B \}$. Therefore, the $M$ has $S(M)$ or $S\mathcal{M}(M)$ representation. 

\vspace{-1cm}

\begin{picture}(0,0)(-35,40)
\tiny
\setlength{\unitlength}{1.5mm}
\drawpolygon(5,5)(20,5)(20,15)(5,15)
\drawpolygon(22,5)(37,5)(37,15)(22,15)


\drawline[AHnb=0](10,5)(10,15)
\drawline[AHnb=0](15,5)(15,15)
\drawline[AHnb=0](27,5)(27,15)
\drawline[AHnb=0](32,5)(32,15)
\drawline[AHnb=0](5,10)(20,10)
\drawline[AHnb=0](22,10)(37,10)

\put(20,10){\dots}
\put(20,5){\dots}
\put(20,15){\dots}

\put(5,3.8){\tiny ${v_{1}}$}
\put(10,3.8){\tiny ${v_{2}}$}
\put(15,3.8){\tiny ${v_{3}}$}
\put(18,3.8){\tiny ${v_{4}}$}
\put(22,3.8){\tiny ${v_{r-2}}$}
\put(27,3.8){\tiny ${v_{r-1}}$}
\put(32,3.8){\tiny ${v_{r}}$}
\put(37,3.8){\tiny ${x_{1}}$}

\put(5.5,8.8){\tiny ${w_{1}}$}
\put(10.5,8.8){\tiny ${w_{2}}$}
\put(15.5,8.8){\tiny ${w_{3}}$}
\put(17.9,8.8){\tiny ${w_{4}}$}
\put(22.5,8.8){\tiny ${w_{r-2}}$}
\put(27.5,8.8){\tiny ${w_{r-1}}$}
\put(32.5,8.8){\tiny ${w_{r}}$}
\put(37.5,8.8){\tiny ${w_{1}}$}

\put(5.5,13.8){\tiny ${x_{1}}$}
\put(10.5,13.8){\tiny ${x_{2}}$}
\put(15.5,13.8){\tiny ${x_{3}}$}
\put(18,13.8){\tiny ${x_{4}}$}
\put(22.5,13.8){\tiny ${x_{r-2}}$}
\put(27.5,13.8){\tiny ${x_{r-1}}$}
\put(32.5,13.8){\tiny ${x_{r}}$}
\put(37.5,13.8){\tiny ${v_{1}}$}

\put(15,0){\tiny Figure 7 : $\mathcal{M}(C, M)$}

\end{picture}

\vspace{3.5cm}

\begin{picture}(0,0)(10,2)
\tiny
\setlength{\unitlength}{1.3mm}
\drawpolygon(5,10)(35,10)(35,20)(5,20)
\drawpolygon(5,5)(20,5)(20,10)(5,10)
\drawpolygon(5,20)(20,20)(20,25)(5,25)


\drawline[AHnb=0](10,5)(10,25)
\drawline[AHnb=0](15,5)(15,25)
\drawline[AHnb=0](20,5)(20,25)
\drawline[AHnb=0](25,10)(25,20)
\drawline[AHnb=0](30,10)(30,20)
\drawline[AHnb=0](5,10)(35,10)
\drawline[AHnb=0](5,15)(35,15)
\drawline[AHnb=0](5,20)(35,20)

\put(5,3.8){\tiny ${w_{4}}$}
\put(10,3.8){\tiny ${w_{5}}$}
\put(15,3.8){\tiny ${w_{6}}$}
\put(20,3.8){\tiny ${w_{1}}$}

\put(5.5,8.8){\tiny ${w_{1}}$}
\put(10.5,8.8){\tiny ${w_{2}}$}
\put(15.5,8.8){\tiny ${w_{3}}$}
\put(20.5,8.8){\tiny ${w_{4}}$}
\put(25.5,8.8){\tiny ${w_{5}}$}
\put(30.5,8.8){\tiny ${w_{6}}$}
\put(35.5,8.8){\tiny ${w_{1}}$}

\put(5.5,13.8){\tiny ${x_{1}}$}
\put(10.5,13.8){\tiny ${x_{2}}$}
\put(15.5,13.8){\tiny ${x_{3}}$}
\put(20.5,13.8){\tiny ${x_{4}}$}
\put(25.5,13.8){\tiny ${x_{5}}$}
\put(30.5,13.8){\tiny ${x_{6}}$}
\put(35.5,13.8){\tiny ${x_{1}}$}

\put(5.5,18.8){\tiny ${u_{1}}$}
\put(10.5,18.8){\tiny ${u_{2}}$}
\put(15.5,18.8){\tiny ${u_{3}}$}
\put(20.5,18.8){\tiny ${u_{4}}$}
\put(25.5,18.8){\tiny ${u_{5}}$}
\put(30.5,18.8){\tiny ${u_{6}}$}
\put(35.5,18.8){\tiny ${u_{1}}$}

\put(5.5,23.8){\tiny ${u_{4}}$}
\put(10.5,23.8){\tiny ${u_{5}}$}
\put(15.5,23.8){\tiny ${u_{6}}$}
\put(20.5,23.8){\tiny ${u_{1}}$}

\put(17,0){\tiny Figure 8 }
\end{picture}

\begin{picture}(0,0)(-90,-3)
\tiny
\setlength{\unitlength}{1.4mm}

\drawpolygon(5,5)(35,5)(35,20)(5,20)

\drawline[AHnb=0](10,5)(10,20)
\drawline[AHnb=0](15,5)(15,20)
\drawline[AHnb=0](20,5)(20,20)
\drawline[AHnb=0](25,5)(25,20)
\drawline[AHnb=0](30,5)(30,20)
\drawline[AHnb=0](5,10)(35,10)
\drawline[AHnb=0](5,15)(35,15)
\drawline[AHnb=0](5,20)(35,20)

\put(5,3.8){\tiny ${w_{4}}$}
\put(10,3.8){\tiny ${w_{1}}$}
\put(15,3.8){\tiny ${x_{1}}$}
\put(20,3.8){\tiny ${u_{1}}$}
\put(25,3.8){\tiny ${u_{4}}$}
\put(30,3.8){\tiny ${x_{4}}$}
\put(35,3.8){\tiny ${w_{4}}$}

\put(5.5,8.8){\tiny ${w_{5}}$}
\put(10.5,8.8){\tiny ${w_{2}}$}
\put(15.5,8.8){\tiny ${x_{2}}$}
\put(20.5,8.8){\tiny ${u_{2}}$}
\put(25.5,8.8){\tiny ${u_{5}}$}
\put(30.5,8.8){\tiny ${x_{5}}$}
\put(35.5,8.8){\tiny ${w_{5}}$}

\put(5.5,13.8){\tiny ${w_{6}}$}
\put(10.5,13.8){\tiny ${w_{3}}$}
\put(15.5,13.8){\tiny ${x_{3}}$}
\put(20.5,13.8){\tiny ${u_{3}}$}
\put(25.5,13.8){\tiny ${u_{6}}$}
\put(30.5,13.8){\tiny ${x_{6}}$}
\put(35.5,13.8){\tiny ${w_{6}}$}

\put(5.5,18.8){\tiny ${w_{1}}$}
\put(10.5,18.8){\tiny ${w_{4}}$}
\put(15.5,18.8){\tiny ${x_{4}}$}
\put(20.5,18.8){\tiny ${u_{4}}$}
\put(25.5,18.8){\tiny ${u_{1}}$}
\put(30.5,18.8){\tiny ${x_{1}}$}
\put(35.5,18.8){\tiny ${w_{1}}$}

\put(12,0){\tiny Figure 9 : K(6, 3, 1)}
\end{picture}

\vspace{-2mm}

Let $S(M)$ denote a representation of $M$. Let $\partial S(M) =\{C, C'\}$ where $C, C'$ are two cycles of type $B$. We repeat the similar argument as in Section 2.1 which is done to define $(r, s, k)$-representation of a map. Let $v \in V(C)$ and $L = P(v,$ \dots, $ w)$ be a path of type $B$ where $w \in V(C')$ and $L$ is not homologous to $C$. Then, we take second cut along $L$. Hence, we get a $(r, s, k)$-representation of $M$ and it is denoted by $K(r, s, k)$ for some $r, s, k$.        Let $S\mathcal{M}(M)$ denote a representation of $M$. We cut along the boundaries of M\"{o}bius strips, and hence, we get three components which are a cylinder namely $S'(M)$ and two M\"{o}bius strips namely $\mathcal{M}', \mathcal{M}''$ where $\mathcal{M}', \mathcal{M}'' \in \{\mathcal{M}(L_1, M), \mathcal{M}_{4}(L_2, M) \mid L_1, L_2$ are two cycles of type $B \}$. Therefore, by the similar argument as in Section 2.1, the map $M$ has a $K(r, s, k)$ or a  $K(S'(M), \mathcal{M}',  \mathcal{M}'')$ representation. So, we have 

\begin{Lemma}\label{lem44:1} The map $M$ has a $K(r, s, k)$ representation. 
\end{Lemma}

\begin{proof} By the preceding above arguments, the map $M$ has either a $K(r, s, k)$ or a $K(S'(M),$ $\mathcal{M}',$ $\mathcal{M}'')$ representation. Assume that $K(S'(M), \mathcal{M}', \mathcal{M}'')$ denotes $M$. Let $\partial S'(M) = \{C, C'\}, C = C(u_1, \dots, u_l)$ and $C' = C(v_1, \dots, v_l)$, and $\mathcal{M}' = \mathcal{M}(L, M), \mathcal{M}'' = \mathcal{M}(L', M)$ for some cycles $L, L'$ of type $B$. (For an example in Figure 8, $L= C(u_1, \dots, u_6), L' = C(w_1, w_2, \dots, w_6), \mathcal{M}' = \mathcal{M}_4(L, M) = |\{[u_1, u_2, u_5, u_4], [u_2, u_3, u_6, u_5], [u_3, u_4, u_1, u_6]\}|, \mathcal{M}''$ $= \mathcal{M}_4(L', M) = |\{[w_1, w_2, w_5, w_4], [w_2, w_3, w_6, w_5], [w_3, w_4, w_1, w_6]\}|$ and $S'(M) = |\{ [x_1,$ $x_2,$ $w_2,$ $w_1], [x_2, x_3, w_3, w_2], [x_3, x_4, w_4, w_3], [x_4, x_5, w_5, w_4], [x_5, x_6, w_6, w_5], [x_6, x_1, w_1, w_6], [x_1, x_2,$ $u_2,$ $u_1], [x_2, x_3, u_3, u_2], [x_3, x_4, u_4, u_3], [x_4, x_5, u_5, u_4], [x_5, x_6, u_6, u_5], [x_6, x_1, u_1, u_6] \}|$.) Let $u \in V(L)$. We cut $\mathcal{M}(L, M)$ along the path of type $B$ through $u$. Thus, we get a rectangular planar representation of $\mathcal{M}(L, M)$ which has three horizontal paths of same length. Since above and lower paths define a cycle, it implies that its length is twice of length($L$). By definition of $B$, we have two paths $Q_1, Q_2$ of type $B$ through $u_1$. Let $Q_1 \subset C$. In this case, we cut $\mathcal{M}(L, M)$ along the path $P(u_1, x, u_{l'+1}) \subset Q_2$ of type $B$ where $x \in V(L)$. Hence, we get a planar polyhedral representation, say $\Box(\mathcal{M})$. (For an example in Figure 8, we cut $\mathcal{M}(L, M)$ along $P(u_1, u_4)$ and hence, we get $\Box(\mathcal{M}_4)$ which is bounded by $P(u_1, u_4), P(u_4, u_5, u_6, u_1), P(u_1, u_4), P(u_4, u_3, u_2, u_1)$.) Again, let $Z_1, Z_2$ denote two paths of type $B$ through $u_{l'+1}$. Let $Z_1 \subset C$. Similarly, we cut $S'(M)$ along the path $P(u_1, \dots, v_{t}) \subset Q_2$ of type $B$ and then, along $P(u_{l'+1}, \dots, v_{l'+t}) \subset Z_2$ of type $B$. Hence, we get two planar polyhedral representations $\Box'$ and $\Box''$. (For example in Figure 8, we cut $S'(M)$ along $P(u_1, x_1, w_1)$ and $P(u_4, x_4, w_4)$ and we get $\Box' = |\{ [x_1, x_2, w_2, w_1], [x_2, x_3, w_3,$ $w_2], [x_3,$ $x_4, w_4, w_3], [x_1, x_2, u_2, u_1], [x_2, x_3, u_3, u_2], [x_3, x_4, u_4, u_3]\}|$ and $\Box'' = |\{ [x_4, x_5, u_5,$ $u_4],$ $[x_5, x_6, u_6, u_5], [x_6, x_1, u_1, u_6], [x_4, x_5, w_5, w_4], [x_5, x_6, w_6, w_5], [x_6, x_1, w_1, w_6]\}|$.) Similarly, we cut $\mathcal{M}(L', M)$ along $P(v_t, z, v_{l'+t}) \subset Q_2$ of type $B$ where $z \in V(L')$. Hence, we get $\Box(\mathcal{M'})$. (For example in Figure 8, we cut $\mathcal{M}(L', M)$ along $P(w_1, w_4)$ and hence, we get $\Box(\mathcal{M}_4')$ which is bounded by $P(w_1, w_4), P(w_4, w_5, w_6, w_1), P(w_1, w_4), P(w_4, w_3, w_2, w_1)$.) The boundaries of $\Box(\mathcal{M})$ are $P(u_1, x, u_{l'+1}), P(u_{l'+1}, \dots, u_{l}, u_1), P(u_{l'+1}, x, u_{1})$ and $P(u_{l'+1}, \dots, u_{1})$. The boundaries of $\Box'$ are $P(u_1, u_2, \dots, u_{l'+1}), P(u_{l'+1}, \dots, v_{t+l'}), P(v_{t+l'}, \dots, v_{t})$ and $P(v_{t}, \dots,$ $u_2, u_{1})$. The boundaries of $\Box''$ are $P(u_{l'+1}, \dots, u_{l}, u_1), P(u_{1}, \dots, v_{t}), P(v_{t}, \dots, v_{t+l'})$ and $P(v_{t+l'}, \dots, u_{l'+1})$. The boundaries of $\Box(\mathcal{M'})$ are $P(v_t, z, v_{l'+t}), P(v_{l'+t}, \dots, v_t), P(v_{t}, z,$ $v_{t+l'})$ and $P(v_{l'+t}, \dots, v_{t})$. Now, we consider $\Box(\mathcal{M})$ and $\Box'$, and identify along $P(u_1, \dots,$ $u_{l'+1})$. Then, we consider $\Box(\mathcal{M})$ and $\Box''$, and identify along $P(u_{l'+1}, \dots, u_{1})$. Again, we consider $\Box''$ and $\Box(\mathcal{M'})$, and identify along $P(v_{t+l'}, \dots, v_{t})$. Hence, we get a planar polyhedral representation of $K(S'(M), \mathcal{M}', \mathcal{M}'')$ which is bounded by cycle $\mathcal{C} = P(v_t, \dots, u_1) \cup P(u_1, \dots, u_{l'+1}) \cup P(u_{l'+1}, \dots, v_{t+l'}) \cup P(v_{t+l'}, \dots, v_t)$ and path $P(v_t, \dots, v_{t+l'})$. Hence, we get a $K(r', s', k')$ representation from $K(S'(M), \mathcal{M}', \mathcal{M}'')$ for some $k'$ where $r'$ = length($\mathcal{C}$). (For example in Figure 9, we get $K(6, 3, 1)$ from Figure 8.) So, every map $M$ has $K(r, s, k)$ representation for some $r, s, k$. Similar argument as above we repeat if `$\mathcal{M}' = \mathcal{M}_4(L, M)$ and $\mathcal{M}'' = \mathcal{M}(L', M)$', and if `$\mathcal{M}' = \mathcal{M}_4(L, M)$ and $\mathcal{M}'' = \mathcal{M}_4(L', M)$'.
\end{proof}

So, by the above Lemma \ref{lem44:1}, $K(r, s, k)$ denotes $M$ for some $r, s, k$. The $K(r, s, k)$ has two cycles of type $B$ through each vertex. So, by the similar argument as in Lemma  \ref{lem36:5}, it has two cycles of type $B$ upto homologous. Let $C'$ and $C''$ denote two cycles in $M$ through a vertex. The cycles are non-contractible and represent two generators of the fundamental group $\mathbb{Z} \times \mathbb{Z}_2$ of the Klein bottle as it is discussed in Section 2.1. Since $\mathbb{Z} \ncong \mathbb{Z}_2$, it follows that through each vertex of $M$, there is only one cycle of type $B$ which is homologous to $\mathbb{Z}$ and denote it by $C'$. In this case, by the above preceding argument, we get a $K(r, s, k)$ representation of $M$. Similarly, when we consider $C''$ then we get a $K(S'(M), \mathcal{M}', \mathcal{M}'')$ representation of $M$. Since these two cycles represent generator of two different groups $\mathbb{Z}, \mathbb{Z}_2$, it implies that there is no isomorphism map which maps cycle $C'$ to $C''$. Therefore, we consider cycle $C'$ and its homologous cycles in the proof of Lemma \ref{lem44:iso}. By the above Lemma \ref{lem44:1}, every map has a $K(r, s, k)$ representation. So, we study only $K(r, s, k)$ representations. 

We define admissible relations among $r, s, k$ of $K(r, s, k)$ such that $K(r, s, k)$ represents a map after identifying its boundaries. The proof of next Lemma \ref{lem44:2} similar as in Lemma \ref{lem36:6}. We consider each case separately and if we consider value of $r, s, k$ not belong to the mentioned ranges in Lemma \ref{lem44:2} then we get some vertex whose link is not a cycle. Thus, we have 

\begin{Lemma}\label{lem44:2} The maps of type $\{4^{4}\}$ of the form $K(r, s, k)$ exist if and only if the following holds : (i) rs $\geq$ 9, (ii) r $\geq$ 3, (iii) s $\ge$ 3, and (iv) $k \in \{ t : 0 \leq t \leq r-1\}$ if $s \ge 3$.
\end{Lemma}

Let $M_{1}$ and $M_{2}$ be two maps of type $\{4^{4}\}$ on the Klein bottle. Then, we have

\begin{Lemma}\label{lem44:iso}
Let $K(r_{i}, s_{i}, k_{i})$ denote a $(r_i, s_i, k_i)$ -representation of $M_{i}$ on $n$ vertices for $i \in \{1, 2\}$. Then,\ $M_{1} \cong M_{2}$ if $(r_1, s_1, l_1) = (r_2, s_2, l_2)$ where $ l_i \in \{k_i~ mod(2), (k_i + r_i)~ mod(2)\}$.
\end{Lemma}

\begin{proof} The proof repeats the similar argument as in Lemma \ref{lem36:iso}. Let $r_1 = r_2, s_1 = s_2, k_1 =k_2$. Define an isomorphism map $f_{1} : V(K(r_{1}, s_{1}, k_{1}))\rightarrow V(K(r_{2}, s_{2}, k_{2}))$ as in Lemma \ref{lem36:iso}. So, $M_{1} \cong M_{2}$ by $f_1$.    When $k_{1} \neq k_{2}$ then we repeat similar argument as in Lemma \ref{lem36:iso}. Hence, we get either a $K(r_i, s_i, 0)$ or a $K(r_i, s_i, 1)$ representation from $K(r_i, s_i, k_i)$. So, by $f_1, M_{1} \cong M_{2}$. There are two cycles $L_{1, i}$ and $L_{2, i}$ through a vertex in $M_i$. Let $L_{1, i}$ denote a horizontal cycle in $K(r_{i}, s_{i}, k_{i})$. By the above preceding argument, there is no isomorphism map which maps $L_{1, i}$ to $L_{2, i}$ since these two cycles represent generator of non-isomorphic groups. So, we do not consider cycle $L_{2, i}$ in the above case. Therefore, the given conditions are the sufficient conditions for isomorphism. This completes the proof.
\end{proof}

Thus, we have 

\begin{cor} \label{cor44:1} Let $K(r_{i}, s_{i}, k_{i})$ denote a $(r_{i}, s_{i}, k_{i})$-representation of $M_{i}$ on $n$ vertices for $i \in \{1, 2\}$. Then, $K(r_{1}, s_{1}, k_{1}) \not\cong K(r_{2}, s_{2}, k_{2})~ \forall~ r_{1} \neq r_{2}, K(r_{1}, s_{1}, k_{1}) \not\cong K(r_{2}, s_{2}, k_{_{2}})~ \forall~ s_{1} \neq s_{2}, K(r_{1}, s_{1}, 0) \cong K(r_{1}, s_{1}, k_{1})$ for $2 \mid k_1$ or $2 \mid (k_1 + r_1)$, and $K(r_{1}, s_{1}, 1) \cong K(r_{1}, s_{1}, k_{1})$ for $2 \mid (k_1 - 1)$ or $2 \mid (k_1 + r_1 - 1)$.
\end{cor}

\noindent \textbf{2.3 Maps of type $\{6^3\}$~:~}
In Section 2.1, we have shown that the maps of type $\{3^6\}$ on the Klein bottle can be classified up to isomorphism. Let $P_1,$ \dots, $P_k$ denote maps of type $\{3^6\}$ on $n$ vertex. Let $M_i$ denote the dual of $P_i$. Let $m = \# F(P_i)$. By the notion of duality, then $M_1,$ \dots, $M_k$ are the only maps of type $\{6^{3}\}$ with $m$ vertices upto isomorphism. In \cite{mu:torus-hc13}, the details of the above claim has been discussed elaborately. Therefore, one can classify maps of type $\{6^3\}$  on the Klein bottle using the classification of maps of type $\{3^6\}$ on the Klein bottle. Thus, we have 

\begin{Lemma} \label{lem63.iso} The maps of type $\{6^3\}$ on the Klein bottle can be classified up to isomorphism using the classifcation of maps of type $\{3^6\}$ on the Klein bottle.
\end{Lemma}

\noindent \textbf{2.4 Maps of type $\{3^{3}, 4^{2}\}$~:~}
Let $M$ be a map of type $\{3^{3}, 4^{2}\}$ on the Klein bottle. We recall some definitions and results from Maity and Upadhyay \cite{mu:torus-hc13} which are defined on the torus. These are also well defined (clear from the definitions) in the map on the Klein bottle. The definitions are as follows. 
Let $P(\dots, u_{i-1}, u_{i}, u_{i+1}, \dots)$ be a path in edge graph of $M$. The path $P$ is of type $A_1$ \cite{mu:torus-hc13} if all the triangles incident with $u_i$ lie on one side and all quadrangles incident with $u_i$ lie on the other side of $P$ at $u_{i}$ for all $i$. 
Let $W$ be a maximal walk of type $A_1$ in $M$. From definition, it follows that there is only one path of type $A_1$ through each vertex of $M$. So, by the similar argument as in Lemma \ref{lem36:1}, the maximal walk $W$ of type $A_1$ is a cycle of type $A_1$. Let $C$ be a cycle of type $A_{1}$ in $M$. Let $S$ denote a set of faces which are incident with $C$. Then, we have 

\begin{Lemma}\label{lem3342.1} The geometric carrier $|S|$ is $S(C, M), S\mathcal{M}_\triangle(C, M)$ or $S\mathcal{M}_\Box(C, M)$.
\end{Lemma}

\begin{proof} We proceed with the similar argument as in Lemma \ref{lem36:2}. Let $lk(w_1) = C(x_1, w_r, \textit{\textbf{v}}_r,$ $v_1, \textit{\textbf{v}}_2, w_2, x_2)$. We consider $|S|$ and cut along the path $P(x_1, w_1, v_1)$. Hence, we have the following cases. If a path $P(x_1, w_1, v_1)$ identifies with $P(x_1, w_1, v_1)$ with out any twist then we get a cylinder (see an example in Figure 10, $C(w_1, \dots, w_r)$). We denote it by $S(C, M)$. If $P(x_1, w_1, v_1)$ identifies with $P(x_1, w_1, v_1)$ with a twist then we get a M\"{o}bius strip (see an example in Figure 12). In this case, we consider $lk(w_1)$ and which is of type $\{3^2, 4, 3, 4\}$. So, this case is not possible. Now by definition of $A_1$, the triangles which are incident with $C$ are lie on one side of $C$ and quadrangles are lie on the other side of $C$. Assume that the cycle $C =C(w_1, w_2, \dots, w_{2r+1})$ is of odd length. By the similar argument as in Lemma \ref{lem36:2}, the triangles which are incident with $C$ forms a M\"{o}bius strip $\mathcal{M}_\triangle(C, M)$. In this case, we get $S\mathcal{M}_\triangle(C, M)$ (see an example in Figure 12). Similarly as above, we proceed for even length cycle and in this case, we get $S\mathcal{M}_\Box(C, M)$ (see an example in Figure 13). Therefore, when length of $C$ is odd then the quadrangles can not repeat and when length of $C$ is even then the triangles can not repeat. So, the triangles and quadrangles can not repeat simultaneously. Therefore, by combining above all the cases, the geometric carrier $|S|$ is $S(C, M), S\mathcal{M}_\triangle(C, M)$ or $S\mathcal{M}_\Box(C, M)$.
\end{proof}

\begin{picture}(0,0)(10,21)
\tiny
\setlength{\unitlength}{1.6mm}
\drawpolygon(5,5)(20,5)(20,15)(5,15)
\drawpolygon(22,5)(37,5)(37,15)(22,15)


\drawline[AHnb=0](10,5)(10,15)
\drawline[AHnb=0](15,5)(15,15)
\drawline[AHnb=0](27,5)(27,15)
\drawline[AHnb=0](32,5)(32,15)
\drawline[AHnb=0](5,10)(20,10)
\drawline[AHnb=0](22,10)(37,10)

\drawline[AHnb=0](5,10)(10,15)
\drawline[AHnb=0](10,10)(15,15)
\drawline[AHnb=0](15,10)(20,15)
\drawline[AHnb=0](22,10)(27,15)
\drawline[AHnb=0](27,10)(32,15)
\drawline[AHnb=0](32,10)(37,15)

\put(20,10){\dots}
\put(20,5){\dots}
\put(20,15){\dots}

\put(5,4){\tiny ${v_{1}}$}
\put(10,4){\tiny ${v_{2}}$}
\put(15,4){\tiny ${v_{3}}$}
\put(20,4){\tiny ${v_{4}}$}
\put(22,4){\tiny ${v_{r-2}}$}
\put(27,4){\tiny ${v_{r-1}}$}
\put(32,4){\tiny ${v_{r}}$}
\put(37,4){\tiny ${v_{1}}$}

\put(5.5,9){\tiny ${w_{1}}$}
\put(10.5,9){\tiny ${w_{2}}$}
\put(15.5,9){\tiny ${w_{3}}$}
\put(20,9){\tiny ${w_{4}}$}
\put(22.5,9){\tiny ${w_{r-2}}$}
\put(27.5,9){\tiny ${w_{r-1}}$}
\put(32.5,9){\tiny ${w_{r}}$}
\put(37.5,9){\tiny ${w_{1}}$}

\put(5.5,14){\tiny ${x_{1}}$}
\put(10.5,14){\tiny ${x_{2}}$}
\put(15.5,14){\tiny ${x_{3}}$}
\put(20,14){\tiny ${x_{4}}$}
\put(22.5,14){\tiny ${x_{r-2}}$}
\put(27.5,14){\tiny ${x_{r-1}}$}
\put(32.5,14){\tiny ${x_{r}}$}
\put(37.5,14){\tiny ${x_{1}}$}

\put(11,0){\scriptsize Figure 10 : $S(C, M)$ : Cylinder}

\end{picture}

\vspace{3cm}

\begin{picture}(0,0)(10,15)
\tiny
\setlength{\unitlength}{1.6mm}
\drawpolygon(5,5)(20,5)(20,15)(5,15)
\drawpolygon(22,5)(37,5)(37,15)(22,15)


\drawline[AHnb=0](10,5)(10,15)
\drawline[AHnb=0](15,5)(15,15)
\drawline[AHnb=0](27,5)(27,15)
\drawline[AHnb=0](32,5)(32,15)
\drawline[AHnb=0](5,10)(20,10)
\drawline[AHnb=0](22,10)(37,10)

\drawline[AHnb=0](5,10)(10,15)
\drawline[AHnb=0](10,10)(15,15)
\drawline[AHnb=0](15,10)(20,15)
\drawline[AHnb=0](22,10)(27,15)
\drawline[AHnb=0](27,10)(32,15)
\drawline[AHnb=0](32,10)(37,15)

\put(20,10){\dots}
\put(20,5){\dots}
\put(20,15){\dots}

\put(5,4){\tiny ${v_{1}}$}
\put(10,4){\tiny ${v_{2}}$}
\put(15,4){\tiny ${v_{3}}$}
\put(20,4){\tiny ${v_{4}}$}
\put(22,4){\tiny ${v_{r-2}}$}
\put(27,4){\tiny ${v_{r-1}}$}
\put(32,4){\tiny ${v_{r}}$}
\put(37,4){\tiny ${x_{1}}$}

\put(5.5,9){\tiny ${w_{1}}$}
\put(10.5,9){\tiny ${w_{2}}$}
\put(15.5,9){\tiny ${w_{3}}$}
\put(20,9){\tiny ${w_{4}}$}
\put(22.5,9){\tiny ${w_{r-2}}$}
\put(27.5,9){\tiny ${w_{r-1}}$}
\put(32.5,9){\tiny ${w_{r}}$}
\put(37.5,9){\tiny ${w_{1}}$}

\put(5.5,14){\tiny ${x_{1}}$}
\put(10.5,14){\tiny ${x_{2}}$}
\put(15.5,14){\tiny ${x_{3}}$}
\put(20,14){\tiny ${x_{4}}$}
\put(22.5,14){\tiny ${x_{r-2}}$}
\put(27.5,14){\tiny ${x_{r-1}}$}
\put(32.5,14){\tiny ${x_{r}}$}
\put(37.5,14){\tiny ${v_{1}}$}

\put(16,0){\scriptsize Figure 11}

\end{picture}

\vspace{3 cm}

\begin{picture}(0,0)(-50,-47)
\tiny
\setlength{\unitlength}{1.6mm}
\drawpolygon(5,5)(20,5)(20,15)(5,15)
\drawpolygon(32,5)(37,5)(37,10)(32,10)
\drawpolygon(39,5)(54,5)(54,10)(39,10)


\drawpolygon(5,5)(20,5)(20,15)(5,15)
\drawpolygon(22,10)(32,10)(37,15)(22,15)
\drawpolygon(22,5)(32,5)(32,10)(22,10)

\drawline[AHnb=0](10,5)(10,15)
\drawline[AHnb=0](15,5)(15,15)
\drawline[AHnb=0](27,5)(27,15)
\drawline[AHnb=0](32,5)(32,15)
\drawline[AHnb=0](5,10)(20,10)
\drawline[AHnb=0](22,10)(37,10)
\drawline[AHnb=0](49,5)(49,10)
\drawline[AHnb=0](44,5)(44,10)
\drawline[AHnb=0](39,10)(54,10)

\put(20,10){\dots}
\put(20,5){\dots}
\put(20,15){\dots}

\drawline[AHnb=0](5,10)(10,15)
\drawline[AHnb=0](10,10)(15,15)
\drawline[AHnb=0](15,10)(20,15)

\drawline[AHnb=0](22,10)(27,15)
\drawline[AHnb=0](27,10)(32,15)
\drawline[AHnb=0](32,10)(37,15)



\put(37,9.8){\dots}
\put(37,5){\dots}

\put(5,4){\tiny ${v_{1}}$}
\put(10,4){\tiny ${v_{2}}$}
\put(15,4){\tiny ${v_{3}}$}
\put(20,4){\tiny ${v_{4}}$}
\put(22,4){\tiny ${v_{r-1}}$}
\put(27,4){\tiny ${v_{r}}$}
\put(32,4){\tiny ${v_{r+1}}$}
\put(35,4){\tiny ${v_{r+2}}$}
\put(39,4){\tiny ${v_{2r-1}}$}
\put(43.5,4){\tiny ${v_{2r}}$}
\put(48.5,4){\tiny ${v_{2r+1}}$}
\put(53.5,4){\tiny ${v_{1}}$}

\put(5.5,9){\tiny ${w_{1}}$}
\put(10.5,9){\tiny ${w_{2}}$}
\put(15.5,9){\tiny ${w_{3}}$}
\put(20,9){\tiny ${w_{4}}$}
\put(22.5,9){\tiny ${w_{r-1}}$}
\put(27.5,9){\tiny ${w_{r}}$}
\put(32.5,9){\tiny ${w_{r+1}}$}
\put(35,10.5){\tiny ${w_{r+2}}$}
\put(39.2,10.5){\tiny ${w_{2r-1}}$}
\put(43.5,10.5){\tiny ${w_{2r}}$}
\put(47,10.5){\tiny ${w_{2r+1}}$}
\put(53,10.5){\tiny ${w_{1}}$}

\put(5,16){\tiny ${w_{r+1}}$}
\put(10,16){\tiny ${w_{r+2}}$}
\put(15,16){\tiny ${w_{r+3}}$}
\put(18.7,16){\tiny ${w_{r+4}}$}
\put(22,16){\tiny ${w_{2r-1}}$}
\put(27,16){\tiny ${w_{2r}}$}
\put(31,16){\tiny ${w_{2r+1}}$}
\put(36,16){\tiny ${w_{1}}$}

\put(20,0){\scriptsize Figure 12 : $S\mathcal{M}_\triangle(C, M)$}

\end{picture}


\begin{picture}(0,0)(-50,-20)
\tiny
\setlength{\unitlength}{1.6mm}
\drawpolygon(5,5)(20,5)(20,15)(5,15)
\drawpolygon(22,5)(32,5)(32,15)(22,15)
\drawpolygon(32,10)(37,15)(32,15)


\drawpolygon(5,5)(20,5)(20,15)(5,15)
\drawpolygon(22,10)(37,10)(37,15)(22,15)
\drawpolygon(22,5)(32,5)(32,10)(22,10)
\drawpolygon(39,10)(54,10)(54,15)(39,15)

\drawline[AHnb=0](10,5)(10,15)
\drawline[AHnb=0](15,5)(15,15)
\drawline[AHnb=0](27,5)(27,15)
\drawline[AHnb=0](32,5)(32,15)
\drawline[AHnb=0](5,10)(20,10)
\drawline[AHnb=0](22,10)(37,10)
\drawline[AHnb=0](49,10)(49,15)
\drawline[AHnb=0](44,10)(44,15)
\drawline[AHnb=0](39,10)(54,10)

\put(20,10){\dots}
\put(20,5){\dots}
\put(20,15){\dots}

\drawline[AHnb=0](5,10)(10,15)
\drawline[AHnb=0](10,10)(15,15)
\drawline[AHnb=0](15,10)(20,15)
\drawline[AHnb=0](22,10)(27,15)
\drawline[AHnb=0](27,10)(32,15)
\drawline[AHnb=0](32,10)(37,15)
\drawline[AHnb=0](39,10)(44,15)
\drawline[AHnb=0](44,10)(49,15)
\drawline[AHnb=0](49,10)(54,15)

\put(37,10){\dots}
\put(37,15){\dots}

\put(5,4){\tiny ${w_{r+1}}$}
\put(10,4){\tiny ${w_{r+2}}$}
\put(15,4){\tiny ${w_{r+3}}$}
\put(18.5,4){\tiny ${w_{r+4}}$}
\put(22,4){\tiny ${w_{3}}$}
\put(27,4){\tiny ${w_{2}}$}
\put(32,4){\tiny ${w_{1}}$}

\put(5.5,9){\tiny ${w_{1}}$}
\put(10.5,9){\tiny ${w_{2}}$}
\put(15.5,9){\tiny ${w_{3}}$}
\put(20,9){\tiny ${w_{4}}$}
\put(22.5,9){\tiny ${w_{r-1}}$}
\put(27.5,9){\tiny ${w_{r}}$}
\put(32.5,9){\tiny ${w_{r+1}}$}
\put(36,9){\tiny ${w_{r+2}}$}
\put(39.5,9){\tiny ${w_{2r-2}}$}
\put(44.5,9){\tiny ${w_{2r-1}}$}
\put(49.5,9){\tiny ${w_{2r}}$}
\put(54,9){\tiny ${w_{1}}$}

\put(5,16){\tiny ${x_{1}}$}
\put(10,16){\tiny ${x_{2}}$}
\put(15,16){\tiny ${x_{3}}$}
\put(20,16){\tiny ${x_{4}}$}
\put(22,16){\tiny ${x_{r-1}}$}
\put(27,16){\tiny ${x_{r}}$}
\put(32,16){\tiny ${x_{r+1}}$}
\put(35.8,16){\tiny ${x_{r+2}}$}
\put(39,16){\tiny ${x_{2r-2}}$}
\put(44,16){\tiny ${x_{2r-1}}$}
\put(49,16){\tiny ${x_{2r}}$}
\put(53.5,16){\tiny ${x_{1}}$}

\put(24,0){\scriptsize Figure 13 : $S\mathcal{M}_\Box(C, M)$}

\end{picture}

\vspace{-1.8cm} 

By Lemma \ref{lem3342.1}, the map $M$ contains a $S(C, M), S\mathcal{M}_\triangle(C, M)$ or a $S\mathcal{M}_\Box(C, M)$. For each of $S(C, M), S\mathcal{M}_\triangle(C, M)$ and $S\mathcal{M}_\Box(C, M)$, we proceed with the argument as in Section 2.1 and hence, we get that the map $M$ has either a $S(M)$ or a $S\mathcal{M}(M)$ representation.  

We first assume that $S(M)$ denotes a representation of $M$. Then, $\partial S(M) = \{C, C'\}$ where $C, C'$ are two cycles of type $A_1$. We recall definition of $(r, s, k)$-representation from \cite{mu:torus-hc13} which is defined on the torus. The similar definition is also introduced in Section 2.1. Let $v \in V(C)$. We recall the following two paths in $M$ as follows. 
Let $P(\dots, v_{i-1}, v_{i}, v_{i+1}, \dots)$ be a path in edge graph of $M$. The path $P$ is of type $A_2$ \cite{mu:torus-hc13} at vertex $v_{i}$ if $lk(v_{i})=C(\textit{\textbf{a}}, v_{i-1}, \textit{\textbf{b}}, c, v_{i+1}, d, e)$ implies $lk(v_{i+1})=C(\textit{\textbf{a}}_0, v_{i+2}, \textit{\textbf{b}}_0, d, v_{i}, c, p)$ and $lk(v_{i}) = C(\textit{\textbf{x}}, v_{i+1},$ $ \textit{\textbf{z}}, l, v_{i-1}, k, m)$ implies $lk(v_{i+1}) = C(\textit{\textbf{l}},v_{i}, \textit{\textbf{m}}, x, v_{i+2}, g, z)$ for all $i$.
Similarly, let $P(\dots, w_{i-1},$ $ w_{i}, w_{i+1}, \dots)$ be a path in edge graph of $M$. The path $P$ is of type $A_3$ at vertex $w_{i}$ \cite{mu:torus-hc13} if $lk(w_{i})=C(\textit{\textbf{a}}, w_{i-1}, \textit{\textbf{b}}, c, d, w_{i+1}, e)$ implies $lk(w_{i+1})=C(\textit{\textbf{a}}_1, w_{i+2}, \textit{\textbf{b}}_1, p, e, w_{i}, d)$ and $lk(w_{i})=C(\textit{\textbf{a}}_2, w_{i+1}, \textit{\textbf{b}}_2, p, e, w_{i-1}, d)$ implies $lk(w_{i+1})=C(\textit{\textbf{p}},$ $w_{i},  \textit{\textbf{d}}, a_2, z_1, w_{i+2}, b_2)$ for all $i$.
We have two paths through $v$ of types $A_2$ and $A_3$ (by definition). Let $L_1 = C, L_{2} = P(a_1 = v, \dots, w), L_{3} = P(a_1 = v, \dots, w')$ denote two paths of types $A_{2}$ and $A_{3}$ respectively through $v$ where $w, w' \in V(C')$. Let $C :=C(a_1, a_2, \dots, a_r)$. Similarly as in Section 2.1, we say that $C$ is a horizontal base cycle and a path is vertical if it is not homologous to $C$. Let faces incident with $C$ be quadrangles and  faces incident with $C'$ be triangles in $S(M)$. Then, we make an another cut along $L_{3}$. Let $s$ denote the number of cycles which are homologous to $C$ along the path $L_{3}$. Let length($L_{3}$) = $m$. Here, one can observe thjat the $s$ is equal to the length of the path $L_{3}$, that is, $s = m$. Now length($C$) = $r$ and number of horizontal cycles along $L_{3}$ is $s$. Thus, we get a planar polyhedral representation of $S(M)$ and it is denoted by {\em $(r,  s)$-representation}. In the representation, we get identification of vertical sides in the natural manner but the identification of the horizontal sides needs some shifting with twist so that a vertex in the lower(base) side is identified with a vertex in the upper side. Let $C(a_{k+1}(=w_m), a_{k}, a_{k-1}, \dots, a_k)$ denote the upper horizontal cycle in $(r,  s)$-representation. So, the vertex $a_{k+1}$ is the starting vertex of the upper horizontal cycle. Denote $k =$ length$(P(a_{k+1} (= w_m), a_k, a_{k-1}, \dots, a_k))$ in $C$. Hence, we represent the $(r, s)$-representation by {\em $(r, s, k)$-representation} and denote it by $K(r, s, k)$.

Assume that $S\mathcal{M}(M)$ denotes a representation of $M$. We cut along the boundaries of the M\"{o}bius strips as in Section 2.1. Hence, we get three components. These are a cylinder and two M\"{o}bius strips. We denote the cylinder by $S'(M)$ and two M\"{o}bius strips by $\mathcal{M}', \mathcal{M}''$ where $\mathcal{M}', \mathcal{M}'' \in \{\mathcal{M}_\triangle(C, M), \mathcal{M}_\Box(L, M) ~|~$ for some cycles $C$ and $L$ of type $A_1 \}$. Thus, in this case, we denote the representation $S\mathcal{M}(M)$ by $K(S'(M), \mathcal{M}', \mathcal{M}'')$. Let $l$ denote the length of a cycle of type $A_1$ in $S'(M)$. Then, the cycles $\partial\mathcal{M}', \partial\mathcal{M}''$ have same length $l$ as $\partial S'(M) = \{\partial\mathcal{M}', \partial\mathcal{M}''\}$. Therefore, we get a $K(l, S'(M), \mathcal{M}', \mathcal{M}'')$ representation of $K(S'(M), \mathcal{M}', \mathcal{M}'')$. We combine above all arguments and hence, we have the following result.     

\begin{Lemma} \label{lem3342:2} The map $M$ has a $K(r, s, k)$ or a $K(l, S'(M), \mathcal{M}', \mathcal{M}'')$ representation.
\end{Lemma}

We show in Lemma \ref{lem3342:3} that the $M$ does not have both the representations. That is, we have 

\begin{Lemma}\label{lem3342:3}
Let $M_1$ and $M_2$ be two maps on same number of vertices. Let $K_1(r, s, k)$ denote $M_1$ and $K_2(S'(M_2), \mathcal{M}'_2, \mathcal{M}''_2)$ denote $M_2$. Then, $M_1 \not\cong M_2$.  
\end{Lemma}

\begin{proof} The map $M_2$ contains a M\"{o}bius strip $\mathcal{M}'_2$ which consists of either triangles or quadrangles. Let $C = \partial \mathcal{M}'_2$. Since $M_1$ does not contain  M\"{o}bius strip which is bounded by any cycle of type $A_1$, so, there is no isomorphism map which maps $C$ to a cycle of type $A_1$ in $M_1$. Therefore, $M_1 \not\cong M_2$. 
\end{proof}
 
From definition of $A_1$, we have one cycle of type $A_1$ through each vertex in $M$. It is discussed in \cite{mu:enu-torus} that the cycles of type $A_{1}$ in $M$ have same length. Thus we have 

\begin{prop}\label{prop3342:1} (\cite{mu:enu-torus}) The cycles of type $A_{1}$ in $M$ have unique length.
\end{prop}

We study $K(r, s, k)$ and $K(S(M), \mathcal{M}', \mathcal{M}'')$, and classify them separately as they are non-isomorphic. 

\textbf{Classification of $K(r, s, k)$ on $n$ vertices :} We define admissible relations among $r, s, k$ of $K(r, s, k)$ such that $K(r, s, k)$ represents a map after identifying its boundaries. The proof of the bellow results similarly as in Section 2.1. So, we have 

\begin{Lemma}\label{lem3342:4} The maps of type $\{3^{3}, 4^{2}\}$ of the form $K(r, s, k)$ exist if and only if the following holds : (i) rs $\geq$ 12, (ii) r $\geq$ 3, (iii) s $\ge$ 4, and (iv) $k \in \{ t : 0 \leq t \leq r-1\}$ if $s \ge 4$. 
\end{Lemma}

\begin{Lemma}\label{lem3342:iso-1}
Let $K(r_{i}, s_{i}, k_{i})$ denote a $(r_{i}, s_{i}, k_{i})$-representation of $M_{i}$ on same number of vertices for $i \in \{1, 2\}$. Then,\ $M_{1} \cong M_{2}$ if $(r_1, s_1, l_1) = (r_2, s_2, l_2)$ where $ l_i \in \{k_i~ mod(2), (k_i + r_i)~ mod(2)\}$.
\end{Lemma}

\begin{proof} The proof repeats similar argument as in Lemma \ref{lem36:iso}. Let $r_1 = r_2, s_1 = s_2, k_1 =k_2$. Define an isomorphism map $f_{1} : V(K(r_{1}, s_{1}, k_{1}))\rightarrow V(K(r_{2}, s_{2}, k_{2}))$. So, $M_{1} \cong M_{2}$ by $f_1$.    Let $k_{1} \neq k_{2}$. Repeat similar argument as in Lemma \ref{lem36:iso}. Hence, $K(r_i, s_i, k_i)$ has either a $K(r_i, s_i, 0)$ or a $K(r_i, s_i, 1)$ representation. Hence, by $f_1, M_{1} \cong M_{2}$. This completes the proof.
\end{proof}

\begin{cor} \label{cor3342:1} Let $K(r_{i}, s_{i}, k_{i})$ denote a $(r_{i}, s_{i}, k_{i})$-representation of $M_{i}$ on same number of vertices for $i \in \{1, 2\}$. Then, $K(r_{1}, s_{1}, k_{1}) \not\cong K(r_{2}, s_{2}, k_{2}) ~\forall~ r_{1} \neq r_{2}, K(r_{1}, s_{1}, k_{1}) \not\cong K(r_{2}, s_{2}, k_{_{2}})~ \forall~ s_{1} \neq s_{2}, K(r_{i}, s_{i}, 0) \cong K(r_{i}, s_{i}, k_{i})$ for $2 \mid k_i$ or $2 \mid (k_i + r_i)$, and $K(r_{i}, s_{i},  1) \cong K(r_{i}, s_{i}, k_{i})$ for $2 \mid (k_i - 1)$ or $2 \mid (k_i + r_i - 1)$.
\end{cor}

\textbf{Classification of $K(l, t)$ on $n$ vertices :} Let $S\mathcal{M}(M)$ denote $M$ on $n$ vertices. By Lemma \ref{lem3342:2}, the $S\mathcal{M}(M)$ has a $K(l, S'(M), \mathcal{M}', \mathcal{M}'')$ representation.  Let $C' = \partial\mathcal{M}', C''=\partial\mathcal{M}''$. Again, similarly as in the Proposition \ref{prop3342:1}, length($C'$) = length($C''$). Let $C' = \partial\mathcal{M}', C''=\partial\mathcal{M}''$ such that $\partial S'(M) = \{C', C''\}$ for some cycles $C', C''$ of type $A_1$. Hence, we denote $K(l, S'(M), \mathcal{M}', \mathcal{M}'')$ by $K(l, t)$ if $l = length(C') = length(C'')$ and $t$ denote the number of cycles of type $A_1$ in $S'(M)$. 

\begin{Lemma}\label{lem3342:5} The maps of the type $\{3^{3}, 4^{2}\}$ of the form $K(l, t)$ exist if and only if the following holds : (i) $tl \geq 10$, (ii) $t \ge 2$ and $2 \mid t$, (iii) $l \ge 5$ if $t=2$, (iv) $l \ge 4$ if $t \ge 4$, (v) $\mathcal{M}'$ and $\mathcal{M}''$ consists of only triangles if $l$ is odd, (vi) $\mathcal{M}'$ and $\mathcal{M}''$ consists of only quadrangles if $l$ is even.
\end{Lemma}

\begin{proof} Let $K(l, S'(M), \mathcal{M}', \mathcal{M}'')$ denote a representation of $M$. It has a decompositions into a cylinder $S'(M)$ and two M\"{o}bius strips $\mathcal{M}'$ and $\mathcal{M}''$. By the similar argument as in Lemma \ref{lem3342.1}, the cylinder $S'(M)$ contains two disjoint boundaries. So, $t \ge 2$. We repeat similar argument as in Lemma \ref{lem36.7} and hence, we get $n = tl, 2 \mid t, l \mid n$ and $t \mid n$.  Similarly we repeat similar arguments as in Lemma \ref{lem3342:4} for Case $3$ and $4$. By Proposition \ref{prop3342:1}, the cycles of type $A_1$ have same length. Observe that by definition, when the length of the cycle is odd then by Lemma \ref{lem3342.1}, $\mathcal{M}'$ and $\mathcal{M}''$ consist of only triangles. When the length of the cycle is even, then by Lemma \ref{lem3342.1}, $\mathcal{M}'$ and $\mathcal{M}''$ consist of only quadrangles. This completes the proof.   
\end{proof}

Let $M_1$ and $M_2$ be two maps on same number of vertices. Let $K(l_i, t_i)$ denote a representation of $M_i$. Then, we have

\begin{Lemma}\label{lem3342:6}
   The $M_1 \cong M_2$ if $l_1 = l_2$.  
\end{Lemma}

\begin{proof} By Proposition \ref{prop3342:1}, the $M_i$ contains cycles of type $A_1$ of same length. We denote it by $l$. As in Lemma \ref{lem3342:iso-1}, let $n = sl$. So, the number of cycles of type $A_1$ in $M_i$ is $s$. Let $C_1$, \dots, $C_s$ denote the cycles of type $A_1$ in $M_1$. Let $C_1 = \partial \mathcal{M}'_1$ and $C_s = \partial \mathcal{M}''_1$. Again, let $L_1, \dots, L_s$ denote the cycles of type $A_1$ in $M_2$. Let $L_1 = \partial \mathcal{M}'_2$ and $L_s = \partial \mathcal{M}''_2$. Since $l_1 = l_2$, so, the $\mathcal{M}'_1, \mathcal{M}''_1, \mathcal{M}'_2$ and $\mathcal{M}''_2$ consist of either triangles or quadrangles. Therefore, we define an isomorphism map $f_1 ~:~ V(M_1) \rightarrow V(M_2)$ by $f_1(V(C_i)) = V(L_i)$ for $1 \le i \le s$ as in Lemma \ref{lem3342:iso-1}. Thus, $M_1 \cong M_2$. This completes the proof.
\end{proof}

Thus, we have

\begin{cor} \label{cor3342:2} Let $K(r_{i}, s_{i}, k_{i})$ denote a $(r_{i}, s_{i}, k_{i})$-representation of $M_{i}$ on same number of vertices for $i \in \{1, 2\}$. Then,\  $K(l_{1}, t_{1}) \not\cong K(l_{2}, t_{2}) ~\forall~ l_{1} \neq l_{2}, K(l_{1}, t_{1}) \not\cong K(l_{2}, t_{2}) ~\forall~ t_{1} \neq t_{2}, M_i$ has an unique representation $K(l_{i}, t_{i})$ for $2 \mid l_i$ and $n = l_it_i$, and $M_i$ has an unique representation $K(l_{i}, t_{i})$ for $2 \mid (l_i+1)$ and $n = l_it_i$.
\end{cor}

\noindent \textbf{2.5 Maps of type $\{3^{2}, 4, 3, 4\}$~:~}
Let $M$ be a map of type $\{3^{2}, 4, 3, 4\}$ on the Klein bottle. Let $P(\dots, u_{i-1}, u_{i}, u_{i+1}, \dots)$ be a path in edge graph of $M$. The path $P$ is of type $B_1$ \cite{mu:torus-hc13} if $lk(u_{i})=C(\textit{\textbf{a}}, u_{i+1}, b, c, \textit{\textbf{d}}, u_{i-1}, e)$ implies $lk(u_{i-1})=C(\textit{\textbf{f}}, g, e, u_{i}, \textit{\textbf{c}}, d, u_{i-2})$ and $lk(u_{i+1})=C(\textit{\textbf{e}}, a, k, u_{i+2}, \textit{\textbf{l}}, b, u_{i})$, and  $lk(u_{i})=C(\textit{\textbf{e}}, h, k, u_{i+1}, \textit{\textbf{l}}, b, u_{i-1})$ implies $lk(u_{i-1})=C(\textit{\textbf{h}}, u_{i}, b, c, \textit{\textbf{d}}, u_{i-2}, e)$ and $lk(u_{i+1})=C(\textit{\textbf{s}}, u_{i+2}, t, l, \textit{\textbf{b}}, u_{i}, k)$ at $u_i$ for all $i$.
Let $W$ be a maximal walk of type $B_1$ in $M$. Similarly as in Section 2.1, it is either a cycle or a closed walk. Observe that by definition, we have two paths of type $B_1$ through each vertex in $M$. Therefore, by the similar argument as in Section 2.2, the map $M$ contains a cycle of type $B_1$ thorough each vertex of $M$. Let $C$ be a cycle of type $B_{1}$ in $M$. Let $S$ denote a set of faces which are incident with $C$. Then, we have

\begin{Lemma}\label{lem32434:1} The $|S|$ is $S(C, M)$ or $\mathcal{M}(C, M)$.
\end{Lemma}

\begin{proof} We first assume that the faces incident with $C$ do not repeat. We repeat the similar argument as we have done in above preceding sections. Let $lk(w_1) = C(x_1, x_2, w_2, \textit{\textbf{v}}_2, v_1, w_r,$ $\textit{\textbf{x}}_r)$. We consider geometric carrier $|S|$ and cut along the path $P(x_1, w_1, v_1)$ of type $B_1$. Thus we have the following cases. Let path $P(x_1, w_1, v_1)$ identify with $P(x_1, w_1, v_1)$ with out any twist. In this case, we get a cylinder (for an example see in Figure 14). We denote it by $S(C, M)$. Let $P(x_1, w_1, v_1)$ identify with $P(x_1, w_1, v_1)$ with a twist. In this case, we get a M\"{o}bius strip (for an example see in Figure 15). We denote it by $\mathcal{M}(C, M)$. 

Assume that the faces incident with $C$ repeat. Let $S'$ denote the set of faces incident with $C$ and lie on one side of $C$. The set $S'$ consists of both triangles and quadrangles. Similarly as in Lemma \ref{lem36:2}, each face of $S'$ repeats exactly twice in the sequence of incident faces of $C$ which are lie one side and belong in $S'$. So, let $F_1, F_2, \dots, F_k, F_1, F_2, \dots, F_k$ denote a sequence of faces which are incident with $C$ for some $k$ and belong to $S'$. That is $\{F_1, F_2, \dots, F_k, F_1, F_2, \dots, F_k\} = \{F_1, F_2, \dots, F_k\}$. Thus we have the following two cases. Let the cycle $C = C(w_1, \dots, w_{2r})$ be of even length. In this case, we consider link of $w_{r+2}$ and hence, we get a face sequence either $\{3, 4, 3, 4\}$ or $\{3^2, 4, 3^2, 4\}$ (for example see in Figure $14$ and $15$). These are different from $\{3^2, 4, 3, 4\}$. So, length$(C)$ is not even. Assume that the length of $C$ is odd. In this case, the face sequence incident with $C$ is either $\triangle'_1, \triangle'_1, \Box_2,$ \dots, $\Box_{r-1}, \triangle'_{r}, \triangle''_{r}$ or $\Box_1, \triangle'_2, \triangle''_2,$ \dots,$\triangle'_{r-1}, \triangle''_{r-1}, \Box_{r}$ ($\triangle', \Box$ represent triangles and quadrangles respectively). In  $\triangle'_1, \triangle''_1, \Box_2,$ \dots, $\Box_{\frac{r-1}{2}}, \triangle'_{\frac{r+1}{2}}, \triangle''_{\frac{r+1}{2}}, \Box_{\frac{r+3}{2}},$ \dots, $\Box_{r-1}, \triangle'_{r}, \triangle''_{r}$ implies $\triangle'_1 = \triangle''_{\frac{r+1}{2}}$ and $\triangle''_1  = \Box_{\frac{r+3}{2}}$ as $(\triangle'_1, \triangle'_1, \Box_2,$ \dots, $\Box_{\frac{r-1}{2}}, \triangle'_{\frac{r+1}{2}}) = ( \triangle''_{\frac{r+1}{2}}, \Box_{\frac{r+3}{2}},$ \dots, $\Box_{r-1}, \triangle'_{r}, \triangle''_{r})$. But $\triangle''_1  \not= \Box_{\frac{r+3}{2}}$. So, this case is not possible. Similarly as above, the sequence $\Box_1, \triangle'_2, \triangle''_2,$ \dots,$\triangle'_{r-1}, \triangle''_{r-1}, \Box_{r}$ does not exist. Therefore,  $|S| = S(C, M)$ (for example see in Figure 14) or $|S| =\mathcal{M}(C, M)$ (for example see in Figure 15).
\end{proof}

So, by Lemma \ref{lem32434:1}, $M$ contains either of a cylinder $S(C, M)$ and a M\"{o}bius strip $\mathcal{M}(C, M)$. If $S(C, M) \subset M$ or $\mathcal{M}(C, M) \subset M$ then by the similar argument as in Section 2.1, the map $M$ has either a $S(M)$ or a $S\mathcal{M}(M)$ representation.   Let $S(M)$ denote $M$ such that $\partial S(M) = \{C, C'\}$ for some cycles $C, C'$ of type $B_1$. Let $v \in V(C)$. By definition, let $L = P(v, \dots, w)$ denote the path of type $B_{1}$ through $v$ where $w \in V(C')$. We repeat the similar argument as in Section 2.4 which is done to define $(r, s, k)$-representation of $M$. Then, we take second cut along $L$ where the starting adjacent face to the base horizontal cycle is a $4$-gon. Thus, we get a representation of $M$ and it is denoted by $K(r, s, k)$. 

Again, let $S\mathcal{M}(M)$ denote $M$. Similarly as in the above preceding sections, we cut along the boundaries of the M\"obius strips and hence, we get three components namely a cylinder $S'(M)$ and two M\"obius strips $\mathcal{M}', \mathcal{M}''$ where $\mathcal{M}', \mathcal{M}'' \in \{\mathcal{M}(L, M) \mid$ for some cycle $L$ of type $B_1$ in $S\mathcal{M}(M) \}$. We denote this representation by $K(S'(M),  \mathcal{M}', \mathcal{M}'')$. By combining above all cases, the map $M$ has either a $K(r, s, k)$ or a $K(S'(M), \mathcal{M}', \mathcal{M}'')$ representation. We show in next lemma that every map $M$ has a $K(r, s, k)$ representation.

\begin{Lemma} \label{lem32434:2} The map $M$ has a $K(r, s, k)$ representation. 
\end{Lemma}

\begin{proof} As above the map $M$ has either a $K(r, s, k)$ or a $K(S'(M), \mathcal{M}', \mathcal{M}'')$ representation. Assume that $K(S'(M), \mathcal{M}', \mathcal{M}'')$ denotes $M$. We repeat similar process as in Lemma \ref{lem44:1}. We cut $\mathcal{M}(L, M)$ along a path $P$ of type $B_1$, and hence, we get a planar polyhedral representation. Let denoted it by $\Box(\mathcal{M})$. Similarly, we cut $S'(M)$ along two paths of type $B_1$. Hence, we get two planar polyhedral representations $\Box'$ and $\Box''$. And also, we cut $\mathcal{M}(L', M)$ along a path type $B_1$. Hence, we get an another polyhedral representation $\Box(\mathcal{M'})$. In this process, we consider paths which are not homologous to the boundary cycles of the cylinder and M\"obius strips through some fixed vertices. Now, similarly as in Lemma \ref{lem44:1}, we consider $\Box(\mathcal{M})$ and $\Box'$, and identify these along their common boundaries. Then, we consider $\Box(\mathcal{M})$ and $\Box''$, and identify these along their common boundaries. Thereafter, we consider $\Box''$ and $\Box(\mathcal{M'})$, and identify these along their common boundaries. Thus, we get a planar polyhedral representation $K(r', s', k')$ from $K(S'(M), \mathcal{M}', \mathcal{M}'')$ for some $r', s', k'$. Therefore, every map $M$ has a $K(r, s, k)$ representation for some $r, s, k$.
\end{proof}

\begin{picture}(0,0)(10,25)
\tiny
\setlength{\unitlength}{1.6mm}
\drawpolygon(5,5)(20,5)(20,15)(5,15)
\drawpolygon(22,5)(37,5)(37,15)(22,15)


\drawline[AHnb=0](10,5)(10,15)
\drawline[AHnb=0](15,5)(15,15)
\drawline[AHnb=0](27,5)(27,15)
\drawline[AHnb=0](32,5)(32,15)
\drawline[AHnb=0](5,10)(20,10)
\drawline[AHnb=0](22,10)(37,10)

\drawline[AHnb=0](5,10)(10,15)
\drawline[AHnb=0](10,10)(15,5)
\drawline[AHnb=0](15,10)(20,15)
\drawline[AHnb=0](22,10)(27,5)
\drawline[AHnb=0](27,10)(32,15)
\drawline[AHnb=0](32,10)(37,5)

\put(20,10){\dots}
\put(20,5){\dots}
\put(20,15){\dots}

\put(5,4){\tiny ${v_{1}}$}
\put(10,4){\tiny ${v_{2}}$}
\put(15,4){\tiny ${v_{3}}$}
\put(20,4){\tiny ${v_{4}}$}
\put(22,4){\tiny ${v_{r-2}}$}
\put(27,4){\tiny ${v_{r-1}}$}
\put(32,4){\tiny ${v_{r}}$}
\put(37,4){\tiny ${v_{1}}$}

\put(5.5,9){\tiny ${w_{1}}$}
\put(10.5,9){\tiny ${w_{2}}$}
\put(15.5,9){\tiny ${w_{3}}$}
\put(20,9){\tiny ${w_{4}}$}
\put(22.5,9){\tiny ${w_{r-2}}$}
\put(27.5,9){\tiny ${w_{r-1}}$}
\put(32.7,9){\tiny ${w_{r}}$}
\put(37.5,9){\tiny ${w_{1}}$}

\put(5.5,14){\tiny ${x_{1}}$}
\put(10.5,14){\tiny ${x_{2}}$}
\put(15.5,14){\tiny ${x_{3}}$}
\put(20,14){\tiny ${x_{4}}$}
\put(22.5,14){\tiny ${x_{r-2}}$}
\put(27.5,14){\tiny ${x_{r-1}}$}
\put(32.5,14){\tiny ${x_{r}}$}
\put(37.5,14){\tiny ${x_{1}}$}

\put(14,0){\tiny Figure 14 : $S(C, M)$ : Cylinder}

\end{picture}

\vspace{3cm}

\begin{picture}(0,0)(14,22)
\tiny
\setlength{\unitlength}{1.8mm}
\drawpolygon(5,5)(20,5)(20,15)(5,15)
\drawpolygon(22,5)(32,5)(32,15)(22,15)


\drawline[AHnb=0](10,5)(10,15)
\drawline[AHnb=0](15,5)(15,15)
\drawline[AHnb=0](27,5)(27,15)
\drawline[AHnb=0](32,5)(32,15)
\drawline[AHnb=0](5,10)(20,10)
\drawline[AHnb=0](22,10)(32,10)

\drawline[AHnb=0](5,10)(10,15)
\drawline[AHnb=0](10,10)(15,5)
\drawline[AHnb=0](15,10)(20,15)
\drawline[AHnb=0](22,10)(27,5)
\drawline[AHnb=0](27,10)(32,15)

\put(20,10){\dots}
\put(20,5){\dots}
\put(20,15){\dots}

\put(5,4){\tiny ${v_{1}}$}
\put(10,4){\tiny ${v_{2}}$}
\put(15,4){\tiny ${v_{3}}$}
\put(20,4){\tiny ${v_{4}}$}
\put(22,4){\tiny ${v_{r-2}}$}
\put(27,4){\tiny ${v_{r-1}}$}
\put(32,4){\tiny ${x_{1}}$}

\put(5.5,9){\tiny ${w_{1}}$}
\put(10.5,9){\tiny ${w_{2}}$}
\put(15.5,9){\tiny ${w_{3}}$}
\put(20,9){\tiny ${w_{4}}$}
\put(22.5,9){\tiny ${w_{r-2}}$}
\put(27.5,9){\tiny ${w_{r-1}}$}
\put(32.5,9){\tiny ${w_{1}}$}

\put(5.5,14){\tiny ${x_{1}}$}
\put(10.5,14){\tiny ${x_{2}}$}
\put(15.5,14){\tiny ${x_{3}}$}
\put(20.5,14){\tiny ${x_{4}}$}
\put(22.5,14){\tiny ${x_{r-2}}$}
\put(27.5,14){\tiny ${x_{r-1}}$}
\put(32.5,14){\tiny ${v_{1}}$}

\put(16,0){\tiny Figure 15}

\end{picture}

\vspace{3cm}

\begin{picture}(0,0)(-47,-44)
\tiny
\setlength{\unitlength}{1.7mm}
\drawpolygon(5,5)(20,5)(20,15)(5,15)
\drawpolygon(22,5)(37,5)(37,15)(22,15)
\drawpolygon(39,5)(54,5)(54,15)(39,15)

\drawline[AHnb=0](10,5)(10,15)
\drawline[AHnb=0](15,5)(15,15)
\drawline[AHnb=0](27,5)(27,15)
\drawline[AHnb=0](32,5)(32,15)
\drawline[AHnb=0](5,10)(20,10)
\drawline[AHnb=0](22,10)(37,10)
\drawline[AHnb=0](49,5)(49,15)
\drawline[AHnb=0](44,5)(44,15)
\drawline[AHnb=0](39,10)(54,10)

\put(20,10){\dots}
\put(20,5){\dots}
\put(20,15){\dots}

\drawline[AHnb=0](5,10)(10,15)
\drawline[AHnb=0](10,10)(15,5)
\drawline[AHnb=0](15,10)(20,15)
\drawline[AHnb=0](22,10)(27,5)
\drawline[AHnb=0](27,10)(32,15)
\drawline[AHnb=0](32,5)(37,10)
\drawline[AHnb=0](39,10)(44,5)
\drawline[AHnb=0](44,10)(49,15)
\drawline[AHnb=0](49,10)(54,5)

\put(37,10){\dots}
\put(37,5){\dots}
\put(37,15){\dots}

\put(5,4){\tiny ${w_{r+1}}$}
\put(10,4){\tiny ${w_{r+2}}$}
\put(15,4){\tiny ${w_{r+3}}$}
\put(18.5,4){\tiny ${w_{r+4}}$}
\put(22,4){\tiny ${w_{2}}$}
\put(27,4){\tiny ${w_{1}}$}
\put(32,4){\tiny ${w_{2}}$}
\put(37,4){\tiny ${w_{3}}$}
\put(39,4){\tiny ${w_{r-2}}$}
\put(44,4){\tiny ${w_{r-1}}$}
\put(49,4){\tiny ${w_{r}}$}
\put(53,4){\tiny ${w_{r+1}}$}

\put(5.5,9){\tiny ${w_{1}}$}
\put(10.5,9){\tiny ${w_{2}}$}
\put(15.5,9){\tiny ${w_{3}}$}
\put(20,9){\tiny ${w_{4}}$}
\put(22.5,9){\tiny ${w_{r}}$}
\put(27,9){\tiny ${w_{r+1}}$}
\put(32,9){\tiny ${w_{r+2}}$}
\put(33.5,10.5){\tiny ${w_{r+3}}$}
\put(39.5,9){\tiny ${w_{2r-2}}$}
\put(44.5,9){\tiny ${w_{2r-1}}$}
\put(49.5,9){\tiny ${w_{2r}}$}
\put(54.5,9){\tiny ${w_{1}}$}

\put(5.5,14){\tiny ${x_{1}}$}
\put(10.5,14){\tiny ${x_{2}}$}
\put(15.5,14){\tiny ${x_{3}}$}
\put(20,14){\tiny ${x_{4}}$}
\put(22.5,14){\tiny ${x_{r}}$}
\put(27.5,14){\tiny ${x_{r+1}}$}
\put(32.5,14){\tiny ${x_{r+2}}$}
\put(35,15.5){\tiny ${x_{r+3}}$}
\put(39.5,14){\tiny ${x_{2r-2}}$}
\put(44,14){\tiny ${x_{2r-1}}$}
\put(49.5,14){\tiny ${x_{2r}}$}
\put(54.5,14){\tiny ${x_{1}}$}

\put(24,0){\tiny Figure 16}

\end{picture}


\begin{picture}(0,0)(-43,-18)
\tiny
\setlength{\unitlength}{1.8mm}
\drawpolygon(5,5)(20,5)(20,15)(5,15)
\drawpolygon(22,5)(37,5)(37,15)(22,15)
\drawpolygon(39,5)(54,5)(54,15)(39,15)

\drawline[AHnb=0](10,5)(10,15)
\drawline[AHnb=0](15,5)(15,15)
\drawline[AHnb=0](27,5)(27,15)
\drawline[AHnb=0](32,5)(32,15)
\drawline[AHnb=0](5,10)(20,10)
\drawline[AHnb=0](22,10)(37,10)
\drawline[AHnb=0](49,5)(49,15)
\drawline[AHnb=0](44,5)(44,15)
\drawline[AHnb=0](39,10)(54,10)

\put(20,10){\dots}
\put(20,5){\dots}
\put(20,15){\dots}

\drawline[AHnb=0](5,10)(10,15)
\drawline[AHnb=0](10,10)(15,5)
\drawline[AHnb=0](15,10)(20,15)
\drawline[AHnb=0](22,10)(27,5)
\drawline[AHnb=0](27,15)(32,10)
\drawline[AHnb=0](32,10)(37,5)
\drawline[AHnb=0](39,15)(44,10)
\drawline[AHnb=0](49,10)(49,5)
\drawline[AHnb=0](49,15)(54,10)

\put(37,10){\dots}
\put(37,5){\dots}
\put(37,15){\dots}

\put(5,4){\tiny ${v_{1}}$}
\put(10,4){\tiny ${v_{2}}$}
\put(15,4){\tiny ${v_{3}}$}
\put(20,4){\tiny ${v_{4}}$}
\put(22,4){\tiny ${v_{r}}$}
\put(27,4){\tiny ${v_{r+1}}$}
\put(32,4){\tiny ${v_{r+2}}$}
\put(36,4){\tiny ${v_{r+3}}$}
\put(39,4){\tiny ${v_{2r-2}}$}
\put(44,4){\tiny ${v_{2r-1}}$}
\put(49,4){\tiny ${v_{2r}}$}
\put(54,4){\tiny ${v_{1}}$}

\put(5.5,9){\tiny ${w_{1}}$}
\put(10.5,9){\tiny ${w_{2}}$}
\put(15.5,9){\tiny ${w_{3}}$}
\put(20.0,9){\tiny ${w_{4}}$}
\put(22.7,9){\tiny ${w_{r}}$}
\put(27.5,9){\tiny ${w_{r+1}}$}
\put(32.5,9){\tiny ${w_{r+2}}$}
\put(33.5,10.5){\tiny ${w_{r+3}}$}
\put(39.5,9){\tiny ${w_{2r-2}}$}
\put(44.5,9){\tiny ${w_{2r-1}}$}
\put(49.5,9){\tiny ${w_{2r}}$}
\put(54.5,9){\tiny ${w_{1}}$}

\put(5.5,14){\tiny ${w_{r+1}}$}
\put(10.5,14){\tiny ${w_{r+2}}$}
\put(15,14){\tiny ${w_{r+3}}$}
\put(19,15.5){\tiny ${w_{r+4}}$}
\put(22.5,14){\tiny ${w_{2r}}$}
\put(27.8,14){\tiny ${w_{1}}$}
\put(32.5,14){\tiny ${w_{2}}$}
\put(37,14){\tiny ${w_{3}}$}

\put(39.5,14){\tiny ${w_{r-2}}$}
\put(44.5,14){\tiny ${w_{r-1}}$}
\put(49.5,14){\tiny ${w_{r}}$}
\put(54.5,14){\tiny ${w_{r+1}}$}

\put(24,0){\tiny Figure 17}

\end{picture}

\vspace{-1.7cm}

As above, by definition, we have two paths through each vertex in $M$. This implies that we have two cycles of type $B_1$ upto homologous in $M$. Let $C'$ and $C''$ denote two non-homologous cycles of type $B_1$ in $M$. As in Section 2.2, let $C'$ homologous to the generator of $\mathbb{Z}$ and $C_2$ homologous to generator of $\mathbb{Z}_2$ of the fundamental group of $M$. As these are two different groups, so, there is no isomorphism which maps one cycle to another cycle. So, we consider only one cycle in the proof of Lemma \ref{lem32434:4}. By the above Lemma \ref{lem32434:2}, every map has a $K(r, s, k)$ representation. So, we study only $K(r, s, k)$ representations in bellow. Next we define admissible relations among $r, s, k$ of $K(r, s, k)$ such that $K(r, s, k)$ represents a map after identifying its boundaries. The proof of next lemmas repeats similar argument as in Lemmas \ref{lem36:6} and \ref{lem3342:iso-1}. Thus, we have  

\begin{Lemma}\label{lem32434:3} The maps of type $\{3^{2}, 4, 3, 4\}$ of the form $K(r, s, k)$ exist if and only if the following holds : (i) rs $\geq$ 12, (ii) r $\geq$ 4, (iii) 2 $\mid$ r, (iv) s $\ge$ 3, (v) 2 $\nmid$ s, and (v) $k \in \{ 2t ~:~ 0 \leq t \leq (\frac{r}{2}-1)\}$ if $s \ge 3$. 
\end{Lemma}

\begin{Lemma}\label{lem32434:4}
Let $K(r_{i}, s_{i}, k_{i})$ denote a $(r_{i}, s_{i}, k_{i})$-representation of $M_{i}$ on same number of vertices for $i \in \{1, 2\}$. Then,\ $M_{1} \cong M_{2}$ if $(r_1, s_1, l_1) = (r_2, s_2, l_2)$ where $ l_i \in \{k_i~ mod(2), (k_i + r_i)~ mod(2)\}$.
\end{Lemma}

\begin{cor}\label{cor32434:1} Let $K(r_{i}, s_{i}, k_{i})$ denote a $(r_{i}, s_{i}, k_{i})$-representation of $M_{i}$ on same number of vertices for $i \in \{1, 2\}$. Then,\ $K(r_{1}, s_{1}, k_{1}) \not\cong K(r_{2}, s_{2}, k_{2}) ~\forall~ r_{1} \neq r_{2}, K(r_{1}, s_{1}, k_{1}) \not\cong K(r_{2}, s_{2}, k_{_{2}}) ~\forall ~ s_{1} \neq s_{2}, K(r_{1}, s_{1}, 0) \cong K(r_{1}, s_{1}, k_{1})$ for $2 \mid k_1$ or $2 \mid (k_1 + r_1)$, and $K(r_{1}, s_{1}, 1)$ $\cong K(r_{1}, s_{1}, k_{1})$ for $2 \mid (k_1 - 1)$ or $2 \mid (k_1 + r_1 - 1)$.
\end{cor}

\noindent \textbf{2.6 Maps of type $\{4, 8^{2}\}$~:~}\label{4821}
Let $M$ be a map of type $\{4, 8^{2}\}$ on the Klein bottle. Through each vertex in $M$, we have a path of type $Z_1$ \cite{mu:torus-hc13} in $M$. Let $P$ be a maximal walk of type $Z_1$. By the similar argument as in Section 2.1, it is a cycle since degree is three. Let $C$ be a cycle of type $Z_{1}$ in $M$. Let $S$ denote a set of faces incident with $C$. Then, we have 

\begin{Lemma}\label{tlem482:1} The $|S|$ is $S(C, M), \mathcal{M}(C, M)$ or $\mathcal{M}_{4,8}(C, M)$.
\end{Lemma}

\begin{proof} We repeat similar arguments as in Lemmas \ref{lem3342.1} and \ref{lem32434:1}. By the similar argument as in Lemma \ref{lem3342.1}, $|S|$ is either a cylinder $S(C, M)$ or a M\"{o}bius strip $\mathcal{M}_{4,8}(C, M)$. In this case, the $\mathcal{M}_{4,8}(C, M)$ denotes a M\"{o}bius strip which consists of $4$-gons and $8$-gons (the similar notion is also introduced in Section 2.1). Again, by the similar argument as in Lemma \ref{lem32434:1}, $|S| =\mathcal{M}(C, M)$. Therefore, $|S| = S(C, M)$ or $|S| =\mathcal{M}(C, M)$ or $|S| =\mathcal{M}_{4,8}(C, M)$. This completes the proof.
\end{proof}

So, by Lemma \ref{tlem482:1}, $S(C, M) \subset M, \mathcal{M}(C, M) \subset M$ or $\mathcal{M}_{4,8}(C, M) \subset M$. We consider each of $S(C, M), \mathcal{M}(C, M)$ and $\mathcal{M}_{4,8}(C, M)$ and repeat the similar process as in Section 2.4, 2.5. Hence, we get either a $K(r, s, k)$ or a $K(S'(M), \mathcal{M}', \mathcal{M}'')$ representation of $M$. Next we show that the map $M$ has a $K(r, s, k)$ representation.

\begin{Lemma} \label{lem482:2} The $M$ has a $K(r, s, k)$ representation.  
\end{Lemma}

\begin{proof} Let $K(S'(M), \mathcal{M}', \mathcal{M}'')$ denote a representation of $M$. Then, $\mathcal{M}', \mathcal{M}'' \in $ $\{\mathcal{M}(C_1, M),$ $\mathcal{M}_{4,8}(C_2, M) \mid$ for some cycles $C_1, C_2$ of type $Z_1 \}$. We repeat the similar construction as in Lemma  \ref{lem32434:2}. Hence, we get a planar polyhedral representation $K(r, s, k)$ for some $r, s, k$.
\end{proof}

Let $C'$ and $C''$ be two non-homologous cycles in $M$. Similarly as in Section 2.5, there is no isomorphic maps which maps cycle $C'$ to $C''$. Therefore, we study only $K(r, s, k)$ representations. Now we define admissible relations among $r, s, k$ of $K(r, s, k)$ such that $K(r, s, k)$ represents a map after identifying its boundaries. The proof of next lemmas repeats similar argument as in Lemmas \ref{lem36:6} and \ref{lem3342:iso-1}.

\begin{Lemma}\label{lem482:3} The maps of type $\{4, 8^{2}\}$ of the form $K(r, s, k)$ exist if and only if the following holds : (i) $rs \geq 24$, (ii) $4 \mid r$, (iii) $s \ge 3$, (iv) $r \geq 8$, (v) $k \in \{ 4t+3 ~\colon~ 0 \leq t \leq \frac{r-4}{4}\}$ if $s \ge 3$ and odd (vi) $k \in \{ (4t+2) (mod~r) ~\colon ~0 \leq t \leq \frac{r-4}{4}\}$ if $s \geq 4$ and even.
\end{Lemma}

\begin{Lemma}\label{lem482:4}
Let $K(r_{i}, s_{i}, k_{i})$ denote a $(r_i, s_i, k_i)$ -representation of $M_{i}$ on same number of vertices for $i \in \{1, 2\}$. Then,\ $M_{1} \cong M_{2}$ if $(r_1, s_1, l_1) = (r_2, s_2, l_2)$ where $ l_i \in \{k_i~ mod(2), (k_i + r_i)~ mod(2)\}$.
\end{Lemma}

\begin{cor} \label{cor482:1} Let $K(r_{i}, s_{i}, k_{i})$ denote a $(r_{i}, s_{i}, k_{i})$-representation of $M_{i}$ on same number of vertices for $i \in \{1, 2\}$. Then,\ $K(r_{1}, s_{1}, k_{1}) \not\cong K(r_{2}, s_{2}, k_{2}) \forall r_{1} \neq r_{2}$,  $K(r_{1}, s_{1}, k_{1}) \not\cong K(r_{2}, s_{2}, k_{_{2}}) \forall s_{1} \neq s_{2}, K(r_{1}, s_{1}, 0) \cong K(r_{1}, s_{1}, k_{1})$ for $2 \mid k_1$ or $2 \mid (k_1 + r_1), K(r_{1}, s_{1}, 1) \cong K(r_{1}, s_{1}, k_{1})$ for $2 \mid (k_1 - 1)$ or $2 \mid (k_1 + r_1 - 1)$.
\end{cor}

\noindent \textbf{2.7 Maps of type $\{3, 6, 3, 6\}$~:~}
Let $M$ be a map of type $\{3, 6, 3, 6\}$ on the Klein bottle. Let $W$ be a maximal walk of type $X_1$ \cite{mu:torus-hc13} through each vertex in $M$. The maximal walk $W$ is either a cycle or a closed walk as in Section 2.2. Similarly as in Section 2.1, the $M$ contains a cycle of type $X_1$. Let $C$ be a cycle of type $X_{1}$ in $M$. Let $S$ denote a set of faces incident with $C$. Then, we have 

\begin{Lemma}\label{lem3636:1} The $|S|$ is $S(C, M), \mathcal{M}(C, M)$ or $\mathcal{M}_{3,6}(C, M)$.
\end{Lemma}

\begin{proof} We argue similarly as in Lemmas \ref{lem3342.1} and \ref{lem32434:1}. Assume that faces which are incident with $C$ are all distinct. In this case, by the similar argument as in Lemmas \ref{lem3342.1} and \ref{lem32434:1}, $|S| = S(C, M)$ and $|S| =\mathcal{M}(C, M)$ respectively. Assume that faces which are incident with $C$ are not all distinct. Let $\triangle_1, F_2, \triangle_3, F_4,$ \dots, $F_{i-1}, \triangle_i, F_{i+1},$ \dots $\triangle_{k-1}, F_k$ denote a sequence of faces which are incident with $C$ in ordered and lie on one side of $C$ where $F_i$ denote a $6$-gon for each $i$. Since $\triangle_i \cap F_{i-1}, \triangle_i \cap F_{i+1}$ and $\triangle_i \cap C$ are each an edge, so, the triangles are distinct in the sequence. But the faces $F_i$s' are repeating in the sequence. Since each $6$-gon $F_i$ has exactly one edge remaining to identify, so, by similar argument as in Lemma \ref{lem32434:1}, $F_i$ repeats exactly once in the sequence. So, the geometric carrier $|S|=\mathcal{M}_{3,6}(C, M)$ and which consists of $3$-gons and $6$-gons. Therefore, the $|S|$ is either of $S(C, M), \mathcal{M}(C, M)$ and $\mathcal{M}_{3,6}(C, M)$. This completes the proof.
\end{proof}

So, by Lemma \ref{lem3636:1}, $S(C, M) \subset M$ or $\mathcal{M}(C, M) \subset M$ or $\mathcal{M}_{3, 6}(C, M) \subset M$ for some cycle $C$ of type $X_1$ in $M$. We repeat similar argument as in Sections 2.1 and 2.5, the map $M$ has either a $S(M)$ representation which is bounded by two identical cycles with opposite direction or a $S\mathcal{M}(M)$ which is bounded by two M\"obius strips $\mathcal{M}', \mathcal{M}'' \in \{\mathcal{M}(C, M), \mathcal{M}_{3,6}(L, M) \}$ from each of $S(C, M), \mathcal{M}(C, M)$ and $\mathcal{M}_{3, 6}(C, M)$. Let $S(M)$ denote a representation of $M$ such that $\partial S(M) = \{C, C'\}, \triangle \in F(M)$ and $C \cap \triangle$ is an edge. By definition of $X_1$, we have three paths of type $X_{1}$ through three edges of $\triangle$. Let $L_{1}, L_{2}, L_3$ be three paths of type $X_1$. We repeat the similar argument as in Section 2.5 to define a $(r, s, k)$-representation of a map $M$. Let $L_1 = C$. We take a cut along $L_3$. Since length$(C)$ = length$(C')$ and $C = C'$, hence, we get a planar polyhedral representation which is bounded by $L_3, C, L_3$ and $C'$. Since $C = C'$, so, it defines a $K(r, s, k)$ representation. Let $S\mathcal{M}(M)$ denote a representation of $M$. We cut along the boundaries of the M\"obius strips. We get three components. These are namely a cylinder $S'(M)$ and two M\"obius strips $\mathcal{M}', \mathcal{M}''$ where $\mathcal{M}', \mathcal{M}'' \in \{\mathcal{M}(C, M), \mathcal{M}_{3,6}(C, M)\}$ for some cycle $C$ of type $X_1$. We denote this representation by $K(S'(M),  \mathcal{M}', \mathcal{M}'')$. By combining above all cases, the map $M$ has either a $K(r, s, k)$ or a $K(S'(M), \mathcal{M}', \mathcal{M}'')$ representation and these are non-isomorphic. We study both the representations separately in bellow.

\smallskip

\textbf{Classification of $K(r, s, k)$ on $n$ vertices :} Let $\triangle$ in $M$ and $e_1, e_2, e_3 \subset \triangle$. By definition of $X_1$, we have three paths and similarly as in Section 2.1, let $C$ be a cycle of type $X_1$ such that $e_1 \subset C$ and $W$ be a closed walk of type $X_1$ such that that $e_2, e_3 \subset C$. So, there is one cycle of type $X_1$ upto homologous. We define admissible relations among $r, s, k$ of $K(r, s, k)$ such that $K(r, s, k)$ represents a map after identifying its boundaries. The proof of next lemmas repeats similar argument as in Lemmas \ref{lem36:6} and \ref{lem3342:iso-1}.

\begin{Lemma}\label{lem3636:2} The maps of type $\{3, 6, 3, 6\}$ of the form $K(r, s, k)$ on $n$ vertices exist if and only if the following holds : (i) $\frac{3}{2} rs \geq 27$, (ii) $s \ge 3$, (iii) $2 \mid r$, (iv) $r \geq 6$, and (v) $k \in \{ 2t+1 ~\colon~ 0 \leq t \leq \frac{r}{2}-1\}$.
\end{Lemma}

\begin{Lemma}\label{lem3636:3}
Let $K(r_{i}, s_{i}, k_{i})$ denote a $(r_i, s_i, k_i)$ -representation of $M_{i}$ on same number of vertices for $i \in \{1, 2\}$. Then,\ $M_{1} \cong M_{2}$ if $(r_1, s_1, l_1) = (r_2, s_2, l_2)$ where $ l_i \in \{k_i~ mod(2), (k_i + r_i)~ mod(2)\}$.
\end{Lemma}

\begin{cor} \label{cor3636:1} Let $K(r_{i}, s_{i}, k_{i})$ denote a $(r_{i}, s_{i}, k_{i})$-representation of $M_{i}$ on same number of vertices for $i \in \{1, 2\}$. Then,\ $K(r_{1}, s_{1}, k_{1}) \not\cong K(r_{2}, s_{2}, k_{2}) \forall r_{1} \neq r_{2}, K(r_{1}, s_{1}, k_{1}) \not\cong K(r_{2}, s_{2}, k_{_{2}}) \forall s_{1} \neq s_{2}, K(r_{1}, s_{1}, 0) \cong K(r_{1}, s_{1}, k_{1})$ for $2 \mid k_1$ or $2 \mid (k_1 + r_1)$, and $K(r_{1}, s_{1}, 1) \cong K(r_{1}, s_{1}, k_{1})$ for $2 \mid (k_1 - 1)$ or $2 \mid (k_1 + r_1 - 1)$.
\end{cor}

\textbf{Classification of $K(l, t)$ on $n$ vertices :} Let $S\mathcal{M}(M)$ denote $M$. Then, by the above argument, $K(S'(M), \mathcal{M}', \mathcal{M}'')$ has $S\mathcal{M}(M)$ representation.  Let $C' = \partial\mathcal{M}', C''=\partial\mathcal{M}''$. Then, length($C'$) = length($C''$) as $C'$ and $C''$ are boundaries of $S'(M)$. Let $l =$ length($C'$). Now  $\mathcal{M}', \mathcal{M}'' \in \{\mathcal{M}(L', M), \mathcal{M}_{3,6}(L'', M) \mid $ for some cycles $L'$ and $L''$ of type $X_1\}$. Let $\mathcal{M}' = \mathcal{M}(L', M)$ and  $\mathcal{M}'' = \mathcal{M}_{3,6}(L'', M)$. By definition and Lemma \ref{lem3636:1}, $C'' =  L''$ and length($C'$) = $2 \times$length($L'$). Thus, the map $M$ contains cycles of type $X_1$ of two different lengths. In this case, we denote $K(S'(M), \mathcal{M}', \mathcal{M}'')$ by $K(l_1, l_2, S'(M), \mathcal{M}', \mathcal{M}'')$ where $l_2 = l$ and $2l_1 = l_2$. When $\mathcal{M}' = \mathcal{M}_{3,6}(L', M)$ and  $\mathcal{M}'' = \mathcal{M}_{3,6}(L'', M)$ then we denote $K(S'(M), \mathcal{M}', \mathcal{M}'')$ by $K(l, S'(M), \mathcal{M}', \mathcal{M}'')$. The above all notions are used in the next lemma.

\begin{Lemma}\label{lem3636:4} The maps of type $\{3, 6, 3, 6\}$ of the form $K(\frac{l}{2}, l, S'(M), \mathcal{M}', \mathcal{M}'')$ exist if and only if the following holds :

(1) $t \ge 2, n = \frac{3}{2} tl + \frac{l}{2}$ where $4 \mid l$ and $l \ge 12$ if $\mathcal{M}' = \mathcal{M}_{3,6}(C', M)$ and $\mathcal{M}'' = \mathcal{M}_{3,6}(C'', M)$.

(2) $t \ge 1, n = l(t + 2)$ where $2 \mid l$ and $l \ge 10$ if $\mathcal{M}' = \mathcal{M}(L', M)$ and $\mathcal{M}'' = \mathcal{M}(L'', M)$ for some cycles $L'$ and $L''$ of type $X_1$.

(3) $t \ge 1, n = l(t + \frac{5}{4})$ where $4 \mid l$ and $l \ge 12$ if $\mathcal{M}' =\mathcal{M}_{3,6}(C', M)$ and $\mathcal{M}'' = \mathcal{M}(L'', M)$.

\end{Lemma}

\begin{proof} The proof repeats similar argument as in Section 2.1. In all the cases, if we consider velues of $t, l$ not in the given range then we get some vertex whose link is not a cycle.
\end{proof}

\begin{Lemma}\label{lem3636:5}
Let $K(S'(M_i), \mathcal{M}'_i, \mathcal{M}''_i)$ denote $M_i$ on same number of vertices. Then, we have the following. 

(1) Let $K(l_i, S'(M_i), \mathcal{M}'_i, \mathcal{M}''_i)$ denote $M_i$. Let $\mathcal{M}'_i = \mathcal{M}_{3,6}(C_i', M_i)$ and $\mathcal{M}''_i = \mathcal{M}_{3,6}(C_i'',$ $M_i)$. Then, $M_1 \cong M_2$ if $l_1 = l_2$.  

(2) Let $K(l_i, \frac{l_i}{2}, S'(M_i), \mathcal{M}'_i, \mathcal{M}''_i)$ denote $M_i$. Let $\mathcal{M}'_i = \mathcal{M}(C_i', M_i)$ and $\mathcal{M}''_i = \mathcal{M}(C_i'',$ $ M_i)$. Then, $M_1 \cong M_2$ if $l_1 = l_2$.

(3) Let $\mathcal{M}'_i =\mathcal{M}_{3,6}(C_i', M_i)$ and $\mathcal{M}''_i = \mathcal{M}(C_i'', M_i)$. Then, $M_1 \cong M_2$ if $l_1 = l_2$.   
\end{Lemma}

\begin{proof} Let $K(l_i, S'(M_i), \mathcal{M}'_i, \mathcal{M}''_i)$ denote $M_i$ and $\mathcal{M}'_i = \mathcal{M}_{3,6}(C', M)$ and $\mathcal{M}''_i = \mathcal{M}_{3,6}(C'',$ $ M)$. By Lemma \ref{lem3636:4}, $\frac{3}{2} t_1l_1 + \frac{1}{2}l_1 = n = \frac{3}{2} t_2l_2 + \frac{1}{2}l_2$ as $l_1 = l_2$. Hence, $t_1 = t_2 = s$ (say). Let $C_1, \dots, C_s$ denote the cycles of type $X_1$ in $M_1$. Let $C_1 = \partial \mathcal{M}'_1$ and $C_s = \partial \mathcal{M}''_1$. Again, let $L_1, \dots, L_s$ denote the cycles of type $X_1$ in $M_2$. Let $C_i= C(u(i, 1), \dots, u(i, l_1))$ and $L_i= C(v(i, 1), \dots, v(i, l_2))$. Let $L_1 = \partial \mathcal{M}'_2$ and $L_s = \partial \mathcal{M}''_2$. We repeat similar argument as in Lemma \ref{lem36:iso}. Hence, we define an isomorphism map $f_1 ~: ~V(M_1) \rightarrow V(M_2)$ by $f_1(V(C_i)) = V(L_i)$ and $f_1(u(j, 1)) = v(j, 1)$ for $1 \le j \le l_1$ and $1 \le i \le s$. So, $M_1 \cong M_2$. Similar argument we repeat for other two cases. If two maps have two different representation in the above three cases, then, they are non-isomorphic as they have non-isomorphic M\"{o}bius strips. This completes the proof.
\end{proof}

\noindent\textbf{2.8 Maps of type $\{3, 12^2\}$~:~}
Let $M$ be a map of type $\{3, 12^2\}$ on the Klein bottle. Let $W$ be a maximal walk of type $G_1$ \cite{mu:torus-hc13} in edge graph of $M$. Then, it is either a cycle or a closed walk (see similar argument as in Section 2.7). Similarly as in Section 2.7, the map $M$ contains a cycle of type $G_1$. Let $C$ be a cycle of type $G_{1}$ in $M$. Let $S$ denote a set of faces which are incident with $C$. Then, $|S| = S(C, M)$ or $|S| =\mathcal{M}(C, M)$ or $|S| =\mathcal{M}_{3,12}(C, M)$ (see similar argument as in Lemma \ref{lem3636:1}). Hence, either $S(C, M) \subset M, \mathcal{M}(C, M)\subset M$ or a $\mathcal{M}_{3, 12}(C, M)\subset M$. Consider each of $S(C, M), \mathcal{M}(C, M)$ and $\mathcal{M}_{3, 12}(C, M)$ and we proceed similarly as in Section 2.7, hence, we get either a 
$K(r, s, k)$ or a $K(S'(M), \mathcal{M}', \mathcal{M}'')$ representation of $M$ and these are non-isomorphic. We study both the representations separately in bellow.

\textbf{Classification of $K(r, s, k)$ on $n$ vertices :} Let $\triangle$ in $F(M)$. Similarly as in the previous Section 2.7, let $C$ be a cycle of type $G_1$ which contains an edge of $\triangle$ and $W$ be a closed walk of type $G_1$ which contains other two edges of $\triangle$. Hence, there is one cycle of type $G_1$ upto homologous in $M$. We define admissible relations among $r, s, k$ of $K(r, s, k)$ in the next lemma and its proof follows from the similar arguments as in Lemma \ref{lem36:6}. Then, the next isomorphism lemma whose proves follows from similar argument as in Lemma \ref{lem3342:iso-1}.

\begin{Lemma}\label{lem3122:1} The maps of type $\{ 3, 12^2\}$ of the form $K(r, s, k)$ on $n$ vertices exist if and only if the following holds : (i) $n = \frac{3}{2}rs$, (ii) $4 \mid r,$ (iii) $6 \mid n$, (iv) $s \geq 3$, (v) $r \geq 12$, and (vi) $k_{1} \in \{ 4t+2 ~\colon~ 0 \leq t \leq \frac{r-4}{4}\}$.
\end{Lemma}

\begin{Lemma}\label{lem3122:2}
Let $K(r_{i}, s_{i}, k_{i})$ denote a $(r_i, s_i, k_i)$-representation of $M_{i}$ on same number of vertices for $i \in \{1, 2\}$. Then,\ $M_{1} \cong M_{2}$ if $(r_1, s_1, l_1) = (r_2, s_2, l_2)$ where $ l_i \in \{k_i~ mod(2), (k_i + r_i)~ mod(2)\}$.
\end{Lemma}

\begin{cor} \label{cor3122:1}  $K(r_{1}, s_{1}, k_{1}) \not\cong K(r_{2}, s_{2}, k_{2})$ if $r_{1} \neq r_{2}, K(r_{1}, s_{1}, k_{1}) \not\cong K(r_{2}, s_{2}, k_{_{2}})$ if $s_{1} \neq s_{2}, K(r_{1}, s_{1}, 0) \cong K(r_{1}, s_{1}, k_{1})$ if $2 \mid k_1$ or $2 \mid (k_1 + r_1)$, and $K(r_{1}, s_{1}, 1) \cong K(r_{1}, s_{1}, k_{1})$ if $2 \mid (k_1 - 1)$ or $2 \mid (k_1 + r_1 - 1)$.
\end{cor}

\textbf{Classification of $K(l, t)$ on $n$ vertices :} Let $M$ be a map of type $\{3, 12^2\}$ on $n$ vertices. We repeat similar arguments and notions of Section 2.7 in bellow. Now we have the next lemma and isomorphism lemma. The proof of the lemma follows similar argument as in Lemma \ref{lem3636:4} and the proof of the isomorphism lemma follows from similar argument as in Lemma \ref{lem3636:5}.

\begin{Lemma}\label{lem3122:3} The maps of type $\{3, 12^2\}$ of the form $K(\frac{l}{2}, l, S'(M), \mathcal{M}', \mathcal{M}'')$ exist if and only if the following holds : 

(1) $t \ge 2, n = \frac{3}{2} tl + \frac{l}{2}$ where $8 \mid l$ and $l \ge 24$ if $\mathcal{M}' = \mathcal{M}_{3,12}(C', M)$ and $\mathcal{M}'' = \mathcal{M}_{3,12}(C'', M)$.

(2) $t \ge 1, n = l(t + 2)$ where $4 \mid l$ and $l \ge 20$ if $\mathcal{M}' = \mathcal{M}(L', M)$ and $\mathcal{M}'' = \mathcal{M}(L'', M)$ for some cycles $L'$ and $L''$ of types $G_1$.

(3) $t \ge 1, n = l(t + \frac{5}{4})$ where $4 \mid l$ and $l \ge 24$ if $\mathcal{M}' =\mathcal{M}_{3,12}(C', M)$ and $\mathcal{M}'' = \mathcal{M}(L'', M)$.
\end{Lemma}


\begin{Lemma}\label{lem3122:4}
Let $K(S'(M_i), \mathcal{M}'_i, \mathcal{M}''_i)$ represent $M_i$ on $n$ vertices. Then, we have the following cases : 

(1) Let $K(l_i, S'(M_i), \mathcal{M}'_i, \mathcal{M}''_i)$ represents $M_i$ such that $\mathcal{M}'_i = \mathcal{M}_{3,12}(C_i', M_i)$ and $\mathcal{M}''_i = \mathcal{M}_{3,12}(C_i'', M_i)$. Then $M_1 \cong M_2$ if $l_1 = l_2$.  

(2) Let $K(l_i, \frac{l_i}{2}, S'(M_i), \mathcal{M}'_i, \mathcal{M}''_i)$ represents $M_i$ such that $\mathcal{M}'_i = \mathcal{M}(C_i', M_i)$ and $\mathcal{M}''_i = \mathcal{M}(C_i'', M_i)$. Then $M_1 \cong M_2$ if $l_1 = l_2$. 

(3) Let $\mathcal{M}'_i =\mathcal{M}_{3,12}(C_i', M_i)$ and $\mathcal{M}''_i = \mathcal{M}(C_i'', M_i)$. Then $M_1 \cong M_2$ if $l_1 = l_2$.  
\end{Lemma}

\noindent \textbf{2.9 Maps of type $\{3^4, 6\}$~:~}\label{3461}
Let $M$ be a map of type $\{3^4, 6\}$ on the Klein bottle. Let $P$ be a maximal walk of type $Y_1$ \cite{mu:torus-hc13}. It is either a cycle or a closed walk as in Section 2.1. As in Section 2.7, the map $M$ contains a cycle, say $C$ of type $Y_{1}$ in $M$. Let $S$ denote a set of faces incident with $C$. Then, similarly as in Lemma \ref{lem3636:1}, $|S| = S(C, M)$. Next we show that there is no map of type $\{3^4, 6\}$ on the Klein bottle.

\begin{theo} \label{thm346:1} There is no map of type $\{3^4, 6\}$ on the Klein bottle.
\end{theo}

\begin{proof} By the above argument, let $S(C, M) \subset M$ for some cycle $C$ of type $Y_1$. We consider $S(C, M)$ and repeat with the arguments as in Sections 2.4, 2.5. In this process, let $S(M)$ denote $M$. 

\smallskip

$Claim : $ The cylindrical representation $S(M)$ does not exist in $M$.

\smallskip

Let $\partial S(M) = \{C', C''\}$ such that $C' = C''$ and $C'$ identifies with $C''$ with a twist. Let $C' = C(w_1, \dots, w_r)$. By definition of $Y_1, C'$ divides number of triangles at $w_i$ one of $0 : 4,  1 : 3$ and $2 : 2$ ratios (here, $a : b$ ratio is to mean $a$ number of triangles lie on one side and $b$ number of triangles lie on the other side of the cycle at a vertex). Let $P(w_{j-1}, w_j, w_{j+1}) \subset C', C''$ for some $j$ as $C' = C''$. Let $i \in \{0, 1, 2\}$. If $P \subset C'$ then $P$ divides incident triangles one of $i : 4-i, 1+i : 3-i$ and $2+i : 2-i$ ratios for some fixed $i$. If $P \subset C''$ then $P$ divides the incident triangles one of $i : 4-i, 2+i : 2-i$ or $1+i : 3-i$ ratios for some fixed $i$. Now we have the following cases. Observe that by definition of $Y_1$, if ratio $i : 4-i$ is at $w_i \in C'$ then $i : 4-i$ is at $w_i \in C''$ in $S(M)$. Again, if ratio $1+i : 3-i$ is at $w_i \in C'$ then $2+i : 2-i$ is at $w_i \in C''$. Similarly again, if ratio $2+i : 2-i$ is at $w_i \in C'$ then $1+i : 3-i$ is at $w_i \in C''$. In later two cases, it is impossible to identify boundaries of $S(M)$ at $w_i$ as the $lk(w_i)$ divides incident faces into two different ratios. So, the $C'$ can not identify with $C''$ with a twist. So, $S(M)$ does not exist. Therefore, there is no map of type $\{3^4, 6\}$ on the Klein bottle. This completes the proof. 
\end{proof}

\noindent \textbf{2.10 Maps of type $\{4, 6, 12\}$~:~}\label{46121} 
Let $M$ be a map of type $\{4, 6, 12\}$ on the Klein bottle. Let $P$ be a maximal walk of type $H_1$ \cite{mu:torus-hc13} in $M$. It is either a cycle or a closed walk (see the similar argument as in Section 2.7). Since degree of each vertex is three, it follows that the map $M$ has a cycle of type $H_1$ as in Section 2.7. Let $C$ be a cycle of type $H_{1}$ in $M$. Then, the geometric carrier of the faces which are incident with $C$ is either $S(C, M)$ or $\mathcal{M}_{4, 12}(C, M)$ (see the similar argument as in Lemma \ref{lem3636:1}). Thus, either $S(C, M) \subset M$ or $S\mathcal{M}_{4, 12}(C, M) \subset M$. We consider each of $S(C, M)$ and $S\mathcal{M}_{4, 12}(C, M)$, and repeat the similar arguments as in Section 2.5. Hence, we get either a $K(r, s, k)$ or a $K(S'(M), \mathcal{M}', \mathcal{M}'')$ representation of $M$.

\textbf{Classification of $K(r, s, k)$ on $n$ vertices :} We define admissible relations among $r, s, k$ of $K(r, s, k)$ such that $K(r, s, k)$ represents a map after identifying its boundaries. The proof of next lemma repeats similar argument as in Lemma \ref{lem36:6} and the proof of Lemma \ref{lem4612:2} follows from similarly as in Lemma \ref{lem3342:iso-1}. Thus, we have 

\begin{Lemma}\label{lem4612:1} The maps of tyep $\{ 4, 6, 12\}$ of the form $K(r, s, k)$ on $n$ vertices exist if and only if the following holds : (i) $n = rs \geq 48$, (ii) $2 \mid s$, (iii) $6 \mid r,$ (iv) $s \ge 4$, (v) $12 \mid n$, (vi) $r \geq 12$, and (vii) $k \in \{ 6t+4 ~\colon ~0 \leq t \leq \frac{r-6}{6}\}$.
\end{Lemma}

\begin{Lemma}\label{lem4612:2}
Let $K(r_{i}, s_{i}, k_{i})$ denote a $(r_i, s_i, k_i)$-representation of $M_{i}$ on same number of vertices for $i \in \{1, 2\}$. Then,\ $M_{1} \cong M_{2}$ if $(r_1, s_1, l_1) = (r_2, s_2, l_2)$ where $ l_i \in \{k_i~ mod(2), (k_i + r_i)~ mod(2)\}$. 
\end{Lemma}

\begin{cor}\label{cor4612:1} The $K(r_{1}, s_{1}, k_{1}) \not\cong K(r_{2}, s_{2}, k_{2})$ if $r_{1} \neq r_{2}, K(r_{1}, s_{1}, k_{1}) \not\cong K(r_{2}, s_{2}, k_{_{2}})$ if $s_{1} \neq s_{2}, K(r_{1}, s_{1}, 0) \cong K(r_{1}, s_{1}, k_{1})$ if $2 \mid k_1$ or $2 \mid (k_1 + r_1)$, and $K(r_{1}, s_{1}, 1) \cong K(r_{1}, s_{1}, k_{1})$ if $2 \mid (k_1 - 1)$ or $2 \mid (k_1 + r_1 - 1)$.
\end{cor}

\textbf{Classification of $K(l, t)$ on $n$ vertices :} Let $M$ be a map of type $\{4, 6, 12\}$ on $n$ vertices. Let $S\mathcal{M}(M)$ denote $M$. Then, it has a $K(l, S'(M), \mathcal{M}', \mathcal{M}'')$ representation (see similar argument as in Section 2.4). Let $C'$ and $C''$ denote two boundary cycles of $\mathcal{M}', \mathcal{M}''$ respectively. Then, $C' = \partial\mathcal{M}', C'' = \partial\mathcal{M}''$ and length($C'$) = length($C''$). Let $C' = \partial\mathcal{M}', C''=\partial\mathcal{M}''$ and $\{C_1, C_2\} = \partial S'(M)$. Let $l = length(C') = length(C'')$ and $t$ denote the number of cycles in $S'(M)$ which are of type $H_1$. The proof of next lemmas repeats similar argument as in Lemmas \ref{lem3342:5} and \ref{lem3636:5}.

\begin{Lemma}\label{lem4612:3} The maps of type $\{4, 6, 12\}$ of the form $K(S'(M), \mathcal{M}', \mathcal{M}'')$ on $n$ vertices exist if and only if the following holds :  (i) $t \ge 2$, (ii) $2 \mid t$, (iii) $n = tl$ where $12 \mid l, l \ge 24$ if $\mathcal{M}', \mathcal{M}'' \in \{\mathcal{M}_{4,12}(C, M)\}$.
\end{Lemma}

Let $K(l_i, S'(M_i), \mathcal{M}'_i, \mathcal{M}''_i)$ denote a representation of $M_i$ on $n$ vertices where $\mathcal{M}'_i, \mathcal{M}''_i \in \{\mathcal{M}_{4,12}(C, M_i) \mid$ for some cycle $C$ of type $H_1\}$. Then, we have :

\begin{Lemma}\label{lem4612:4}
The map $M_1 \cong M_2$ if $l_1 = l_2$.   
\end{Lemma}

Let $t_i$ denote the number of cycles of type $H_1$ in $M_i$. Then $K(l_{i}, t_i)$ denotes $K(S'(M_i),$ $\mathcal{M}'_i,$ $\mathcal{M}''_i)$ for $i \in \{1, 2\}$. Then, we have 

\begin{cor} \label{cor4612:2}  The $K(l_{1}, t_{1}) \not\cong K(l_{2}, t_{2})$ if $l_{1} \neq l_{2}$ and $K(l_{1}, t_{1}) \not\cong K(l_{2}, t_{2})$ if $t_{1} \neq t_{2}$.
\end{cor}

\noindent \textbf{2.11 Maps of type $\{3, 4, 6, 4\}$~:~}
Let $M$ be a map of type $\{3, 4, 6, 4\}$ on the Klein bottle. It has a walk of type $W_1$ \cite{mu:torus-hc13} in $M$ through each vertex in $M$. Similarly as in Section 2.1, let $C$ be a cycle of type $W_{1}$ in $M$. Let $S$ denote a set of faces which are incident with $C$. Then $|S|$ is one of $S(C, M), \mathcal{M}_{3, 4}(C, M)$ and $\mathcal{M}_{4, 6}(C, M)$ as in Lemma \ref{lem3636:1}. So, the map $M$ contains either $S(C, M), \mathcal{M}_{3, 4}(C, M)$ or $\mathcal{M}_{4, 6}(C, M)$. We consider each of $S(C, M), \mathcal{M}_{3, 4}(C, M)$ and $\mathcal{M}_{4, 6}(C, M)$, and repeat the similar arguments as in Sections 2.4, 2.5. Hence, we get either a $K(r, s, k)$ or a $K(S'(M), \mathcal{M}', \mathcal{M}'')$ representation of $M$. We study both the representations in bellow. 

\textbf{Classification of $K(r, s, k)$ on $n$ vertices :} We define admissible relations among $r, s, k$ of $K(r, s, k)$ such that $K(r, s, k)$ represents a map after identifying its boundaries. It's proof repeats similar argument as in Lemma \ref{lem36:6} and the proof of next Lemma \ref{lem3464:2} follows from similar argument as in Lemma \ref{lem3342:iso-1}. Thus, we have

\begin{Lemma}\label{lem3464:1} The maps of type $\{3, 4, 6, 4\}$ of the form $K(r, s, k)$ on $n$ vertices exist if and only if the following holds : (i) $n = rs \geq 24$, (ii) $3 \mid r$, (iii) $2 \mid s$, (iv) $s \ge 4$, (v) $r \geq 6$, and (vi) $k \in \{ 3t+2 ~\colon~ 0 \leq t \leq \frac{r-3}{3}\}$.
\end{Lemma}

\begin{Lemma}\label{lem3464:2}
Let $K(r_{i}, s_{i}, k_{i})$ denote a $(r_i, s_i, k_i)$-representation of $M_{i}$ on same number of vertices for $i \in \{1, 2\}$. Then,\ $M_{1} \cong M_{2}$ if $(r_1, s_1, l_1) = (r_2, s_2, l_2)$ where $ l_i \in \{k_i~ mod(2), (k_i + r_i)~ mod(2)\}$. 
\end{Lemma}

\begin{cor}\label{cor3464:1}  $K(r_{1}, s_{1}, k_{1}) \not\cong K(r_{2}, s_{2}, k_{2})$ if $r_{1} \neq r_{2}$,  $K(r_{1}, s_{1}, k_{1}) \not\cong K(r_{2}, s_{2}, k_{_{2}})$ if $s_{1} \neq s_{2}, K(r_{1}, s_{1}, 0) \cong K(r_{1}, s_{1}, k_{1})$ if $2 \mid k_1$ or $2 \mid (k_1 + r_1)$, and $K(r_{1}, s_{1}, 1) \cong K(r_{1}, s_{1}, k_{1})$ if $2 \mid (k_1 - 1)$ or $2 \mid (k_1 + r_1 - 1)$.
\end{cor}

\textbf{Classification of $K(l, t)$ on $n$ vertices :} We repeat the similar argument as in Section 2.10. Let $K(l, S'(M), \mathcal{M}', \mathcal{M}'')$ denote a representation of $M$. Let $C'$ and $C''$ denote two boundary cycles of $\mathcal{M}', \mathcal{M}''$ respectively. Let $C' = \partial\mathcal{M}', C'' = \partial\mathcal{M}''$ and $\partial S'(M) = \{C', C''\}$. Let $l = length(C') = length(C'')$ and $t$ denote the number of cycles of type $W_1$ in $S'(M)$. Let $\mathcal{M}', \mathcal{M}'' \in \{\mathcal{M}_{3,4}(L', M), \mathcal{M}_{4,6}(L'', M)\}$ for some cycles $L'$ and $L''$ of type $W_1$. Then we have the following lemmas. 

\begin{Lemma}\label{lem3464:3} The maps of type $\{3, 4, 6, 4\}$ of the form $K(S'(M), \mathcal{M}', \mathcal{M}'')$ on $n$ vertices exist if and only if the following holds :   

(1) $t \ge 2, n = tl$ where $6 \mid (l-3)$ and $l \ge 9$ if $\mathcal{M}' = \mathcal{M}_{3,4}(C', M)$ and $\mathcal{M}'' = \mathcal{M}_{3,4}(C'', M)$.

(2) $t \ge 2, n = lt$ where $t \ge 2, 6 \mid l$ and $l \ge 12$ if $\mathcal{M}' = \mathcal{M}_{4, 6}(L', M)$ and $\mathcal{M}'' = \mathcal{M}_{4, 6}(L'', M)$.

(3) $K(S'(M), \mathcal{M}', \mathcal{M}'')$ not exists if $\mathcal{M}' =\mathcal{M}_{3,4}(C', M)$ and $\mathcal{M}'' = \mathcal{M}_{4, 6}(L', M)$.
\end{Lemma}

\begin{proof} We repeat similar argument as in Lemma \ref{lem3342:5} for first two cases. If there exists a map $M$ which has a $K(l, S'(M), \mathcal{M}', \mathcal{M}'')$ representation such that $\mathcal{M}' =\mathcal{M}_{3,4}(C', M)$ and $\mathcal{M}'' = \mathcal{M}_{4, 6}(L', M)$. Then by the above cases, $6 \mid (l-3)$ and $6 \mid l$. But there is no such $l$ which satisfies both the conditions. So, $M$ does not exist. This completes the proof.   
\end{proof}

The proof of Lemma \ref{lem3464:4} repeats similar argument as in Lemma \ref{lem3636:5}. Thus, we have

\begin{Lemma}\label{lem3464:4}
Let $K(l_i, S'(M_i), \mathcal{M}'_i, \mathcal{M}''_i)$ denote $M_i$ on $n$ vertices. Then, we have the following : 

(1) Let $\mathcal{M}'_i = \mathcal{M}_{3, 4}(C_i', M_i)$ and $\mathcal{M}''_i = \mathcal{M}_{3, 4}(C_i'', M_i)$. Then $M_1 \cong M_2$ if $l_1 = l_2$.  

(2) Let $\mathcal{M}'_i = \mathcal{M}_{4, 6}(C_i', M_i)$ and $\mathcal{M}''_i = \mathcal{M}_{4, 6}(C_i', M_i)$. Then $M_1 \cong M_2$ if $l_1 = l_2$.   
\end{Lemma}

\section{Proofs}\label{result:all}

\noindent{Proof of Theorem} \ref{theo:main} : From Section 2.1, maps of type $\{3^6\}$ have either $K(r, s, k)$ or $K(l, t)$ representation for some $r, s, k, l, t$. By Lemma \ref{lem36:6}, it follows that $rs \geq 9, r \geq 3, s \ge 3$. Let $m = s, r = \frac{n}{m}$. Then $m = s \geq 3, r = \frac{n}{m} \geq 3$ i.e. $n \geq 3m$. Now we have following cases. If $gcd(r, 2) = 1$ i.e. $gcd(n, 2m) = m$ then $r+k$ is even for odd $k \in \{ t : 0 \leq t \leq r-1\}$. That is, $k$ and $r+k$ are both even for $gcd(r, 2) = 1$ and $k \in \{ t : 0 \leq t \leq r-1\}$. In this case, we get only one map which is isomorphic to $K(r, s, 0)$. If $gcd(r, 2) = 2$ i.e. $gcd(n, 2m) = 2m$ then $r+k, k$ are odd for odd $k \in \{ t : 0 \leq t \leq r-1\}$. In this case, we get two maps where one is isomorphic to $K(r, s, 1)$ and another is isomorphic to $K(r, s, 0)$. Therefore, in this case, we get two maps. By Lemma \ref{lem36.7}, $K(l, t)$ exists if (i) $tl \ge 10$, (ii) $t \ge 2$,  (iii) $2 \nmid l$ and $l \ge 5$. Let $m = t$. Then $m \geq 2, m \mid n, n \geq 5m$. If $gcd(n, 2m) = m$ i.e. $\frac{n}{m}$ is even then it has an unique representation $K(\frac{n}{m}, m)$. If $gcd(n, 2m) = 2m$ then there is no map by Lemma \ref{lem36.7}. Since $K(r, s, k)$ and $K(l, t)$ are non-isomorphic by Lemma \ref{lem36:4}, so, by combining above all three cases, if the number of maps, say $i(n)_{\{3^{6}\}}$ on $n$ vertices then $i(n)_{\{3^{6}\}} = \sum_{i = 1, 2} i \times |\{(m,n)~|~m \geq 3, n \geq 3m,~ gcd(n,2m) = i \times m\}|  + |\{(m,n)~|~m \geq 2, m \mid n, n \geq 5m,~ gcd(n, 2m) = m\}|$.

From Section 2.2, by Lemma \ref{lem44:2}, rs $\geq$ 9, r $\geq$ 3, s $\ge$ 3 and $k \in \{ t : 0 \leq t \leq r-1\}$ for $s \ge 3$. Let $m = s, r = \frac{n}{m}$. Similarly as above, if $k \in \{ t : 0 \leq t \leq r-1\}$ then $k$ and $r+k$ both are even for $gcd(r, 2) = 1$. In this case, we get only one map which is isomorphic to $K(r, s, 0)$. If $gcd(r, 2) = 2$ then similarly, $r+k, k$ are odd for odd $k \in \{ t : 0 \leq t \leq r-1\}$. In this case, we get two maps where one is isomorphic to $K(r, s, 1)$ and another is isomorphic to $K(r, s, 0)$. Thus, we get two non-isomorphic maps. By combining above all two cases, if number of maps $i(n)_{\{4^{4}\}}$ (say) on $n$ vertices then $i(n)_{\{4^{4}\}} = \sum_{i = 1, 2} i \times |\{(m, n)~|~m \geq 3, n \geq 3m,~ gcd(n,2m) = i \times m\}|$.

Let $i(2n)_{\{6^{3}\}}$ denote the number of maps of type $\{6^{3}\}$ on $2n \geq 14$ vertices. Then, by duality, $i(2n)_{\{6^{3}\}} = i(n)_{\{3^{6}\}}$ as number of faces in map of type $\{3^{6}\}$ is equal to the number of vertices in map of type $\{6^{3}\}$.

Similarly as above by Lemmas \ref{lem3342:4}, \ref{lem3342:5}, let $i(n)_{\{3^{3},4^{2}\}}$ denote the number of maps of type $\{3^{3},4^{2}\}$ on $n$ vertices. Then $i(n)_{\{3^{3},4^{2}\}} =  \sum_{i=1, 2} i \times |\{(m,n)~|~m \geq 4, 2 \mid m, n \geq 3m,~ gcd(n,2m) =i \times m\}| + |\{(m, n)~|~m \geq 5, n = 2m\}|  + |\{(m, n)~|~m \geq 4, n \geq 4m, 2m \mid n\}|$.

By Lemma \ref{lem32434:3}, let $i(n)_{\{3^{2}, 4, 3, 4\}}$ denote the number of maps of type $\{3^{2}, 4, 3, 4\}$ on $n$ vertices. Then $i(n)_{\{3^{2}, 4, 3, 4\}} = |\{(m,n)~|~m \geq 3, 2 \nmid m, 2m \mid n, n \geq 12\}|$.

By Lemma \ref{lem482:3}, let $i(n)_{\{4, 8^{2}\}}$ denote the number of maps of type $\{4, 8^{2}\}$ on $n$ vertices. Then $i(n)_{\{4,8^{2}\}} = |\{(l,m,n)~|~ m \geq 3, 2 \mid m, n \geq 8m, 4m \mid n, 0 \le l \le (\frac{n}{4m}-1)\}| + |\{(l, m, n)~|~ m \geq 3, 2 \nmid m, n \geq 8m, 4m \mid n, 0 \le l \le (\frac{n}{4m}-1)\}|$.

By Lemmas \ref{lem3636:2}, \ref{lem3636:4}, let $i(n)_{\{3,6,3,6\}}$ denote the number of maps of type $\{3,6,3,6\}$ on $n$ vertices. Then $i(n)_{\{3,6,3,6\}} = |\{(l, m, n)~|~ m \geq 3, 3m \mid n, n \geq 9m, 0 \le l \le (\frac{n}{3m}-1)\}| +  |\{(m, n)~|m \geq 2, 6m+2 \mid n, n \geq 6(3m+1)\}|  +  |\{(m, n)~|~m \geq 1, 2(m+2) \mid n, n \geq 10(m+2)\}|  +  |\{(m, n)~|~m \geq 1, 4m+5 \mid n, n \geq 12(4m+5)\}|$.

By Lemmas \ref{lem3122:1}, \ref{lem3122:3}, let $i(n)_{\{3,6,3,6\}}$ denote the number of maps of type $\{3,12^2\}$ on $n$ vertices. Then $i(n)_{\{3,12^2\}} = |\{(l, m, n) ~|~ m \geq 3, 3m \mid n, n \geq 9m, 0 \le l \le (\frac{n}{6m}-1)\}| + |\{(m, n)~|~m \geq 2, 4(3m+1) \mid n, n \geq 12(3m+1)\}| + |\{(m, n)~|~m \geq 1, 4(m+2) \mid  n, n \geq 20(m+2)\}| + |\{(m, n)~|~m \geq 1, 4m+5 \mid n, n \geq 24(4m+5)\}|$.

By Theorem \ref{thm346:1}, there is no map of type $\{3^4,6\}$ on $n$ vertices. So, $i(n)_{\{3^4,6\}} = 0$ for all $n$.

By Lemmas \ref{lem4612:1}, \ref{lem4612:3}, let $i(n)_{\{4, 6, 12\}}$ denote the number of maps of type $\{4, 6, 12\}$ on $n$ vertices. Then $i(n)_{\{4, 6, 12\}} = \{(l, m, n)~|~ m \geq 4, 2 \mid m, 6m \mid n, n \geq 12m, 0 \le l \le (\frac{n}{6m}-1)\}| + |\{(m, n)~|~m \geq 2, 2 \mid m, 12m \mid n, n\geq 24m\}|$.

By Lemmas \ref{lem3464:1}, \ref{lem3464:3}, let $i(n)_{\{3, 4, 6, 4\}}$ denote the number of maps of type $\{3, 4, 6, 4\}$ on $n$ vertices. Then $i(n)_{\{3, 4, 6, 4\}} = |\{(l, m, n)~|~ m \geq 4, 2 \mid m, 3m \mid n, n \geq 6m, 0 \le l \le (\frac{n}{3m}-1)\}| + |\{(m, n)~|~m \geq 2, 6m \mid (n-3m), n \geq 9m\}|  +  |\{(m, n)~|~m \geq 2, 6m \mid n, n \geq 12m\}|$.
\hfill$\Box$

\smallskip

\noindent{Proof of Theorem} \ref{classification-semi-maps} : The proof follows from the Sections 2.1, 2.2, 2.3, 2.4, 2.5, 2.6, 2.7, 2.8, 2.9, 2.10, 2.11. Let $M$ be a map of type $\{3^{6}\}$ on the Klein bottle with $n$ vertices. By Lemmas \ref{lem36:6}, \ref{lem36.7}, we consider all admissible $K(r, s, k)$ and $K(l, t)$ representations of $M$ on $n$ vertices. Then, we classify them by Lemmas \ref{lem36:iso} and \ref{lem36.8}. Let $r_1, r_2 \mid n$. By Corollary \ref{cor36:1}, \ref{cor36:2}, $K(r_1, s, k) \not\cong K(r_2, s, k), K(r_1, t) \not\cong K(r_2, t)$ if $r_1 \neq r_2$ for all $s, k, t$, and $K(r, s, k)$ is isomorphic to either of $K(r, s, 0), K(r, s, 1)$. Hence, every map of type $\{3^{6}\}$ on $n$ vertices is isomorphic to one of $K(r, s, 0), K(r, s, 1)$ and $K(r, s)$ for some $r \mid n, s = \frac{n}{r}$. 

Similarly by Corollary \ref{cor44:1}, $K(r_1, s, k) \not\cong K(r_2, s, k)$ if $r_1 \neq r_2$ for all $s, k$, and every map of type $\{4^{4}\}$ on $n$ vertices is isomorphic to one of $K(r, s, 0), K(r, s, 1)$ for some $r \mid n$ and $s = \frac{n}{r}$.

By Corollary \ref{cor3342:1}, \ref{cor3342:2}, $K(r_1, s, k) \not\cong K(r_2, s, k)$ and $K(r_1, t) \not\cong K(r_2, t)$ if $r_1 \neq r_2$ for all $s, k, t$, and $K(r, s, k)$ is isomorphic to either of $K(r, s, 0), K(r, s, 1)$. So, every map of type $\{3^{3}, 4^2\}$ on $n$ vertices is isomorphic to one of $K(r, s, 0), K(r, s, 1), K(r, s)$ for some $r \mid n$ and $s = \frac{n}{r}$.

By Corollary \ref{cor32434:1}, $K(r_1, s, k) \not\cong K(r_2, s, k)$ if $r_1 \neq r_2$ for all $s, k$, and every map of type $\{3^{2}, 4, 3, 4\}$ on $n$ vertices is isomorphic to one of $K(r, s, 0), K(r, s, 1)$ for some $r \mid n$ and $s = \frac{n}{r}$.

By Corollary \ref{cor482:1}, $K(r_1, s, k) \not\cong K(r_2, s, k)$ if $r_1 \neq r_2$ for all $s, k$, and every map of type $\{4, 8^{2}\}$ on $n$ vertices is isomorphic to one of $K(r, s, 0), K(r, s, 1)$ for some $r \mid n$ and $s = \frac{n}{r}$.

By Corollary \ref{cor3636:1} and Lemma \ref{lem3636:5}, $K(r_1, s, k) \not\cong K(r_2, s, k)$ and $K(r_1, t) \not\cong K(r_2, t)$ if $r_1 \neq r_2$ for all $s, k, t$, and $K(r, s, k)$ is isomorphic to either of $K(r, s, 0), K(r, s, 1)$. So, every map of type $\{3, 6, 3, 6\}$ on $n$ vertices is isomorphic to one of $K(r, s, 0), K(r, s, 1)$ and $K(r, s)$ for some $r \mid n$ and $s \in \{\frac{2n}{3r}, \frac{2n}{3r+1}, \frac{4n}{4r+5}\}$.

By Corollary \ref{cor3122:1} and Lemma \ref{lem3122:4}, $K(r_1, s, k) \not\cong K(r_2, s, k)$ and $K(r_1, t) \not\cong K(r_2, t)$ if $r_1 \neq r_2$ for all $s, k, t$, and $K(r, s, k)$ is isomorphic to either of $K(r, s, 0), K(r, s, 1)$. So, every map of type $\{3, 12^{2}\}$ on $n$ vertices is isomorphic to one of $K(r, s, 0), K(r, s, 1)$ and $K(r, s)$ for some $r \mid n$ and $s \in \{\frac{2n}{3r}, \frac{2n}{3r+1}, \frac{4n}{4r+5}\}$.

By Corollary \ref{cor4612:1} and \ref{cor4612:2}, $K(r_1, s, k) \not\cong K(r_2, s, k)$ and $K(r_1, t) \not\cong K(r_2, t)$ if $r_1 \neq r_2$ for all $s, k, t$, and $K(r, s, k)$ is isomorphic to either of $K(r, s, 0), K(r, s, 1)$. So, every map of type $\{4, 6, 12\}$ on $n$ vertices is isomorphic to one of $K(r, s, 0), K(r, s, 1), K(r, s)$ for some $r \mid n$ and $s = \frac{n}{r}$.

By Corollary \ref{cor3464:1} and Lemma \ref{lem3464:4}, $K(r_1, s, k) \not\cong K(r_2, s, k)$ and $K(r_1, t) \not\cong K(r_2, t)$ if $r_1 \neq r_2$ for all $s, k, t$, and $K(r, s, k)$ is isomorphic to either of $K(r, s, 0), K(r, s, 1)$. So, every map of type $\{3, 4, 6, 4\}$ on $n$ vertices is isomorphic to one of $K(r, s, 0), K(r, s, 1), K(r, s)$ for some $r \mid n$ and $s = \frac{n}{r}$.
\hfill$\Box$

\section{Acknowledgement}

The first author is supported by post doctoral scholarship of Harish Chandra Research Institute, Department of Atomic Energy, Government of India. Work of second author is partially supported by SERB, DST grant No. SR/S4/MS:717/10.

\end{document}